\documentclass[11pt,reqno]{amsart}
\usepackage[utf8]{inputenc}
\usepackage[T1]{fontenc}
\usepackage{lmodern}
\usepackage[a4paper]{geometry}
\usepackage{amsmath}
\usepackage{amssymb}
\usepackage{amsfonts}
\usepackage{mathrsfs}
\usepackage{amsthm}
\usepackage{commath}
\usepackage{mathtools}
\usepackage{hyperref}
\usepackage{doi}
\usepackage{bigints}
\usepackage{xcolor}
\usepackage{dsfont}
\usepackage[nolabel]{showlabels}
\usepackage{graphicx}
\usepackage{enumitem}
\usepackage{import}
\usepackage{chngcntr}
\usepackage{enumitem}
\usepackage{comment}

\def\dA{{\mathcal{A}}}
\def\dB{{\mathcal{B}}}
\def\dC{{\mathcal{C}}}
\def\dF{{\mathcal{F}}}
\def\dG{{\mathcal{G}}}
\def\dN{{\mathcal{N}}}
\def\dS{{\mathcal{S}}}

\def\bC{{\mathbb{C}}}

\newcommand{\LD}{\mathsf{LD}}
\newcommand{\Top}{\mathsf{Top}}
\newcommand{\Tr}{\mathrm{Tree}}
\newcommand{\Stop}{\mathsf{Stop}}
\newcommand{\End}{\mathsf{End}}

\def\cC{{\mathscr{C}}}

\def\cF{{\mathscr{F}}}

\def\one{\mathds{1}}

\newcommand{\Tree}{\mathsf{Tree}}

\def\ve{\epsilon} 

\renewcommand{\d}{{\partial}}
\def\lec{\lesssim}

\def\spn{\mathop\mathrm{span}} 	
\DeclareMathOperator{\diam}{diam}
\def\dim{\mathop\mathrm{dim}} 					
\def\dist{\mathop\mathrm{dist}} 						
\def\id{\mathop\mathrm{id}}						

\newcommand{\ps}[1]{\left( #1 \right)}

\newcommand{\ck}[1]{\left\{#1 \right\}}

\newcommand{\fr}[2]{\frac{#1}{#2}}

\def\XXint#1#2#3{{\setbox0=\hbox{$#1{#2#3}{\int}$ }
		\vcenter{\hbox{$#2#3$ }}\kern-.58\wd0}}

\newtheorem{theorem}{Theorem}[section]
\newtheorem{lemma}[theorem]{Lemma}

\theoremstyle{definition}
\newtheorem{definition}[theorem]{Definition}

\theoremstyle{remark}
\newtheorem{remark}[theorem]{Remark}

\numberwithin{equation}{section}

\newcommand{\R}{\mathbb{R}}
\newcommand{\N}{\mathbb{N}}
\newcommand{\nbd}{\mathcal{N}}
\newcommand{\Z}{\mathbb{Z}}

\newcommand{\hd}{\mathcal{H}^d}

\newcommand{\hdc}{\mathcal{H}^d_\infty}

\newcommand{\B}{\mathbb{B}}

\newcommand{\wt}{\widetilde}

\newcommand{\dD}{\mathcal{D}}
\newcommand{\DD}{\mathcal{D}}

\newcommand{\dH}{\mathcal{H}}

\newcommand{\Reg}{\mathsf{Reg}_*}

\DeclareMathOperator{\supp}{supp}

\DeclareMathOperator{\lip}{Lip}
\DeclareMathOperator{\Fav}{Fav}

\newcommand{\Next}{\mathsf{Next}}

\newcommand{\Child}{\mathsf{Ch}}
\newcommand{\Nei}{\mathsf{Nbd}}

\newcommand{\dL}{\mathcal{L}}
\theoremstyle{theorem}

\newtheorem{coro}[theorem]{Corollary}
\newtheorem{propo}[theorem]{Proposition}

\newtheorem*{claim*}{Claim}

\newtheorem*{theorem*}{Theorem}
\newtheorem*{lemma*}{Lemma}

\theoremstyle{definition}

\theoremstyle{remark}

\numberwithin{equation}{section}
\begin{document}

	\title[Analytic capacity and dimension of sets with PBP]{Analytic capacity and dimension\\ of sets with plenty of big projections}	
	
	\author[D. D\k{a}browski]{Damian D\k{a}browski}
	\address{Damian D\k{a}browski, University of Jyväskylä, P.O. Box 35 (MaD), 40014, Finland}	
	\email{damian.m.dabrowski@jyu.fi}
	
	\author[M. Villa]{Michele Villa}
	\address{Michele Villa, Research Unit of Mathematical Sciences, University of Oulu, P.O. Box 8000, FI-90014, University of Oulu, Finland}
	\email{michele.villa@oulu.fi}
	
	
	\thanks{D.D. and M.V. were supported by the Academy of Finland via the project \textit{Incidences on Fractals}, grant No. 321896. M.V. was also supported by a starting grant of the University of Oulu}
	

		\begin{abstract}
		Our main result marks progress on an old conjecture of Vitushkin. We show that a compact set in the plane with plenty of big projections (PBP) has positive analytic capacity, along with a quantitative lower bound. A higher dimensional counterpart is also proved for capacities related to the Riesz kernel, including the Lipschitz harmonic capacity. The proof uses a construction of a doubling Frostman measure on a lower content regular set, which may be of independent interest.
		
		Our second main result is the Analyst's Traveling Salesman Theorem for sets with plenty of big projections. As a corollary, we obtain a lower bound for the Hausdorff dimension of uniformly wiggly sets with PBP. The second corollary is an estimate for the capacities of subsets of sets with PBP, in the spirit of the quantitative solution to Denjoy's conjecture. 

	\end{abstract}

	\maketitle
	
	\tableofcontents
	
	\section{Introduction}
	The aim of this paper is to study sets with plenty of big projections.
	\begin{definition}
		We say that $E\subset\R^n$ has \emph{plenty of ($d$-dimensional) big projections} (abbreviated to PBP, or $d$-PBP) if there exists a constant $0<\delta\le 1$ such that the following holds. For all $x\in E$ and $0<r<\diam(E)$ there exists $V_{x,r}\in \dG(n,d)$ such that
		\begin{equation*}
		\hd(\pi_V(E\cap B(x,r)))\ge \delta r^d\quad\text{for all }V\in B(V_{x,r},\delta).
		\end{equation*}
	Here $\dG(n,d)$ is the Grassmannian manifold of $d$-dimensional (linear) planes in $\R^n$, $\pi_V:\R^n \to V$ is the orthogonal projection, and the ball $B(V_{x,r}, \delta)$ is defined with respect to the standard metric on $\dG(n,d)$ (see \secref{sec:notation}). 
	\end{definition}
	This definition first appeared in \cite{david1993quantitative}, although originally David and Semmes assumed additionally that sets with PBP are Ahlfors regular.
	\begin{definition}
		A set $E \subset \R^n$ is called \emph{Ahlfors $d$-regular} if there exists a constant $C \geq 1$ such that for all $x\in E$ and $0<r<\diam(E)$ we have
		\begin{equation*}
		C^{-1} r^d \leq \dH^d(E\cap B(x,r)) \leq C r^d.
		\end{equation*}
		The smallest constant $C$ for which the above holds is called the \emph{Ahlfors regularity constant} of $E$.
	\end{definition}
	In \cite{david1993quantitative} David and Semmes conjectured that Ahlfors regular sets with PBP satisfy the so-called \emph{big pieces of Lipschitz graphs condition}, a strong quantitative rectifiability property (see Definition \ref{def:BPLG}). Their conjecture was recently solved in a breakthrough work of Orponen \cite{orponen2020plenty} (see \thmref{thm:orponen} for the precise statement).
 
\vspace{0.5em}
The novelty of our article is twofold: on the level of results, we use Orponen's theorem to study the analytic capacity of sets with PBP, thus making progress on Vitushkin's and quantitative Denjoy's conjectures, and the dimension of wiggly sets. Beyond these results, our proof presents methodological interest, and this is the other level of novelty. Indeed, for the problems at hand we are required to work outside the Ahlfors regularity category: we thus propose a method that allows to move from this rather strong size condition to sets which might even have non-$\sigma$-finite $\dH^d$-measure. This method is highly flexible, and relies on three distinct steps: starting with a (compact) set $E$ with $d$-PBP (for which the $\dH^d$-measure is not necessarily $\sigma$-finite), we
\begin{itemize}\label{program}
    \item construct a Frostman measure supported on $E$, which is also doubling,
    \item give a multiscale approximation of this measure by Ahlfors regular sets,
    \item show that the PBP property is inherited by the approximating sets: this, by the result of Orponen, allow us to infer some key geometric facts about $E$.
\end{itemize}
In the finite measure case (but not necessarily Ahlfors regular), this method was layed out rather explicitly in \cite{azzam2019quantitative}. The first step above allows us to move this method to the general case, which is highly relevant for the Vitushkin's conjecture mentioned above, and, more generally, in contexts where the natural size to consider is capacity, rather than Hausdorff measure. We suspect that this method can be applied to other quantitative properties developed in the David-Semmes theory \cite{david1991singular,david1993analysis}, thus giving bounds for analytic and Lipschitz harmonic capacities for sets satisfying any one of these properties.

Now we move on to the precise statements of our results.
	
	\subsection{Analytic capacity: Vitushkin's conjecture and quantitative Denjoy's conjecture}
	Recall that a compact set $E\subset\bC\simeq \R^2$ is said to be \emph{removable for bounded analytic functions} if all bounded analytic functions $f$ defined on $\bC\setminus E$ are constant. In the 40s Ahlfors \cite{ahlfors1947bounded} managed to quantify the notion of removability by introducing \emph{analytic capacity}. Recall that the analytic capacity of a compact set $E\subset\bC$ is defined as
	\begin{equation*}
		\gamma(E) = \sup |f'(\infty)|,
	\end{equation*}
	where the supremum is taken over all analytic functions $f:\bC\setminus E\to\bC$ with $\|f\|_\infty\le 1$, and $f'(\infty) = \lim_{z\to\infty} z(f(z)-f(\infty))$. Ahlfors proved that a set is removable for bounded analytic functions if and only if $\gamma(E)=0$.
	
	In 1967 Vitushkin \cite{vitushkin1967analytic} conjectured a geometric characterization of removability in terms of orthogonal projections. He asked the following: is it true that for compact sets $E\subset\bC$ one has
	\begin{equation}\label{eq:Vitushkinsconj}
		\gamma(E)=0\quad\Leftrightarrow\quad \Fav(E)=0,
	\end{equation}
	where $\Fav(E)$ is the Favard length of $E$, defined as
	\begin{equation*}
	\Fav(E)=\int_0^\pi\mathcal{H}^1(\pi_\theta(E))\ d\theta.
	\end{equation*}
	In the above $\pi_\theta:\bC\to \ell_\theta$ is the orthogonal projection to $\ell_\theta = \spn((\sin\theta,\cos\theta))$.
	
	It has been known for a very long time that sets with $\dH^1(E)=0$ are removable, while sets with $\dim_H(E)>1$ are non-removable, so Vitushkin's conjecture holds for such sets. The case of sets $E$ satisfying $0<\mathcal{H}^1(E)<\infty$ was much more difficult to establish, but the answer to Vitushkin's question is also positive: the implication $(\Rightarrow)$ in \eqref{eq:Vitushkinsconj} is due to Calderón \cite{calderon1977cauchy}, while $(\Leftarrow)$ was shown by David \cite{david1998unrectifiable}. Finally, the case of sets with $\sigma$-finite $\dH^1$-measure follows from the finite measure case together with subadditivity of analytic capacity, which is due to Tolsa \cite{tolsa2003painleve}.
	
	In 1986 Mattila \cite{mattila1986smooth} showed that Vitushkin's conjecture fails for sets with Hausdorff dimension $1$ and non-$\sigma$-finite $\dH^1$-measure. It wasn't clear from Mattila's proof which of the implications in \eqref{eq:Vitushkinsconj} is false. Soon thereafter Jones and Murai \cite{jones1988positive} constructed an example of a set with $\gamma(E)>0$ and $\Fav(E)=0$. A simpler example was found later on by Joyce and M\"{o}rters \cite{joyce2000set}. Hence, the implication $(\Leftarrow)$ in \eqref{eq:Vitushkinsconj} is false. For more information on analytic capacity and Vitushkin's conjecture see the books \cite{tolsa2014analytic,dudziak2011vitushkin,pajot2002analytic}. See also recent surveys on related topics \cite{verdera2021birth,mattila2021rectifiability}.
	
	After all these extraordinary developments one question remains: what about the implication $(\Rightarrow)$ in \eqref{eq:Vitushkinsconj}? Equivalently, is it true that 
	\begin{equation}\label{eq:Vitushkin2}
		\Fav(E)>0\quad\Rightarrow\quad\gamma(E)>0?
	\end{equation}
	A quantitative version of this question is the following: is it true that
	\begin{equation}\label{Vituskin3}
		\gamma(E)\gtrsim \Fav(E)?
	\end{equation}
	
	As far as we know, the only partial result related to these open problems is due to Chang and Tolsa \cite{chang2017analytic} (more on that in \remref{rem:CT}). In this article we make further progress on \eqref{eq:Vitushkin2} and \eqref{Vituskin3}. If the assumption $\Fav(E)>0$ is replaced by the (significantly stronger) PBP assumption, then the analytic capacity is positive.
	\begin{theorem}\label{thm:main analytic cap}
		Let $E\subset\R^2$ be a compact set with $1$-PBP. Then,
		\begin{equation}\label{eq:thm analytic cap}
		\gamma(E)\gtrsim \diam(E),
		\end{equation}
		where the implicit constant depends only on the PBP constant.
	\end{theorem}
	\begin{remark}
		In the theorem above we do not assume $\dH^1(E)<\infty$. However, our result bears new information even in the $\dH^1(E)<\infty$ case. Recall that in this regime Vitushkin's conjecture is true, so in particular $\Fav(E)>0$ implies $\gamma(E)>0$. However, the existing proof does not offer any quantitative lower bound on $\gamma(E)$ in terms of $\Fav(E)$. Roughly speaking, the reason is the following: the Besicovitch projection theorem states that if $\Fav(E)>0$, then there exists a $1$-dimensional Lipschitz graph $\Gamma$ such that $\dH^1(E\cap\Gamma)>0$, and then $\gamma(E)\ge \gamma(E\cap\Gamma)>0$ follows from \cite{calderon1977cauchy}. However, the Besicovitch projection theorem does not say anything about the size of $\dH^1(E\cap\Gamma)$, and so deriving a quantitative lower bound for $\gamma(E)$ is impossible as long as one follows this strategy. In other words, the question \eqref{Vituskin3} remains wide open even in the $\dH^1(E)<\infty$ case. Although \thmref{thm:main analytic cap} is quite far from establishing \eqref{Vituskin3}, it is (as far as we know) the first quantitative lower bound for $\gamma(E)$ which depends only on the projections of $E$. 
	\end{remark}
		
		In the finite measure case, we are able to prove more: we show a quantitative bound for the analytic capacity of \textit{subsets} of sets with $1$-PBP. This is a form of quantitative Denjoy's conjecture, where rectifiable curves are substituted with sets with PBP of finite measure. Recall that Denjoy's conjecture, which predates Vitushkin's conjecture by about 60 years, stated that if $\Gamma\subset\R^2$ is a rectifiable curve, and a compact set $E\subset\Gamma$ satisfies $\mathcal{H}^1(E)>0$, then $\gamma(E)>0$. This statement can be seen as a special case of Vitushkin's conjecture, if one restricts attention to the implication \eqref{eq:Vitushkin2} and sets of finite length. Denjoy's conjecture was confirmed by Calderón in \cite{calderon1977cauchy}. A quantitative version of Calderón's result was established by Murai \cite{murai1987comparison}, who showed that if $\Gamma$ is a $1$-rectifiable graph, and $E\subset\Gamma$ is compact, then 
		\begin{equation*}
			\gamma(E)\gtrsim \frac{\mathcal{H}^1_\infty(E)^{\frac{3}{2}}}{\mathcal{H}^1(\Gamma)^{\frac{1}{2}}}.
		\end{equation*}
		Later on Verdera observed that Murai's estimate can be obtained for arbitrary rectifiable curves $\Gamma$ using Menger curvature and Jones' Travelling Salesman Theorem, see \cite[Theorem 4.31]{tolsa2014analytic}. The exponent $3/2$ above is optimal, see \cite{murai1990power} and \cite[\S 4.8]{tolsa2014analytic}.
		
		Using an Analyst's Travelling Salesman Theorem for sets with PBP (see Theorem \ref{t:main} below) we obtain the following.
	
\begin{theorem}\label{t:corollary-2}
	Let $\Sigma \subset \R^2$ be a compact set with $1$-PBP and $0<\mathcal{H}^1(\Sigma)<\infty$. Then, for any compact subset $E \subset \Sigma$
	\begin{align*}
		\gamma(E) \gtrsim \frac{\mathcal{H}^1_\infty(E)^{\frac{3}{2}}}{\mathcal{H}^1(\Sigma)^{\frac{1}{2}}},
	\end{align*}
	where the implicit constant depends only on the PBP constant.
\end{theorem}
\noindent

Theorem \ref{t:corollary-2} is interesting for two reasons: first, as mentioned above, it applies to subsets of sets with PBP. These subsets could very well have vanishing density in many places, and have very small projections. Note that a similar statement cannot be true for sets $\Gamma$ with infinite length: for example, the unit square $[0,1]^2$ has PBP, but the 4-corner Cantor set $K\subset[0,1]^2$ satisfies $\mathcal{H}^1(K)\sim 1$ and $\gamma(K)=0$.

The second interesting aspect of Theorem \ref{t:corollary-2} is that it suggests that, for many geometric problems, sets with PBP and finite measure are just as good as rectifiable curves. 
Let us however make the following remark: from the Analyst's TST for sets with PBP (Theorem \ref{t:main} below) it follows that a set with PBP and finite $\dH^1$-measure can be covered by a rectifiable curve of comparable length. Thus, Theorem \ref{t:corollary-2} can be derived directly from Theorem \ref{t:main} together with the previous result of Murai and Verdera. However, in higher dimension (see the section below) the result is altogether new, since no result analogous to that of Murai was known for capacities associated to vector-valued Riesz kernel.

\subsection{Higher dimensional variants}

	Higher dimensional variants of \thmref{thm:main analytic cap} and Theorem \ref{t:corollary-2} are also true. Instead of analytic capacity one has to consider certain capacities associated to the vector-valued Riesz kernel. Let $0<d<n$ be integers. Given a compact set $E\subset\R^n$ the capacity $\Gamma_{n,d}(E)$ is defined as
	\begin{equation*}
		\Gamma_{n,d}(E) = \sup |\langle T, 1\rangle |,
	\end{equation*}
	where the supremum is taken over all real distributions $T$ supported in $E$ such that 
	\begin{equation*}
		\fr{x}{|x|^{d+1}}\ast T\in L^{\infty}(\R^n)\quad\text{and}\quad \bigg\| \fr{x}{|x|^{d+1}}\ast T\bigg\|_{L^{\infty}(\R^n)}\le 1.
	\end{equation*}
	
	We remark that by \cite{tolsa2003painleve} one has $\Gamma_{2,1}(E)\sim\gamma(E)$. In the codimension 1 case the capacity $\Gamma_{d+1,d}(E)$ is also called the \emph{Lipschitz harmonic capacity}, and it is often denoted by $\kappa(E)$. It was introduced by Paramonov, who observed that sets with $\kappa(E)=0$ are precisely the sets removable for Lipschitz harmonic functions, see \cite[Remark 2.4]{paramonov1990harmonic} and \cite{mattila1995geometric}. The counterpart of Vitushkin's conjecture for $\kappa(E)$ and sets with $\dH^d(E)<\infty$ was established by Nazarov, Tolsa and Volberg in \cite{nazarov2014,nazarov2014on}. More specifically, they showed that for a compact set $E\subset\R^{d+1}$ with $\dH^d(E)<\infty$ one has $\kappa(E)=0$ if and only if $E$ is purely $d$-unrectifiable (equivalently, $\dH^d(\pi_V(E))=0$ for a.e. $V\in\mathcal{G}(n,d)$). Whether an analogous statement is true for $\Gamma_{n,d}$ with $1<n<d-1$ is an open problem.

	The higher dimensional variants of \thmref{thm:main analytic cap} and Theorem \ref{t:corollary-2} are the following.
	\begin{theorem}\label{thm:main higher dim}
		Let $E\subset\R^n$ be a compact set with $d$-PBP. Then,
		\begin{equation}\label{eq:main higher dim}
		\Gamma_{n,d}(E)\gtrsim \diam(E)^d,
		\end{equation}
		where the implicit constant depends only on $n,d,$ and the PBP constant.
	\end{theorem}

\begin{theorem}\label{t:corollary-2b}
	Let $\Sigma\subset \R^n$ be a compact set with $d$-PBP and $0<\hd(\Sigma)<\infty$. Then for any compact subset $E \subset \Sigma$, 
	\begin{align*}
	\Gamma_{n,d}(E) \gtrsim \frac{\hdc(E)^{\frac{3}{2}}}{\hd(\Sigma)^{\frac{1}{2}}},
	\end{align*}
	where the implicit constant depends only on $n,d$, and the PBP constant.
\end{theorem}

	\begin{remark}\label{rem:CT}
		Theorems \ref{thm:main analytic cap} and \ref{thm:main higher dim} should be compared with the results of Chang and Tolsa \cite{chang2017analytic}, who proved the following. If $I\subset [0,\pi)$ is an interval and $E\subset\mathbb{C}$ is a compact set supporting a probability measure $\mu$ such that $\pi_\theta\mu\in L^2(\ell_\theta)$ for a.e. $\theta\in I$, then 
		\begin{equation*}
		\gamma(E)\gtrsim \fr{1}{\int_I \|\pi_\theta\mu\|^2_{L^2}\, d\theta},
		\end{equation*}
		with the implicit constant depending on $\dH^1(I)$. They also prove a higher dimensional analogue of this estimate for capacities $\Gamma_{n,d}$.
		
		The main advantage of \cite{chang2017analytic} over our results is that their ``$L^2$-projections'' assumption is single-scale, while the PBP condition is multi-scale. On the other hand, our result may be easier to apply as it derives a lower bound for $\gamma(E)$ directly from the information on the projections of $E$, without the need of constructing a measure supported on $E$ and studying its projections. 
		
		Even more importantly, note that the $L^2$-projections assumption $\pi_\theta\mu\in L^2(\ell_\theta)$ implies a big projection $\dH^1(\pi_\theta(E))\ge\|\pi_\theta\mu\|^{-2}_{2}$ using the Cauchy-Schwarz inequality
		\begin{equation*}
		1 = \mu(E) = \|\pi_\theta\mu\|_{1} \le \|\pi_\theta\mu\|_{2}\,\dH^1(\pi_\theta(E))^{1/2}.
		\end{equation*}		
		The difference between $L^2$-projections and big projections is fundamental. A characterization of Ahlfors regular sets with big pieces of Lipschitz graphs in terms of $L^2$-projections has been achieved by Martikainen and Orponen in \cite{martikainen2018characterising}, but it took another major breakthrough \cite{orponen2020plenty} to find an analogous characterization in terms of PBP. Moreover, a ``single-scale version'' of \cite{orponen2020plenty} is an open problem, see \cite[\S1.3]{orponen2020plenty}, whereas \cite{martikainen2018characterising} contains the relevant single-scale result, see \cite[Theorem 1.7]{martikainen2018characterising}.
	\end{remark} 
	
	\begin{remark}
		The proof of \thmref{thm:main higher dim} is robust, and our techniques are likely to be useful in the future work on problems \eqref{eq:Vitushkin2} and \eqref{Vituskin3}. 
		
		In particular, suppose that the following strengthening of Orponen's result (\thmref{thm:orponen}) was available: ``if a set $E\subset\R^n$ is Ahlfors $d$-regular and it has \emph{uniformly large Favard length} (there exists $C>0$ such that for all $x\in E$ and all $0<r<\diam(E)$ we have $\Fav(E\cap B(x,r))\ge Cr^d$), then $E$ has big pieces of Lipschitz graphs.'' With such result at hand, the proof of \thmref{thm:main higher dim} would immediately yield that any set $E\subset\R^d$ with uniformly large Favard length satisfies
		\begin{equation*}
		\Gamma_{n,d}(E)\gtrsim \diam(E)^d.
		\end{equation*}
	\end{remark}
	A related question is of course the following: suppose that $E \subset [0,1]^2$ is a $1$-Ahlfors regular set with ${\rm Fav}(E) \geq \delta$. Is it then true that there exists a Lipschitz graph $\Gamma$ so that $\mathcal{H}^1(E \cap \Gamma) \gtrsim \delta$? A recent related result by A. Chang, T. Orponen and the authors \cite{chang2022maxfav} shows that if $E$ has almost maximal Favard length, then it is contained in a Lipschitz graph with small constant, save for a tiny subset. A natural question is whether this result can be used to obtain a new estimate for analytic capacity.
	\subsection{Analyst's Traveling Salesman Theorem}
	Another main result of this article is the {Analyst's Travelling Salesman Theorem (TST) for sets with PBP}. Recall that the original Analyst's TST is due to Jones \cite{jones1990rectifiable}, and it is a characterization of sets $E\subset\R^2$ such that there exists a rectifiable curve $\Gamma$ containing $E$, along with a quite precise estimate on $\dH^1(\Gamma)$. Much work has been put into proving analogs of this result in other spaces; for example, Okikiolu \cite{okikiolu1992characterization} proved an Analyst's TST for curves in $\R^d$, and in \cite{schul2007subsets} Schul further generalized it to the Hilbert space setting. Another research direction consists in finding statements analogous to Jones' TST but for higher dimensional sets rather than curves. %
	While we are still lacking a theorem as complete as that of Jones, in recent years there has been much progress in this direction. Two major difficulties were: first, the coefficients used by Jones simply do not work for higher dimensional set; second, it is not obvious what to use as an analogue to curves (e.g. topological spheres do not work, see \cite[Introduction]{villa2019higher}). These issues have largely been overcome in the series of papers \cite{azzam2018analyst,azzam2019quantitative,villa2019higher,hyde2020ddimensional,hyde2021d}.  We refer the reader to Sections 1 and 3 of \cite{villa2019higher} for a more in-depth discussion of Analyst's TST and its relevance. Let us mention that while these results are related to the rectifiability of sets, they do not say much about whether these sets can be parameterised. This is a major open problem in the area.
	
	Our present result continues the line of work of Azzam, Schul, and the second author \cite{azzam2018analyst,azzam2019quantitative}. Let $\dD$ be a system of Christ-David cubes (see \lemref{theorem:christ}), and set 
	\begin{align}\label{e:BETA}
	\beta(Q_0) = \beta_{E, C_0, p, d} (Q_0): = \sum_{Q \in \dD(Q_0)} \beta_E^{d,p} (C_0 B_Q)^2 \ell(Q)^d,
	\end{align}
	where the coefficients $\beta_E^{d,p}$ are a variant of Jones' $\beta$-numbers \eqref{e:beta-jones} introduced by Azzam and Schul in \cite{azzam2018analyst}, see Definition \ref{d:content-beta}.
	
	\begin{theorem}\label{t:main}
		Let $E\subset \R^n$ be a set with $d$-PBP with constant $\delta>0$, $\dD$ be a system of Christ-David cubes, $Q_0 \in \dD$ and $C_0\geq 3$. Let $1 \leq p < p(d)$, where $p(d)=\tfrac{2d}{d-2}$ for $d>2$ and $p(d)=\infty$ if $d\le 2$.
		Then
		\begin{align}\label{e:main}
			\diam(Q_0)^d + \beta(Q_0) \sim \hd(Q_0),
		\end{align}
		where the implicit constant depends on $\delta$, $C_0$, $p$, $n,d$ and on the constants from the Azzam-Schul TST \textup{(}see Theorem A.1 in \cite{azzam2019quantitative}\textup{)}.
	\end{theorem}	
	
	\begin{remark}
		What we really prove here is one direction of the inequality (bound on the $\beta$ sum with the measure), as the other one follows from \cite{azzam2018analyst}. See Proposition \ref{l:main-lemma}.
	\end{remark}
	Estimates similar to \eqref{e:main} have been proven in \cite{azzam2018analyst} and \cite{azzam2019quantitative} for general \emph{lower content regular sets} (see \secref{subsec:LCR} for the definition; sets with PBP are lower content regular), with the following crucial caveat: in general, an additional error term needs to be added to the right hand side of \eqref{e:main}. In \cite{villa2019higher} the second author proved that the error term may be omitted if one assumes that $E$ is a \emph{topologically stable $d$-surface} (a condition satisfied e.g. by Reifenberg flat sets, or Semmes surfaces), see \cite[Theorem 3.6]{villa2019higher}. In \thmref{t:main} we prove that the error term disappears also in the case of sets with PBP.
	In other words, both PBP and topological stability give the set enough rigidity so that one can estimate its $\beta$-numbers using the Hausdorff measure. 
	
	To see why the estimate \eqref{e:main} may be useful, let us mention that in \cite[Theorem II]{azzam2018analyst} it was shown that a lower content regular set with $\beta(E)<\infty$ is rectifiable. Thus, any class of lower content regular sets for which \eqref{e:main} holds satisfies
	\begin{equation}\label{form10000}
		\mathcal{H}^d(E)<\infty \, \implies\,  \mbox{ rectifiability.}
	\end{equation}
	In particular, \eqref{form10000} holds for topologically stable surfaces, and for sets with PBP (of course, for sets with PBP \eqref{form10000} easily follows from the Besicovitch projection theorem, without the need to refer to $\beta$-numbers and TST). Identifying classes of sets for which \eqref{form10000} holds is an interesting problem. For example, while \eqref{form10000} is true for connected one dimensional sets, there exist $2$-dimensional sets, homeomorphic to the $2$-sphere, with finite $\mathcal{H}^2$-measure, and containing a purely 2-unrectifiable set of positive measure (see \cite[Figure 1]{villa2019higher}) See also \cite{david2019rectifiability} and \cite{david2020note} in connection with \eqref{form10000} and quantitative topological conditions.
	
	Theorem \ref{t:main} is used in the proofs of Theorems \ref{t:corollary-2} and \ref{t:corollary-2b}. Another application is an estimate for the dimension of \textit{wiggly sets}.
	
	\subsection{Dimension of wiggly sets}
	
		A  set $E$ is said to be \textit{uniformly wiggly} of dimension $d$ and with parameter $\beta_0$ if for all balls $B$ centered on $E$ and with $0<r(B)<\diam(E)$ it holds that
		\begin{align}\label{e:beta-jones}
	\beta_{E, \infty}^d(B) > \beta_0.
		\end{align}

		Here $\beta_{E, \infty}^d$ is the $d$-dimensional version of the so-called $L^\infty$ $\beta$-number of Peter Jones \cite{jones1990rectifiable} (see Definition \ref{d:jones-beta}). Wiggly sets appear naturally in various contexts where some form of self-similarity is present. Examples include limit sets of certain Kleinian groups, some Julia sets of polynomials, and random sets (see \cite{bishop2017fractals}, pp. 340-341).  They were first studied via the Analyst's TST by Bishop and Jones in \cite{bishop1997wiggly}, were they proved a lower bound on the Hausdorff dimension of wiggly connected sets in the plane, assuming $d=1$. Their result was generalised to continua in metric spaces by Azzam \cite{azzam2015hausdorff}. A result in this vein was later proved by David \cite{david2004hausdorff} (and quantified in \cite{villa2019higher}), this time for uniformly non flat sets of \textit{any integer} dimension, satisfying a topological condition. David's work was motivated by a question of L. Potyagailo, concerning higher dimensional limit sets (see the introduction of \cite{david2004hausdorff}). The result below can be seen as a version of David's theorem with the topological condition replaced by the PBP assumption.

	\begin{theorem}\label{t:corollary}
		Let $n \geq 2$ and $1 \leq d \leq n-1$. Let $E \subset \R^n$ be a closed set with $d$-PBP \textup{(}with parameter $\delta>0$\textup{)} and which is uniformly wiggly of dimension $d$ and constant $\beta_0$. Then 
		\begin{equation}\label{e:wiggly-dim}
		{\rm dim}_H(E) \geq d + c\beta_0^{2(d+1)},
		\end{equation}
		where $c$ depends on $\delta, n,d,p$.
	\end{theorem}
	
	\begin{remark}
		We defined uniformly wiggly sets in terms of $\beta_{E,\infty}^d$, but, in principle, we could have used the $L^p$ versions of the coefficients instead. Indeed, if we take a set $E$ that satisfies the hypotheses of Theorem \ref{t:corollary} with respect to the Azzam-Schul $\beta_E^{d,p}$-numbers, assuming $1 \leq p < p(d)$ (where $p(d)$ is defined in Theorem \ref{t:main}), then we obtain a somewhat sharper dimension estimate
		\begin{equation}
			{\rm dim}_H (E) \geq d+  c \beta_0^{2},
		\end{equation}
	see Proposition \ref{prop-content-beta-wiggly}. This is not unexpected, since the correct $\beta$-coefficients to use when working with higher dimensional sets are the averaged ones.

	Another comment on definitions is in order: many interesting sets satisfy non-flatness hypotheses weaker than uniform wiggly. For example, attractors of many dynamical systems are only \textit{mean} wiggly, in the sense that the $\beta$ coefficients are large in many (but not all) scales and locations. See \cite{graczyk2012metric, graczyk2022optimal} for a definition. See also \cite{koskela1997hausdorff} for a result similar in spirit concerning mean porous sets. 
	\end{remark}
	
	\begin{remark}
	 The dependence of $\dim_H(E)$ on the PBP parameters is not a proof artifact: suppose we had a lower bound for Hausdorff diemnsion depending only on $\beta_0$ and \textit{not} on $\delta$. Consider the four corners Cantor set $E$, constructed in the usual way, except that we dilate by constant $\eta>1$ the squares in the construction. In the limit, we will obtain that a) $\dim_{H}(E) = 1 + C(\eta)$ with $C(\eta)\to 0$ as $\eta\to 1$, b) $\beta_E^{1,\infty}(B)\gtrsim 1 \geq \beta_0>0$ for any ball $B$ centered on $E$, and c) $E$ will have PBP with parameters depending on $\eta$. But then, our hypothetical (and false) theorem would tell us that $\dim_{H}(E) \geq 1+ c \beta_0^4$ independently of $\eta$. 
	  This example was pointed out by T. Orponen in a discussion on a first draft, where the exact dependence of the dimension estimate on the PBP constant had been overlooked. 
	  
	  In fact, after a minute's thinking, the dependence on $\delta$ of \eqref{e:wiggly-dim} appears altogether natural: the parameter $\delta$ regulates the ``rigidity'' of a set with PBP, where one should think of connected set as having `rigidity' 1.
	\end{remark}

	\subsection{Outline of the paper and proof ideas} In Section \ref{sec:prels} we establish notation and state some quantitative rectifiability results used in the paper. Here the reader will find Lemma \ref{lem:Frostman measure} where a doubling Frostman measure is constructed, and which represents the first step of the program outlined on p. \pageref{program}. While the idea of the proof is rather simple (a redistribution policy), its execution is lengthy, and thus we have deferred it to Appendix \ref{appendix}.

  In Section \ref{s:coronization} we carry out the second step of our program. In fact, a coronization of lower content regular sets with finite Hausdorff measure in terms of Ahlfors regular sets was given in \cite{azzam2019quantitative}. Here we instead coronize a general lower content regular set (that is, with possibly non-$\sigma$ finite Hausdorff measure) by coronizing the Frostmann measure coming from the first step. This means: given a lower content regular set $E$, we apply Lemma \ref{lem:Frostman measure} to obtain a doubling Frostman measure $\mu$ supported on $E$. Next, we run a stopping time algorithm to decompose the dyadic cubes of $\mu$ intro trees, where the density of $\mu$ is roughly constant. The lower regularity of $E$, and the fact that $\mu$ is doubling, allow us to say that the cubes where the algorithm stopped (i.e. those scale and locations where $\mu$ was not behaving like an Ahlfors regular measure) are few, in a precise sense: the stopped cubes satisfy a Carleson packing condition. At the level of each tree we construct an approximating Ahlfors regular measure. This completes the coronization procedure. 

  In Section \ref{s:PBP} we show that the property of having plenty of big projection behaves well with our type of approximation. We are able to show that if the set $E$ has PBP, then the support of each approximating measure also has PBP. By applying the aforementioned result of Orponen, we conclude that each of the approximating measures is uniformly rectifiable.
 
	In Section \ref{sec:analytic} we prove \thmref{thm:main higher dim}. Since $\Gamma_{2,1}(E)\sim\gamma(E)$ by \cite{tolsa2003painleve}, \thmref{thm:main analytic cap} follows from \thmref{thm:main higher dim} by taking $n=2,\, d=1$. 	
	To prove \thmref{thm:main higher dim}, we first use the results from Sections \ref{s:coronization}--\ref{s:PBP} to obtain a family of uniformly rectifiable measures approximating the Frostman measure $\mu$ on our set with PBP. Uniform rectifiability implies that the approximating measures satisfy a flatness condition involving $\beta$-numbers. We transfer the $\beta$-numbers estimates back to $\mu$, and then conclude the proof of Theorem \ref{thm:main higher dim} by using a lower bound on $\Gamma_{n,d}$ due to Prat \cite{prat2012semiadd} and Girela-Sarrión \cite{girela-sarrion2018}, see \thmref{thm:capacitybetas}. 
	
	In Section \ref{sec:TST} we prove the TST for sets with PBP, \thmref{t:main}. The general strategy is similar to that from Section \ref{sec:analytic}. The two main differences are that, firstly, the flatness condition is expressed in terms of a different type of $\beta$-numbers, and secondly, we wish to obtain estimates in terms $\mathcal{H}^d(E)$, and not some capacity. In particular, here we use the coronization from \cite{azzam2019quantitative}, and not the coronization of a Frostman measure.

	The last two sections are dedicated to Theorem \ref{t:corollary} and \thmref{t:corollary-2b} (note that \thmref{t:corollary-2} follows from \thmref{t:corollary-2b} since $\gamma(E)\sim\Gamma_{2,1}(E)$). The strategy of the proof of \thmref{t:corollary} is that of Bishop and Jones \cite{bishop1997wiggly}, while the proof of \thmref{t:corollary-2b} follows Verdera's \cite[Theorem 4.31]{tolsa2014analytic}. A key ingredient of the original proofs is Jones' Analyst's TST theorem \cite{jones1990rectifiable}, which here is replaced by \thmref{t:main}.
	
	In Appendix \ref{appendix} we construct a doubling Frostman measure supported on a lower content regular set. We use this modified Frostman measure in the proof of \thmref{thm:main higher dim}. The idea of the proof is rather simple: the doubling condition is broken if two neighbouring cubes have very unequal amount of mass. We fix this by recursively redistributing ``the wealth'', that is, the mass, form the cubes which are too wealthy, to the poorer ones. We are grateful to Tuomas Orponen for helping us with the construction.
	
	\subsection*{Acknowledgments}
	We thank Alan Chang, Matthew Hyde, Sebastiano Nicolussi Golo, Tuomas Orponen, Polina Perstneva, and Xavier Tolsa for useful comments and discussions on the manuscript and related topics. In particular, T. Orponen suggested us to work on the problem
	\eqref{eq:Vitushkin2}, he pointed out that a previous version of Theorem \ref{t:corollary} was incorrectly stated, suggesting some counterexamples, and finally he helped us with the construction of a doubling Frostman measure from the appendix. We are grateful to A. Chang for helping us with the proof of \lemref{lem:skeletonsprojection}.
	
	Parts of this work were completed during the conference ``Rajchman, Zygmund, and Marcinkiewicz'' at IMPAN (Warsaw), and during the Trimester Programme ``Interactions between Geometric measure theory, Singular integrals, and PDE'' at HIM (Bonn). We thank the organisers of both events for their kind hospitality.
	
	Finally, the second author would like to thank the Mathematical Research unit at the University of Oulu for giving him full support to pursue his research during the first year spent at the university. 

	\section{Preliminaries}\label{sec:prels}
	\subsection{Notation} \label{sec:notation}
	We gather here some notation and some results which will be used later on.
	We write $a \lesssim b$ if there exists an constant $C$ such that $a \leq Cb$. If the constant $C$ depends on a parameter $t$, we write $a\lesssim_t b$. By $a \sim b$ we mean $a \lesssim b \lesssim a$.

	For two subsets $A,B \subset \R^n$, we let
	$
	\dist(A,B) := \inf_{a\in A, b \in B} |a-b|.
	$
	For a point $x \in \R^n$ and a set $A \subset \R^n$, 
	$
	\dist(x, A):= \dist(\{x\}, A)= \inf_{a\in A} |x-a|.
	$
	The cardinality of a set $A$ is denoted by $\#A$.
	
	We write 
	$
	B(x, r) := \{y \in \R^n \, :\,|x-y|<r\},
	$
	and, for $\lambda >0$,
	$
	\lambda B(x,r):= B(x, \lambda r).
	$
	At times, we may write $\B$ to denote $B(0,1)$. When necessary we write $B_n(x,r)$ to distinguish a ball in $\R^n$ from one in $\R^d$, which we may denote by $B_d(x, r)$. Given a ball $B$, we denote by $r(B)$ its radius. 
	
	If $\mu$ is a Radon measure on $\R^n$, then the $d$-dimensional density of $\mu$ in the ball $B=B(x,r)$ is
	\begin{equation*}
	\theta_\mu(B)=\theta_\mu(x,r)=\fr{\mu(B(x,r))}{r^d}.
	\end{equation*}
	
	We denote by $\dG(n,d)$ the Grassmannian, that is, the manifold of all $d$-dimensional linear subspaces of $\R^n$. A ball in $\dG(n,d)$ is defined with respect to the standard metric
	\begin{align*}
	d_{\dG}(V, W) = \|\pi_V - \pi_W\|_{{\rm op}}.
	\end{align*}
	Recall that $\pi_V: \R^n \to V$ is the standard orthogonal projection onto $V$.
	With $\dA(n,d)$ we denote the affine Grassmannian, the manifold of all affine $d$-planes in $\R^n$.

	\subsection{Dyadic lattice and Christ-David cubes}
	The family of dyadic cubes in $\R^n$ will be denoted by $\Delta$, and the family of dyadic cubes with sidelength $\ell(I)=2^{-k}$ by $\Delta_k$. The $d$-dimensional skeleton of $I\in\Delta$ (i.e. the union of the $d$-dimensional faces of $I$) will be denoted by $\d_d I$.
	
	The idea to consider \emph{generalized dyadic cubes}, that is, nested partitions of sets with nice properties, goes back to David \cite{david1988morceaux} and Christ \cite{christ1990tb}. In many contexts it is important that these generalized cubes have thin boundaries with respect to a given measure. Since we will not need this property, we may use for example the cubes from \cite{kaenmaki2012existence}. A special case of their construction gives the following.
	\begin{lemma}[\cite{kaenmaki2012existence}] \label{theorem:christ}
		Let $E\subset\R^n$, $\rho = 1/1000$ and $c_0 = 1/500.$ Then, for each $k \in \Z$, there is a collection $\dD_k$ of generalized cubes on $E$ such that the following hold.
		\begin{enumerate}
			\item For each $k \in \Z, \ E = \bigcup_{Q \in \dD_k}Q,$ and the union is disjoint.
			\item If $Q_1, Q_2 \in \bigcup_k \dD_k$ and $Q_1 \cap Q_2\neq \varnothing$, then either $Q_1 \subset Q_2$ or $Q_2 \subset Q_1$.
			\item For $Q \in \bigcup_k \dD_k$, let $k(Q)$ be the unique integer so that $Q \in \dD_k$ and set $\ell(Q)=5 \rho^k$. Then there is $x_Q \in Q$ such that 
			\begin{align*}
			B(x_Q,c_0\ell(Q))\cap E \subseteq Q \subseteq B(x_Q , \ell(Q)).
			\end{align*}
		\end{enumerate}
	\end{lemma}
	
	We introduce some notation related to cubes. Given two integers $k<l$, we set $\DD_k^l = \bigcup_{i=k}^l \DD_i$. If $R\in\DD$, we will denote the descendants of $R$ by
	\begin{equation*}
	\DD(R)=\{Q\in\DD\ :\ Q\subset R \},
	\end{equation*}
	and $\DD_k(R)=\DD(R)\cap\DD_k$. On the other hand, $Q^1$ will denote the \emph{parent} of $Q$, i.e. the unique cube such that $\ell(Q)=\rho\ell(Q^1)$ and $Q\subset Q^1$. 
	
	The child-parent relation endows $\DD$ with a natural tree structure. We will say that a collection $\mathcal{T}\subset\DD$ is a \emph{tree} if
\begin{itemize}
	\item $\mathcal{T}\subset \DD(R)$ for some $R\in\mathcal{T}$. This maximal cube $R$ will be called the \emph{root} of $\mathcal{T}$.
	\item for any $Q\in\mathcal{T}$ we also have $P\in\mathcal{T}$ for all cubes $P\in\DD$ with $Q\subset P\subset R,$ where $R$ is the root of $\mathcal{T}$.
\end{itemize}
	The minimal cubes of $\mathcal{T}$ will be called its \emph{stopping cubes}.
	
	For every $Q\in\DD$ we set $B(Q):= B(x_Q, c_0\ell(Q))$ and $B_Q = B(x_Q , \ell(Q))$, so that
	\begin{equation*}
		B(Q)\cap E \subset Q\subset B_Q.
	\end{equation*}
	Note that if $P\subset Q$, then $2B_P\subset 2B_Q$.
	
	Given a Radon measure $\mu$ and a cube $Q\in\mathcal{D}$, we define the $d$-dimensional density as
	\begin{equation*}
	\theta_\mu(Q)=\frac{\mu(Q)}{\ell(Q)^d}.
	\end{equation*}
		
	\subsection{Lower content regular sets}\label{subsec:LCR}
	
	Recall that a set $E \subset \R^n$ is lower content $(d, c_1)$-regular if, for all balls $B$ centered on $E$, 
	\begin{align*}
	\hdc(E \cap B) \geq c_1 r(B)^d. 
	\end{align*}
	We show below that sets with $d$-PBP are lower content regular.
	\begin{lemma}\label{l:low-reg-E}
		Let $E \subset \R^n$ be a set with $d$-PBP with constants $\delta>0$. Then $E$ is lower content $d$-regular with constant $c\sim_d \delta$.
	\end{lemma}
	
	\begin{proof}
		Without loss of generality, we identify $V$ with $\R^d$. For an arbitrary $\ve_1>0$, let $\dB$ be a family of balls in $\R^d$ so that $\sum_{B' \in \dB} r(B')^d \leq  \hd (\pi_V(E \cap B)) - \ve_1$. Note that, since these are balls in a $d$-plane, $\hd(B' \cap \pi_V(B\cap E)) \lesssim_d r(B')^d$. Let $\delta$ be the parameter with which $E$ satisfied $d$-PBP. Fix a ball $B$ centered on $E$, with $r(B) \leq \diam(E)$, and a plane $V$ in $B(V_B, \delta)$. Then, 
		\begin{align*}
		\delta r(B)^d \leq & \hd(\pi_V(E \cap B))\\
		& \leq \sum_{B' \in \dB} \hd(\pi_V(E \cap B) \cap B') \lesssim_d \sum_{B' \in \dB} r(B')^d \leq C\hd(\pi_V(E \cap B)) + C\ve_1 \\
		& \quad \quad \leq C'\hdc(\pi_V(E \cap B)) + C\ve_1. 
		\end{align*}
		Now, since $\pi_V$ is $1$-Lipschitz and $\ve_1$ was arbitrary, we obtain the lemma. The lower content regularity constant $c$ depends only on $\delta$ and $d$, since $C$ in the above display only depends on $d$. 
	\end{proof}
	\subsection{Frostman measure associated to \texorpdfstring{$E$}{E}}
	In Appendix \ref{appendix} we construct a particularly nice Frostman measure supported on a lower content regular set $E$. Its properties are listed in the lemma below.
	\begin{lemma}\label{lem:Frostman measure}
		Let $E\subset\R^n$ be a compact lower content $(d, c_1)$-regular set. Then, there exists a measure $\mu$ with $\supp\mu\subset E$ satisfying the following properties:
		\begin{enumerate}
			\item $\mu(E) = \hdc(E)\gtrsim c_1\diam(E)^d$,
			\item $\mu$ has polynomial growth, that is, there exists a constant $C_1\ge 1$ such that for all $x\in E$ and $0<r<\diam(E)$ we have
			\begin{equation*}
			\mu(B(x,r))\le C_1 r^d,
			\end{equation*}
			\item $\mu$ is doubling, that is, there exists a constant $C_{db}>1$ such that for all $x\in E$ and $0<r<\diam(E)$ we have
			\begin{equation}\label{eq:doubling}
			\mu(B(x,2r))\le C_{db}\, \mu(B(x,r))
			\end{equation}
			\item \label{it:density}the $d$-dimensional density of $\mu$ is almost monotone, that is, there exists a constant $A\ge 1$ such that if $P, Q\in\DD$, and $P\subset Q$, then
			\begin{equation*}
			\theta_\mu(P)\le A\, \theta_\mu(Q).
			\end{equation*} 
		\end{enumerate}
		In the above, $C_1$ may depend only on $d,n$, while $C_{db}$ and $A$ may also depend on the LCR-constant $c_1$.
	\end{lemma}
			
	Observe that thanks to \eqref{eq:doubling}, if $E$ is lower content regular and $\DD$ is the associated David-Christ lattice, then for all $Q\in\DD$
	\begin{equation}\label{eq:doubling cubes}
	\mu(Q)\le \mu(2B_Q)\lesssim_{c_1}\mu(B(Q))\le \mu(Q).
	\end{equation}
	
		\subsection{Quantitative rectifiability and \texorpdfstring{$\beta$}{beta}-numbers}
  	Recall that one of the quantitative notions of rectifiability introduced by David and Semmes in \cite{david1991singular} is given by the big pieces of Lipschitz graphs condition.
	\begin{definition}\label{def:BPLG}
		An Ahlfors $d$-regular set $E\subset\R^n$ has \emph{big pieces of Lipschitz graphs (BPLG)} if there exist constants $C_0, L> 0$ such that for any $x\in E$ and $0<r<\diam(E)$ there exists a Lipschitz graph $\Gamma\subset\R^n$ with $\lip(\Gamma)\le L$ and
		\begin{equation}\label{eq:BPLG}
		\mathcal{H}^d(E\cap B(x,r)\cap\Gamma)\ge C_0r^d.
		\end{equation}
	\end{definition}

 In \cite{david1993quantitative} David and Semmes conjectured that for Ahlfors regular sets, the PBP and BPLG conditions are equivalent. This was confirmed in a recent breakthrough result of Orponen \cite{orponen2020plenty}.
	\begin{theorem}[\cite{orponen2020plenty}]\label{thm:orponen}
		Suppose that a set $E\subset\R^n$ is Ahlfors $d$-regular. Then, it has $d$-PBP if and only if it has BPLG.
	\end{theorem}

	We recall different variants of $\beta$-numbers that we will use.
	\begin{definition}[Jones]\label{d:jones-beta}
		Let $E \subseteq \R^n$ and $B$ a ball. Define
		\begin{align*}
		\beta_{E,\infty}^d(B) = \frac{1}{r(B)}\inf_{L\in\mathcal{A}(n,d)} \sup\{\text{dist}(y,L) : y \in E \cap B\}.
		\end{align*}
		For $x\in\R^n$ and $r>0$ we set also $\beta_{E,\infty}^d(x,r) = \beta_{E,\infty}^d(B(x,r))$, and we will use the same notation for other types of $\beta$-numbers, defined below. We will also usually omit the superscript $d$.
	\end{definition}
	\begin{definition}[David-Semmes]
		Let $\mu$ be a Radon measure on $\R^n$, $B\subset\R^n$ a ball, and $1\le p<\infty$. The $L^p$ variant of Jones' $\beta$-numbers is defined as
		\begin{equation*}		
		\beta_{\mu,p}^{d}(B) = \inf_{L\in\mathcal{A}(n,d)}\left( \frac{1}{r(B)^d} \int_B \left( \frac{\dist(y, L)}{r(B)}\right)^p \, d\mu(y) \right)^{\frac{1}{p}}.
		\end{equation*} 
	\end{definition}
	The following is a special case of a classical result of David and Semmes.
 \begin{theorem}[\cite{david1991singular}]
		Let $E\subset\R^n$ be a bounded Ahlfors regular set with BPLG, and let $\mu=\mathcal{H}^d|_E$. Then, there exists a constant $C$, depending only on $n,d,$ and the BPLG and Ahlfors regularity constants of $E$, such that
		\begin{equation*}
		\int_E\int_0^{\diam(E)} \beta_{\mu,2}(x,r)^2\, \fr{dr}{r}d\mu(x)\le C\mu(E).
		\end{equation*}
	\end{theorem}
	Together with \thmref{thm:orponen} this gives the following.
	\begin{coro}\label{cor:PBPbetas}
		If an Ahlfors regular set $E$ has PBP, then the surface measure $\mu=\mathcal{H}^d|_E$ satisfies
		\begin{equation}\label{eq:PBPbetas}
		\int_E\int_0^{\diam(E)} \beta_{\mu,2}(x,r)^2\, \fr{dr}{r}d\mu(x)\le C\mu(E),
		\end{equation}
		where $C>0$ depends only on $n,d,$ and the PBP and Ahlfors regularity constants of $E$.
	\end{coro}

	The estimate \eqref{eq:PBPbetas} is extremely useful for estimating $\Gamma_{n,d}(E)$ due to the following result.
	\begin{theorem}[\cite{prat2012semiadd,girela-sarrion2018}]\label{thm:capacitybetas}
		Let $E\subset\R^n$ be compact. Then,
		\begin{equation}\label{eq:capacitybetas}
		\Gamma_{n,d}(E)\gtrsim \sup \{\mu(E)\, :\, \mu\in\mathcal{F}(E)\},
		\end{equation}
		where $\mathcal{F}(E)$ is the set of Radon measures with $\supp\mu\subset E$ satisfying the polynomial growth condition
		\begin{equation*}
		\mu(B(x,r))\le r^d\quad \text{for all $x\in\supp\mu$ and $r>0$},
		\end{equation*}
		and the flatness condition
		\begin{equation*}
		\iint_0^{\infty} \beta_{\mu,2}(x,r)^2\,\theta_\mu(x,r)\, \fr{dr}{r}d\mu(x)\le \mu(E).
		\end{equation*}
	\end{theorem}
	The theorem above is a combination of two results. Prat \cite{prat2012semiadd} related the capacity $\Gamma_{n,d}$ with the supremum over measures whose Riesz transform is in $L^2$, whereas Girela-Sarrión \cite{girela-sarrion2018} proved that this Riesz transform condition is true for measures $\mu\in\mathcal{F}(E)$. Prat's result was first proved by Tolsa \cite{tolsa2003painleve} for $d=1,\, n=2$, and by Volberg \cite{volberg2003calderon} in the case $d=n-1$. The result of Girela-Sarrión was first shown for $d=1,\, n=2,$ by Azzam and Tolsa \cite{azzam2015characterization}. Finally, let us mention that while \eqref{eq:capacitybetas} holds for all $1\le d<n$, in the codimension-1 case $d=n-1$ an estimate converse to \eqref{eq:capacitybetas} is also known to be true. This was shown for $d=1$ by Azzam and Tolsa \cite{tolsa2005bilipschitz,azzam2015characterization} and for general $d\in\mathbb{N}$ by Tolsa and the first author \cite{dkabrowski2021measures,tolsa2021measures}. Whether the same is true in codimension larger than 1 is an open problem.

	In the statement of \thmref{t:main} we used the content $\beta$-numbers of Azzam and Schul, which we recall below. For $1 \leq p<\infty$ and $A \subset \R^n$ Borel, we define the $p$-Choquet integral as 
	\begin{align*}
	\int_A f(x)^p\, d \hdc(x) := \int_0^\infty \hdc(\{x \in A\, :\, f(x)>t\}) \, t^{p-1}\, dt.
	\end{align*}
	We refer the reader to \cite{mattila} for more details on Hausdorff measures and content and to Section 2 and the Appendix of \cite{azzam2018analyst} for more details on Choquet integration.  
	
	\begin{lemma}[{\cite[Lemma 2.3]{azzam2018analyst}}]\label{l:jensen}
		Let $E \subseteq \R^n$ be either compact or bounded and open so that $\hd(E) >0,$ and let $f \geq 0$ be continuous on $E$. Then for $1 < p \leq \infty,$
		\[ \frac{1}{\hd_\infty(E)} \int_E f \, d\hd_\infty \lesssim_n \left( \frac{1}{\hd_\infty(E)} \int_E f^p \, d\hd_\infty \right)^\frac{1}{p} \] 
	\end{lemma}
	\begin{definition}[Azzam-Schul]\label{d:content-beta}
		Let $1 \leq p < \infty,$ $E \subseteq \R^n$ and $B$ a ball. For a $d$-dimensional plane $L$ define
		\begin{align} \label{e:beta-content} \beta^{d,p}_E(B,L) = \left( \frac{1}{r(B)^d} \int_{E \cap B} \left( \frac{\dist(y,L)}{r(B)} \right)^p \, d\mathcal{H}^d_\infty(y) \right)^\frac{1}{p}.
		\end{align}
		Then $\beta_{E}^{d,p}(B) = \inf_{L\in \dA(n,d)} \beta_{E}^{d,p}(B,L)$. 
	\end{definition}


\section{Coronization of lower content regular sets by Ahlfors regular sets}\label{s:coronization}
In this section we consider multiscale approximations of lower content regular sets in terms of Ahlfors regular sets. We distinguish two cases: sets of finite measure, and general sets, possibly with non $\sigma$-finite $\mathcal{H}^d$-measure. A coronization of sets with finite measure was proven in \cite{azzam2019quantitative}; the general case is new. 

\begin{remark}
    The key difference between the two approximation results is the following. In the finite measure case, the packing estimate \eqref{e:ADR-packing} involves diameters of cubes and the Hausdorff measure of the set. In the general case, we use the Frostman measure $\mu$ from Lemma \ref{lem:Frostman measure}, and the packing estimate \eqref{e:packing-corona-nonsigmafinite} involves the $\mu$-measure of cubes and the $\mu$-measure of the set. It seems unlikely either of the results implies the other.
\end{remark}

Throughout this section, $\dD$ will denote the Christ-David cubes from \lemref{theorem:christ} on a fixed lower content regular set $E$. Recall that $\Delta$ denotes the family of usual dyadic cubes on $\R^n$, and $\partial_d I$ is the $d$-dimensional skeleton of $I\in\Delta$.
\subsection{Finite measure case} \label{s:ER}    
	 The following approximation result will be used in the proof of Theorem \ref{t:main}.
	\begin{propo}[\cite{azzam2019quantitative}, Main Lemma] \label{l:corona}
		Let $N>0$ be an integer, and $E \subset\R^n$ be a compact lower content $(d, c_1)$-regular set. Let $Q_{0}\in \dD_{0}$ and $\dD_0^N=\bigcup_{k=0}^{N}\{Q\in \dD_{k}\, : \,Q\subseteq Q_0\}$. Then, there exists a family $\Top=\Top(N)\subseteq\dD_0^N$ such that for any $R \in \Top$ we have a tree of cubes denoted by $\Tree(R)$ with root $R$, the trees partition $\dD_0^N$
		\begin{equation*}
		\dD_0^N = \bigcup_{R\in\Top}\Tree(R),
		\end{equation*}
		and this partition has the following properties:
		\begin{enumerate}[leftmargin=0.8cm]
			\item There exists a constant $\eta=\eta(d, c_1)$ so that
			\begin{equation}
			\label{e:ADR-packing}
			\sum_{R \in \Top} \ell(R)^{d} \leq \eta^{-1} \dH^{d}(Q_0),
			\end{equation}
			and $\eta\to 0$ as $c_1 \to 0$\footnote{This is not explicitly stated in \cite{azzam2019quantitative}, but it can be deduced from the proof, specifically see (3.4), (3.5) and (3.10) there.}.
			\item Given $R\in \Top$ set
			\begin{align}\label{e:d_F}
			d_{R}(x) := \inf_{Q \in \Tree(R)} \ps{ \ell(Q) + \dist(x,Q)}.
			\end{align}
			For any $A>4$ and $\tau>0$, there is a collection $\cC_R \subset \Delta$ of disjoint dyadic cubes covering $AB_{R}\cap E$ such that the approximating set
			\[
			\Gamma_R:=\bigcup_{I\in \cC_R} \d_{d} I,\]
			satisfies:
			\begin{enumerate}[label=\textup{(}\alph*\textup{)}, leftmargin=0.8cm]
				\item $\Gamma_R$ is Ahlfors $d$-regular with constants depending on $A,\tau,d,$ and $c_1$.
				\item We have
				\begin{equation}
				\label{e:contains}
				AB_{R}\cap E \subseteq \bigcup_{I\in \cC_R} I\subseteq 2AB_{R}.
				\end{equation}
				
				\item $E$ is close to $\Gamma_R$ in $AB_{R}$ in the sense that
				\begin{equation}
				\label{e:adr-corona}
				\dist(x,\Gamma_R)\lec  \tau d_{R}(x) \;\; \mbox{ for all }x\in E\cap AB_{R}.
				\end{equation}
				\item The dyadic cubes in $\cC_R$ satisfy
				\begin{equation}
				\label{e:whitney-like}
				\ell(I)\sim \tau \inf_{x\in I} d_{R}(x) \mbox{ for all }I\in \cC_R.
				\end{equation}
			\end{enumerate}
		\end{enumerate}
	\end{propo}

\vspace{0.5cm}

\subsection{The general case}

In the following subsections we prove the following proposition, which is similar to Proposition \ref{l:corona}, with the key difference that we do not assume our set to have finite Hausdorff $d$-dimensional measure\footnote{Remark that, although $\mathcal{H}^d(E)<\infty$ is not explicitly assumed in Proposition \ref{l:corona}, in the infinite measure case that proposition gives no information, since the packing condition \eqref{e:ADR-packing} is vacuous.}
\begin{propo}\label{prop:coronization}
    Let $N>0$ be an integer and $E \subset \R^n$ be a lower content $(d,c_1)$-regular set. Let $\mu$ be the Frostman measure on $E$ constructed in Lemma \ref{lem:Frostman measure}. Let $Q_0 \in \dD_0$ and $\dD_0^N:=\cup_{k=0}^N \{ Q \in \dD_k \, : \, Q \subset Q_0\}$. Then, there exists a family $\Top=\Top(N) \subset \dD_0^N$ such that for any $R \in \Top$ there is a tree of cubes denoted by $\Tr(R)$ with root $R$, the trees partition $\dD_0^N$
    \begin{equation*}
		\dD_0^N = \bigcup_{R\in\Top}\Tr(R),
		\end{equation*}
    and this partition has the following properties:
    \begin{enumerate}
        \item We have
        \begin{equation}\label{e:packing-corona-nonsigmafinite}
            \sum_{R \in \Top} \theta_\mu(R) \mu(R) \leq 2 \theta_\mu(Q_0)\mu(Q_0).
        \end{equation}
        \item For every $R\in\Top$ and $Q\in\Tree(R)$ we have $\mu(Q)\sim_{c_1} \theta_\mu(R) \ell(Q).$
        \item Given $R \in \Top$, there is a collection $\cC_R$ of dyadic cubes covering $2 B_R \cap E$  such that the approximating set
        \begin{equation*}
            \Gamma_R := \bigcup_{I \in \cC_R} \partial_d I,
        \end{equation*}
        satisfies:
        \begin{enumerate}
            \item $\Gamma_R$ is Ahlfors $d$-regular, with Ahlfors regularity constants depending on $n,d$ and $c_1$.
            \item We have
            \begin{equation*}
                2B_R \cap E \subset \bigcup_{I \in \cC_R} I \subset 6 B_R.
            \end{equation*}
            \item $E$ is close to $\Gamma_R$ in $2B_R$ in the sense that if $x \in E\cap 2B_R$, then
            \begin{equation*}
                \dist(x, \Gamma_R) \lesssim d_R(x),
            \end{equation*}
            where $d_R(x)$ is as in \eqref{e:d_F}.
            \item The dyadic cubes $I\in\cC_R$ satisfy
				\begin{equation*}
				\ell(I)\sim \inf_{x\in I} d_{R}(x).
				\end{equation*}
        \end{enumerate}
        \item For every $R\in\Top$ there exists an Ahlfors regular measure $\nu$ such that $\nu = g \mathcal{H}^d|_{\Gamma_R}$ with $g\sim \theta_\mu(R)\one_{\Gamma_R}$, and we have
        \begin{equation}\label{eq:betasss}
            \sum_{Q\in\Tree(R)}\beta_{\mu,2}(2B_Q)^2\mu(Q)\lesssim_{c_1} \int_{\Gamma_R}\int_0^{\ell(R)} \beta_{\nu,2}(x,r)^2\, \fr{dr}{r}d\nu(x) + \theta_\mu(R)\mu(R).
        \end{equation}
    \end{enumerate}
\end{propo}
We will prove this proposition in the next few subsections. We fix $N>0$, $E\subset\R^n$, and $Q_0\in\DD_0$ as above. For simplicity, we allow the implicit constants in the estimates below to depend on $n,d$ and $c_1$, without further mentioning it.

	\subsection{Stopping time argument}	
	We begin by conducting a stopping time argument which will be used later on to define the collection $\Top$.
 
	Let $R\in\DD_0^{N}$. We will say that $Q\in\LD_0(R)$ (here $\LD$ stands for ``low density'') if $Q\in \DD_0^N,\ Q\subset R$, and
	\begin{equation*}
	\theta_\mu(Q)\le \tau\,\theta_\mu(R),
	\end{equation*}
	where $\tau = \rho^2 = 10^{-6}$. The refined family $\LD(R)$ consists of the maximal cubes from $\LD_0(R)$ (i.e., cubes $Q\in\LD_0(R)$ which are not properly contained in any other cube from $\LD_0(R)$). 
 We define also
	\begin{equation*}
	\End(R) = \big\{Q\in\DD_N\,:\, Q\subset R,\ Q\cap \bigcup_{P\in\LD(R)}P = \varnothing \big\}.
	\end{equation*}
	Finally, we set 
	\begin{equation*}
	\Stop(R) = \LD(R)\cup \End(R),
	\end{equation*}
	and
	\begin{equation*}
	\Tr(R) = \{Q\in\DD_0^N\, :\, Q\subset R,\ \text{there exists $P\in\Stop(R)$ such that } P\subset Q \}.
	\end{equation*}
	
	It follows immediately from the definition that $\Stop(R)$ is a family of pairwise disjoint cubes covering $R$. Moreover, $R\notin\LD(R)$ by the definition of $\LD_0(R)$, and $R\in\End(R)$ if and only if $R\in\DD_N$. Observe also that $\Stop(R)\subset \Tr(R)$.

 \vspace{0.5cm}
	
	In the lemma below we show that $\mu$ is $d$-Ahlfors regular at the scales and locations of $\Tr(R)$. This will be useful later on in the construction of an Ahlfors regular approximating measure $\nu$.
	\begin{lemma}\label{lem:muADR}
		Let $R\in\DD_0^{N}$. Then, for all $Q\in\Tr(R)$
		\begin{equation*}
			\tau\theta_\mu(R)\lesssim \theta_\mu(Q)\lesssim \theta_\mu(R).
		\end{equation*}
	\end{lemma}
	\begin{proof}
		The upper estimate $\theta_\mu(Q)\lesssim \theta_\mu(R)$ follows immediately from property (\ref{it:density}) in \lemref{lem:Frostman measure}, and the fact that $Q\subset R$.
		
		To see the lower bound $\tau\theta_\mu(R)\lesssim \theta_\mu(Q)$, note that this is obvious for $Q\in\Tr(R)\setminus\LD(R)$: for such cubes we have
		\begin{equation*}
			\theta_\mu(Q)>\tau\theta_\mu(R),
		\end{equation*}
		by the stopping time condition of $\LD_0(R)$. 
		
		Assume now that $Q\in\LD(R)$, and denote the parent of $Q$ by $Q^1$. Then, $Q^1\in\Tr(R)\setminus\LD(R)$, so that $\theta_\mu(Q^1)>\tau\theta_\mu(R)$. But now we see by the doubling property of $\mu$ \eqref{eq:doubling} that $\theta_\mu(Q)\gtrsim\theta_\mu(Q^1)$, so that
		\begin{equation*}
			\theta_\mu(Q)\gtrsim\theta_\mu(Q^1)\ge \tau\theta_\mu(R).
		\end{equation*}
	\end{proof}

\subsection{Coronization}
	We are ready to decompose $\DD_0^N$ into trees using the stopping time argument of the previous subsection.

    We define the family $\Top$ by induction. Set $\Top_0 = \{Q_0\}$. Assume that $\Top_k$ has already been defined for some $k\ge 0$, and let $R\in\Top_k.$ If $R\in\DD_N$, we set $\Next(R)=\varnothing.$ Otherwise, we define
    \begin{equation*}
        \Next(R) = \{Q\in\DD_0^N : Q^1\in\LD(R)\}.
    \end{equation*}
    An equivalent definition is that $\Next(R)$ consists of cubes $Q\in \DD_0^N$ satisfying $Q\notin\Tree(R)$ but $Q^1\in\Tree(R)$.
    
    We define
	\begin{equation*}
	\Top_{k+1} = \bigcup_{R\in\Top_k} \Next(R).
	\end{equation*}
	Note that for each $k$ the family $\Top_k$ consists of pairwise disjoint cubes. Observe also that there exists $1\le k_0\le N$ such that $\Top_{k_0}\neq\varnothing$, and then $\Top_k=\varnothing$ for all $k> k_0$. We define
	\begin{equation*}
		\Top = \bigcup_{k=0}^{k_0}\Top_k.
	\end{equation*}
	Remark that
	\begin{equation*}
		\DD_0^N = \bigcup_{R\in\Top}\Tr(R),
	\end{equation*}
	and the sum above is disjoint.

\subsection{Packing condition}
Now we prove the packing condition \eqref{e:packing-corona-nonsigmafinite}.
  First, we claim that the $\mu$-densities of cubes in $\Top$ decay geometrically: 
  \begin{equation*}
      \theta_\mu(R)\le \rho^{k}\theta_\mu(Q_0)\quad\text{for $R\in\Top_k,\ 0\le k\le k_0.$}
  \end{equation*}  
  Indeed, this follows from a simple induction argument: it is true for $k=0$ because $\Top_0=\{Q_0\}$. If $R\in\Top_{k+1}$, then $R^1\in\LD(P)$ for some $P\in\Top_k$. Since $\theta_\mu(P)\le \rho^{k}\theta_\mu(Q_0)$ by the inductive assumption, we get from the definition of $\LD(P)$ that
	\begin{equation*}
		\theta_\mu(R)=\frac{\mu(R)}{\ell(R)^d}\le\rho^{-1}\frac{\mu(R^1)}{\ell(R^1)^d} = \rho^{-1}\theta_\mu(R^1)\le\rho^{-1}  \tau\theta_\mu(P)
  = \rho\theta_\mu(P)\le \rho^{k+1}\theta_\mu(Q_0),
	\end{equation*}
	which closes the induction.
	
	Consequently,
	\begin{equation*}
		\sum_{k=0}^{k_0}\sum_{R\in\Top_k}\theta_\mu(R)\mu(R) \le \sum_{k=0}^{k_0}\sum_{R\in\Top_k}\rho^k\theta_\mu(Q_0)\mu(R) \le \sum_{k=0}^{k_0} \rho^k \theta_\mu(Q_0)\mu(Q_0),
	\end{equation*}
	where in the last estimate we used the fact that for each $k$ the cubes in $\Top_k$ are pairwise disjoint and contained in $Q_0$. Recalling that $\rho=0.001$, this concludes the proof of \eqref{e:packing-corona-nonsigmafinite}.
	
	\subsection{Regularizing trees}\label{subsec:regularizing}
	Fix $R\in\DD_0^{N}$. In what follows, it will be more convenient to work with \emph{regular} trees, in the sense that nearby stopping cubes have comparable sidelengths. In order to regularize $\Tr(R)$ we will use a technique from \cite{david1991singular} which is nowadays considered fairly standard. It will also be useful to work with certain cubes neighboring $R$. We provide the details below.
	
	We define the set of cubes neighboring $R$ as
	\begin{equation}\label{eq:defnbd}
	\nbd(R) = \{R'\in\DD\ :\ \ell(R')=\ell(R),\ R'\cap 2B_R\neq\varnothing \}.
	\end{equation}
	Note that $R\in\nbd(R)$. We define also an extended version of $\DD(R)$:
	\begin{equation*}
	\DD_*(R) = \bigcup_{R'\in\nbd(R)}\DD(R'),
	\end{equation*}
    and
    \begin{equation*}
        \DD_*^N = \{Q\in\DD_*(Q_0): Q\in\DD_k,\, 0\le k\le N\}.
    \end{equation*}
	
	Now we perform the regularization algorithm. Given $x\in\R^n$ set
	\begin{equation*}
	d_R(x) = \inf_{Q\in\Tr(R)} \dist(x,Q)+\ell(Q)
	\end{equation*}
	and for $Q\in\DD$ set
	\begin{equation*}
	d_R(Q)=	\max\big(\fr{1}{20}\inf_{x\in Q} d_R(x), 5\rho^N\big),
	\end{equation*}
	where the parameter $\rho=1/1000$ comes from the definition of $\DD$. Observe that the quantity $d_R(Q)$ is monotone in the sense that if $P\subset Q$, then $d_R(P)\ge d_R(Q)$. 	
	
	We define $\Reg(R)$ to be the family of maximal cubes $Q\in\DD_*(R)$ satisfying
	\begin{equation}\label{eq:defreg}
	\ell(Q)\le d_R(Q).
	\end{equation}	
	Note that the "$5\rho^N$" term in the definition of $d_R(Q)$ ensures that the inequality above is satisfied by all $Q\in\DD_N$, so that $\Reg(R)\subset\DD_*(R)\cap\DD_*^N$. Observe that the cubes in $\Reg(R)$ are pairwise disjoint, by maximality, and also
	\begin{equation}\label{eq:Regcovers}
	\bigcup_{Q\in\Reg(R)}Q=\bigcup_{R'\in\nbd({R})}R'\supset 2B_R\cap E.
	\end{equation} 
	 
	\begin{lemma}\label{lem:regularprop}
		 If $Q\in\Reg(R)$ and $x\in 5B_Q$, then $d_R(x)\sim \ell(Q)$. Consequently, if $Q,P\in\Reg(R)$, and $5B_Q\cap 5B_P\neq\varnothing$, then $\ell(Q)\sim\ell(P)$.
	\end{lemma}
	\begin{proof}
		Let $Q\in\Reg(R)$ and $x\in 5B_Q$. First we show that $d_R(x)\ge \ell(Q)$. By the definition of $\Reg(R)$, we have $\ell(Q)\le d_R(Q)$, so it suffices to show that $d_R(x)\ge d_R(Q)$, that is
		\begin{equation*}
		d_R(x)\ge \max\big(\fr{1}{20}\inf_{y\in Q} d_R(y), 5\rho^N\big).
		\end{equation*}
		
		Suppose the maximum above is achieved by $5\rho^N$, i.e. $d_R(Q)=5\rho^N$. Then, the estimate $d_R(x)\ge 5\rho^N$ is clear by the definition of $d_R(x)$, since $\Tr(R)\subset\DD_0^N$.
		
		Assume now that $d_R(Q)=\fr{1}{20}\inf_{y\in Q} d_R(y)$. Since the function $d_R$ is $1$-Lipschitz, we have
		\begin{multline*}
		\fr{1}{20}d_R(x)\ge \fr{1}{20}d_R(x_Q) - \fr{1}{20}|d_R(x)-d_R(x_Q)|\ge \fr{1}{20}d_R(x_Q) - \fr{5}{20}\ell(Q)\\
		\ge \fr{1}{20}\inf_{y\in Q} d_R(y) - \fr{1}{4}\ell(Q) = d_R(Q) -  \fr{1}{4}\ell(Q) \ge \fr{3}{4}\ell(Q).
		\end{multline*}
		Hence, $d_R(x)\ge 15\,\ell(Q)$.
		
		We move on to the estimate $d_R(x)\lesssim \ell(Q)$. Recall that, by the definition of $\Reg(R)$, $Q$ is a maximal cube satisfying $\ell(Q)\le d_R(Q)$. In particular, $Q^1$ satisfies
		\begin{equation*}
		\ell(Q^1)> d_R(Q^1) = \max\big(\fr{1}{20}\inf_{y\in Q^1} d_R(y), 5\rho^N\big)\ge \fr{1}{20}\inf_{y\in Q^1} d_R(y).
		\end{equation*}
		Let $y\in Q^1$ be such that $\ell(Q^1)\ge 1/20\, d_R(y)$. Since $x,y\in 2B_{Q^1}$, we may use the $1$-Lipschitz property of $d_R$ to conclude that
		\begin{equation*}
		\ell(Q^1)\ge \fr{1}{20} d_R(y) \ge \fr{1}{20} d_R(x) - \fr{1}{20} |d_R(y)  - d_R(x)|\ge \fr{1}{20} d_R(x) - \fr{4}{20} \ell(Q^1).
		\end{equation*}
		Thus, $\ell(Q)=\rho^{-1}\ell(Q^1)\gtrsim d_R(x)$.
	\end{proof}

	For every $R\in\Top$ we define the extended, regularized ``tree'' as
	\begin{equation*}
	\Tr_*(R) = \{Q\in\DD_*(R)\, :\, \text{there exists $P\in\Reg(R)$ such that } P\subset Q \}.
	\end{equation*}
	The family $\Tr_*(R)$ might not be a ``true'' tree since we cannot guarantee that all $Q\in\Tr_*(R)$ are contained in $R$. Nevertheless, it is a union of a bounded number of trees, each of the form $\Tr_*(R)\cap \DD(R'),\ R'\in\nbd(R)$.
	
	Remark that since $\Reg(R)\subset \DD_*^N$, we also have $\Tr_*(R)\subset\DD_*^N$. Below we prove that $\Tr_*(R)$ is larger than the original tree $\Tr(R)$.
	\begin{lemma}
		We have $\Tr(R)\subset\Tr_*(R)$.
	\end{lemma}
	\begin{proof}
		It suffices to show that each $Q\in\Stop(R)$ contains some $P\in\Reg(R)$. To this end, observe that if $x\in Q$, then from the definition of $d_R$ we have $d_R(x)\le \ell(Q)$. It follows that
		\begin{equation*}
		\fr{1}{20}\inf_{x\in Q} d_R(x)\le \fr{\ell(Q)}{20}.
		\end{equation*}
		There are two cases to consider.
		
		\emph{Case} $Q\in\Stop(R)\cap\DD_N$. Then $d_R(Q) = 5\rho^N = \ell(Q)$, so $Q$ satisfies \eqref{eq:defreg}, while $\ell(Q^1)>5\rho^N=d_R(Q^1)$, so that $Q^1$ does not satisfy \eqref{eq:defreg}. Consequently, $Q\in\Reg(R)$.
		
		\emph{Case} $Q\in\Stop(R)\cap\DD_0^{N-1}$. Then we have
			\begin{equation*}
			\ell(Q)>\max\bigg(\fr{\ell(Q)}{20}, 5\rho^N\bigg)\ge d_R(Q),
			\end{equation*}
			so that $Q$ does not satisfy \eqref{eq:defreg}. Taking into account \eqref{eq:Regcovers} we get that there exists $P\in\Reg(R)$ such that $P\subset Q$.
	\end{proof}

	Recall that in \lemref{lem:muADR} we proved that $\mu$ is Ahlfors regular at the scales and locations of $\Tr(R)$. In the lemma below we show that despite enlarging $\Tr(R)$ to its regularized version $\Tr_*(R)$ we did not lose this property.
	\begin{lemma}\label{lem:muADR2}
		For each $Q\in\Tr_*(R)$ we have
		\begin{equation}\label{eq:densest}
		\tau\,\theta_\mu(R)\lesssim \theta_\mu(Q)\lesssim \theta_\mu(R).
		\end{equation}
	\end{lemma}
	\begin{proof}
		If $Q\in\Tr(R)$, then this was already shown in \lemref{lem:muADR}. Suppose that $Q\in\Tr_*(R)\setminus\Tr(R)$. Let $P\in\Reg(R)$ be such that $P\subset Q$. By \lemref{lem:regularprop} we have $d_R(x_P)\sim\ell(P)$. By the definition of $d_R(x_P)$ there exists $P'\in\Tr(R)$ such that $d_R(x_P)\sim \dist(x_P,P')+\ell(P')$. Hence,
		\begin{equation}\label{eq:5}
		\ell(P)\sim \dist(x_P,P')+\ell(P').
		\end{equation} 
		
		If $\ell(P')<\ell(Q)$ let $Q'\in\Tr(R)$ be the ancestor of $P'$ with $\ell(Q')=\ell(Q)$, otherwise set $Q'=P'$. We claim that
		\begin{equation*}
		\ell(Q')\sim\ell(Q).
		\end{equation*}
		This is clearly the case if $\ell(P')<\ell(Q)$. On the other hand, if $Q'=P'$, then $\ell(Q')=\ell(P')\ge \ell(Q)$, and at the same time $\ell(P')\lesssim \ell(P)\le\ell(Q)$. This shows that $\ell(Q')\sim\ell(Q)$.
		
		Recalling \eqref{eq:5} we get that
		\begin{equation*}
		\dist(Q,Q')\lesssim\ell(Q)\sim\ell(Q').
		\end{equation*}
		Hence, $B_Q\subset CB_{Q'}\subset 2C B_{Q}$ for some $C\sim 1$. The estimate \eqref{eq:densest} then follows from the fact that it is satisfied by $Q'$ (which was shown in \lemref{lem:muADR}) and from the doubling property of $\mu$.

	\end{proof}
	The following auxiliary result will be useful later on.
	\begin{lemma}\label{lem:PcapQ}
		There exists $C_0>1$, with $C_0\sim 1$, such that for any $Q\in\Tr_*(R)$ and $P\in\Reg(R)$ satisfying $2B_P\cap 2B_Q\neq\varnothing$ we have 
		\begin{equation*}
		P\subset 2B_P\subset C_0B_Q.
		\end{equation*}
	\end{lemma}
	\begin{proof}
		To prove the lemma it suffices to show that $\ell(P)\lesssim\ell(Q)$. If $\ell(Q)\ge\ell(P),$ we are done. Assume $\ell(Q)\le \rho\, \ell(P)$, so that $Q\subset 3B_P$. 
		
		By the definition of $\Tr_*(R)$, there exists $Q'\in\Reg(R)$ such that $Q'\subset Q$. In particular, we have $B_{Q'}\cap 3B_P\neq\varnothing$, and since $Q', P\in\Reg(R)$, we get from \lemref{lem:regularprop} that $\ell(P)\sim\ell(Q')\le \ell(Q)$.
	\end{proof}
	
	\subsection{Construction of $\Gamma_R$ and the approximating measure $\nu$}\label{subsec:approx}
	In this subsection we construct a $d$-Ahlfors regular set $\Gamma_R$ that approximates $E$ at the level of $\Tr_*(R)$. \\
 
	Recall that $\Delta$ denotes the family of usual half-open dyadic cubes in $\R^n$. For each $Q\in\Tr_*(R)$ set
	\begin{equation*}
	\Delta_Q = \{I\in\Delta\ :\ \ell(Q)/2\le  \ell(I)<\ell(Q),\, I\cap Q\neq\varnothing\}.
	\end{equation*}
	In particular, 
	\begin{equation}\label{eq:DQsubset2BQ}
	Q\subset \bigcup_{I\in\Delta_Q} I\subset 2B_Q.
	\end{equation}
	We define also
	\begin{equation*}
	Q_{\Gamma} = \bigcup_{I\in\Delta_Q}\d_d I,
	\end{equation*}
	where $\d_d I$ denotes the $d$-dimensional skeleton of $I\in\Delta_Q$. Observe that
	\begin{equation}\label{eq:QGamma2BQ}
	Q_\Gamma\subset 2B_Q,
	\end{equation}
	and
	\begin{equation}\label{eq:QGammameas}
	\hd(Q_{\Gamma}) \sim \ell(Q)^d
	\end{equation}
	because $\hd(\d_d I)\sim \ell(Q)^d$ for each $I\in\Delta_Q$, and $\#\Delta_Q\lesssim 1$.
	
	Finally, we set
    \begin{equation*}
        \cC_R = \bigcup_{Q\in\Reg(R)}\Delta_Q
    \end{equation*}
    and
	\begin{equation*}
	\Gamma_R = \bigcup_{Q\in\Reg(R)} Q_{\Gamma} = \bigcup_{I\in\cC_R}\d_d I.
	\end{equation*}
	Note that 
	\begin{equation}\label{eq:Gamma6BR}
	\Gamma_R\subset \bigcup_{R'\in\nbd(R)}2B_{R'}\subset 6B_{R},
	\end{equation}
	by \eqref{eq:QGamma2BQ} and the definition of $\nbd(R)$ \eqref{eq:defnbd}. At the same time, by \eqref{eq:Regcovers} we have
    \begin{equation*}
        2B_R\cap E\subset \bigcup_{I\in\cC_R}I.
    \end{equation*}
    This gives property (b) from Proposition \ref{prop:coronization}.

    Property (c) is also easy to establish: if $x\in E\cap 2B_R$, then there is a unique $Q\in\Reg(R)$ with $x\in Q$. Hence,
    \begin{equation*}
        \dist(x,\Gamma_R)\le \dist(x,Q_\Gamma) \lesssim \ell(Q) \sim d_R(x),
    \end{equation*}
    where in the last estimate we used Lemma \ref{lem:regularprop}. 
    
    Property (d) follows immediately from Lemma \ref{lem:regularprop}.
    
 Now we turn to property (a): Ahlfors regularity of $\Gamma_R$. To simplify notation, below we drop the subscript and we write $\Gamma$ instead of $\Gamma_R$.

    We begin by showing that although the family $\cC_R$ is not necessarily disjoint, it is not too far off: it has bounded intersection.
	\begin{lemma}
		We have
		\begin{equation}\label{eq:sumoneQ}
		\sum_{Q\in\Reg(R)}\one_{Q_\Gamma} \sim \one_\Gamma.
		\end{equation}
	\end{lemma}
	\begin{proof}
		Observe that if $Q,P\in\Reg(R)$ are such that $2B_Q\cap 2B_P=\varnothing$, then $Q_\Gamma\cap P_\Gamma=\varnothing$, by \eqref{eq:QGamma2BQ}. 
		
		It may happen that $Q_\Gamma\cap P_\Gamma\neq\varnothing$ for $Q,P\in\Reg(R)$ such that $2B_Q\cap2B_P\neq\varnothing$. However, for any fixed $Q\in\Reg(R)$ there is only a bounded number of $P\in\Reg(R)$ where this can happen, by \lemref{lem:regularprop}.
	\end{proof}
	
	The set $\Gamma$ approximates $E$ at the scales and locations of $\Tr_*(R)$. Now we also define an Ahlfors regular measure $\nu$ approximating $\mu$. To that end, we define a density $g:\Gamma\to\R$, so that $\nu=g\,\hd|_\Gamma$.
	
	For each $Q\in\Reg(R)$ we define $g_Q:\Gamma\to \R$ as
	\begin{equation*}
	g_Q = \fr{\mu(Q)}{\hd(Q_\Gamma)}\one_{Q_\Gamma},
	\end{equation*}
	so that
	\begin{equation}\label{eq:gQmeas}
	\int_\Gamma g_Q\ d\hd = \mu(Q).
	\end{equation}
	Note also that 
	\begin{equation}\label{eq:gQsimtheta}
	g_Q\sim \theta_\mu(Q)\one_{Q_\Gamma}\sim \theta_\mu(R)\one_{Q_\Gamma},
	\end{equation}
	by \eqref{eq:QGammameas} and \lemref{lem:muADR2}.
	
	Set
	\begin{equation*}
	g=\sum_{Q\in\Reg(R)}g_Q
	\end{equation*}
	and
	\begin{equation}\label{eq:nu-def}
	\nu = g\,\hd|_\Gamma.
	\end{equation}
	Remark that for any $x\in\Gamma$
	\begin{equation*}
		g(x) = \sum_{Q\in\Reg(R)} g_Q(x) \overset{\eqref{eq:gQsimtheta}}{\sim} \theta_\mu(R)\sum_{Q\in\Reg(R)} \one_{Q_\Gamma}(x)\sim \theta_\mu(R),
	\end{equation*}
	where in the last estimate we used \eqref{eq:sumoneQ}. It follows that
	\begin{equation}\label{eq:nuGammarel}
	\nu=  \theta_\mu(R)\cdot\, h\hd|_\Gamma,
	\end{equation}
	where $h(x)=g(x)/\theta_\mu(R)\sim \one_\Gamma$.

 
	
	\begin{lemma}\label{lem:GammaADR}
		The measure $\nu$ constructed above is $d$-Ahlfors regular. More precisely, for $x\in\Gamma$ and $0<r<\diam(\Gamma)$ we have
		\begin{equation}\label{eq:nuADR}
		\nu(B(x,r))\sim \theta_\mu(R)\, r^d.
		\end{equation}
		In consequence, $\Gamma$ is a $d$-Ahlfors regular set, with constants depending only on $n,d,$ and $c_1$.
	\end{lemma}
	\begin{proof}
		Fix $x\in\Gamma$, and let $Q\in\Reg(R)$ be a cube such that $x\in Q_\Gamma$; if $Q$ is non-unique, just choose one.
		
		First we check that $\nu$ is upper Ahlfors regular. Assume that $0<r\le \ell(Q)$. Then, 
		\begin{multline*}
		\nu(B(x,r)) = \sum_{P\in\Reg(R)}\int_{B(x,r)} g_P\ d\hd\overset{\eqref{eq:gQsimtheta}}{\lesssim} \sum_{P\in\Reg(R),\, P_{\Gamma}\cap B(x,r)\neq\varnothing} \int_{B(x,r)}\theta_\mu(R)\one_{P_\Gamma}\ d\hd\\
		 = \theta_\mu(R)\sum_{P\in\Reg(R),\, P_{\Gamma}\cap B(x,r)\neq\varnothing} \hd({P_\Gamma\cap B(x,r)})\lesssim \theta_\mu(R)\sum_{P\in\Reg(R),\, P_{\Gamma}\cap B(x,r)\neq\varnothing} r^d,
		\end{multline*}
		where in the last estimate we used the fact that each $P_\Gamma$ is upper Ahlfors regular. To estimate the last term above, note that $B(x,r)\subset 2B_Q$ and so $B(x,r)$ may intersect only a bounded number $P_\Gamma,\ P\in\Reg(R),$ by \eqref{eq:QGamma2BQ} and \lemref{lem:regularprop}. Hence,
		\begin{equation*}
		\nu(B(x,r)) \lesssim \theta_\mu(R)\sum_{P\in\Reg(R),\, P_{\Gamma}\cap B(x,r)\neq\varnothing} r^d \sim \theta_\mu(R)\, r^d.
		\end{equation*}
		
		Now suppose that $\ell(Q)<r<\diam(\Gamma)$.
		\begin{multline*}
		\nu(B(x,r)) = \sum_{P\in\Reg(R)}\int_{B(x,r)} g_P\ d\hd \le \sum_{P\in\Reg(R),\, P_{\Gamma}\cap B(x,r)\neq\varnothing} \int g_P\ d\hd\\ \overset{\eqref{eq:gQmeas}}{=} \sum_{P\in\Reg(R),\, P_{\Gamma}\cap B(x,r)\neq\varnothing} \mu(P).
		\end{multline*}
		Let $Q'\in\Tr_*(R)$ be the minimal cube satisfying $Q\subset Q'$ and $\ell(Q')>r$. Note that $\ell(Q')\sim r$, and $B(x,r)\subset 2B_{Q'}$. If $P\in\Reg(R)$ satisfies $P_{\Gamma}\cap B(x,r)\neq\varnothing$, then in particular we have $2B_P\cap 2B_{Q'}\neq\varnothing$. Thus, by \lemref{lem:PcapQ}, we have $P\subset C_0B_{Q'}$.	
		In consequence,
		\begin{equation*}
		\nu(B(x,r)) \le  \sum_{P\in\Reg(R),\, P_{\Gamma}\cap B(x,r)\neq\varnothing} \mu(P)\le \mu(C_0B_{Q'})\sim \mu(Q'),
		\end{equation*}
		where in the last estimate we used the doubling property of $\mu$. We have $\mu(Q')\sim \theta_\mu(R)\ell(Q')^d$ by \lemref{lem:muADR2}. Since $\ell(Q')\sim r$, we get the desired upper bound
		\begin{equation*}
		\nu(B(x,r))\lesssim \theta_\mu(R)\, r^d.
		\end{equation*}
		
		We turn to the lower regularity of $\nu$. The proof at scales $0<r\le20\,\ell(Q)$ is very similar to the one before. We have
		\begin{multline*}
		\nu(B(x,r)) = \sum_{P\in\Reg(R)}\int_{B(x,r)} g_P\ d\hd\overset{\eqref{eq:gQsimtheta}}{\gtrsim}  \int_{B(x,r)}\theta_\mu(R)\one_{Q_\Gamma}\ d\hd\\
		= \theta_\mu(R)\ \hd({Q_\Gamma\cap B(x,r)})\sim \theta_\mu(R)\, r^d,
		\end{multline*}
		where in the last line we used Ahlfors regularity of $Q_\Gamma$.
		
		Concerning scales $20\,\ell(Q)\le r<\diam(\Gamma)$, we proceed as follows:
		\begin{multline*}
		\nu(B(x,r)) = \sum_{P\in\Reg(R)}\int_{B(x,r)} g_P\ d\hd \ge \sum_{P\in\Reg(R),\, P_{\Gamma}\subset B(x,r)} \int g_P\ d\hd\\ \overset{\eqref{eq:gQmeas}}{=} \sum_{P\in\Reg(R),\, P_{\Gamma}\subset B(x,r)} \mu(P).
		\end{multline*}
		Let $Q'\in\Tr_*(R)$ be the maximal cube satisfying $Q\subset Q'$ and $2B_{Q'}\subset B(x,r)$ (such cube exists because $2B_Q\subset B(x,r)$). Note that $\ell(Q')\sim r$. 
		
		We claim that if $P\in\Reg(R)$ and $P\subset Q'$, then $P_\Gamma\subset B(x,r)$ (so in particular, $P$ appears in the sum on the right hand side above). Indeed, we have $P_\Gamma\subset 2B_P\subset 2B_{Q'}\subset B(x,r)$, by the definition of $Q'$. Thus,
		\begin{multline*}
		\nu(B(x,r)) \ge \sum_{P\in\Reg(R),\, P_{\Gamma}\subset B(x,r)} \mu(P)\\
		\ge \sum_{P\in\Reg(R),\, P\subset Q'} \mu(P)=\mu(Q')\sim\theta_\mu(R)\,\ell(Q')^d\sim \theta_\mu(R)\, r^d,
		\end{multline*}
		where we have used \lemref{lem:muADR2} and the fact that $\ell(Q')\sim r$. This finishes the proof of Ahlfors egularity of $\nu$.
		
		As a consequence of \eqref{eq:nuGammarel} and \eqref{eq:nuADR} we get that $\Gamma$ is a $d$-Ahlfors regular set, with Ahlfors regularity constants depending only on $n,d,$ and the the lower content regularity constant $c_1$. 
	\end{proof}
    We have established property (a) from Proposition \ref{prop:coronization}.
    \subsection{$\beta$-numbers estimate}
    To complete the proof of Proposition \ref{prop:coronization} it remains to prove the estimate \eqref{eq:betasss}, namely
    \begin{equation*}
            \sum_{Q\in\Tree(R)}\beta_{\mu,2}(2B_Q)^2\mu(Q)\lesssim \int_{\Gamma}\int_0^{\ell(R)} \beta_{\nu,2}(x,r)^2\, \fr{dr}{r}d\nu(x) + \theta_\mu(R)\mu(R).
        \end{equation*}
    We begin by showing the following auxiliary estimate.
    \begin{lemma}
		For any $Q\in\Tr(R)$ we have
		\begin{equation}\label{eq:betaestmu}
			\beta_{\mu,2}(2B_Q)^2\lesssim \beta_{\nu,2}(C_0B_Q)^2 + \ell(Q)^{-d}\sum_{P\in\Reg(R),\, 2B_P\subset C_0B_Q} \fr{\ell(P)}{\ell(Q)}\mu(P),
		\end{equation}
		where $C_0\sim 1$ is the constant from \lemref{lem:PcapQ}.
	\end{lemma}
	\begin{proof}
		Fix $Q\in\Tr(R)$, and let $L_Q$ be a $d$-plane minimizing $\beta_{\nu,2}(C_0B_Q)$.
		
		Observe that for any $Q\in \Tr(R)$, the set $2B_Q\cap E$ is fully covered by $P\in\Reg(R)$ satisfying $P\cap 2B_Q\neq\varnothing$. Indeed, this is true for $R$ because by the definition of $\nbd(R)$ and $\Reg(R)$ we have 
		\begin{equation*}
		2B_R\cap E\subset \bigcup_{P\in\Reg(R),\, P\cap 2B_R\neq\varnothing}P.
		\end{equation*}
		Since $2B_Q\subset 2B_R$ for any $Q\in\Tr(R)$, the claim follows.
		
		The observation above gives
		\begin{multline*}
		\beta_{\mu,2}(2B_Q)^2\ell(Q)^d \lesssim \int_{2B_Q}\bigg(\fr{\dist(x,L_Q)}{\ell(Q)}\bigg)^2\, d\mu(x)\\
		\le\sum_{P\in\Reg(R),\, P\cap 2B_Q\neq\varnothing} \int_{P}\bigg(\fr{\dist(x,L_Q)}{\ell(Q)}\bigg)^2\, d\mu(x)\\
		\le \sum_{P\in\Reg(R),\, 2B_P\subset C_0B_Q} \int_{P}\bigg(\fr{\dist(x,L_Q)}{\ell(Q)}\bigg)^2\, d\mu(x),
		\end{multline*}
		where in the last line we used \lemref{lem:PcapQ}.
		
		Let $P\in\Reg(R),\, 2B_P\subset C_0B_Q$. Then,
		\begin{multline*}
		\int_{P}\bigg(\fr{\dist(x,L_Q)}{\ell(Q)}\bigg)^2\, d\mu(x)
		 = \int_{\Gamma}\bigg(\fr{\dist(x,L_Q)}{\ell(Q)}\bigg)^2g_P(x)\, d\hd(x)\\
		  + \int \bigg(\fr{\dist(x,L_Q)}{\ell(Q)}\bigg)^2\, (\one_P(x)d\mu(x)-g_P(x)d\hd|_\Gamma(x))
		 \eqqcolon I_1(P)+I_2(P).
		\end{multline*}		
		It follows that
		\begin{equation*}
		\beta_{\mu,2}(2B_Q)^2\ell(Q)^d\lesssim \sum_{P\in\Reg(R),\, 2B_P\subset C_0B_Q}I_1(P) + \sum_{P\in\Reg(R),\, 2B_P\subset C_0B_Q}I_2(P) \eqqcolon S_1+S_2.
		\end{equation*}
		
		Estimating $S_1$ is straightforward. Recall that for $P\in\Reg(R)$ we have $\supp g_P\subset P_\Gamma\subset 2B_P$. Hence,
		\begin{multline*}
		S_1 = \sum_{P\in\Reg(R),\, 2B_P\subset C_0B_Q}\int_{\Gamma}\bigg(\fr{\dist(x,L_Q)}{\ell(Q)}\bigg)^2g_P(x)\, d\hd(x)\\
		 \le\sum_{P\in\Reg(R)}\int_{\Gamma\cap C_0B_Q}\bigg(\fr{\dist(x,L_Q)}{\ell(Q)}\bigg)^2g_P(x)\, d\hd(x)\\
		  = \int_{\Gamma\cap C_0B_Q}\bigg(\fr{\dist(x,L_Q)}{\ell(Q)}\bigg)^2g(x)\, d\hd(x)\\
		 = \int_{C_0B_Q}\bigg(\fr{\dist(x,L_Q)}{\ell(Q)}\bigg)^2\, d\nu(x)\sim\beta_{\nu,2}(C_0B_Q)^2\ell(Q)^d,
		\end{multline*}
		where in the last part we used our choice of $L_Q$. This gives us the first term from the right hand side of \eqref{eq:betaestmu}.
		
		We turn to estimating $S_2$. Let $P\in\Reg(R)$. Using the fact that $\int_\Gamma g_P\, d\hd = \mu(P)$ we get
		\begin{multline*}
		I_2(P)=\int \bigg(\fr{\dist(x,L_Q)}{\ell(Q)}\bigg)^2\, (\one_P(x)d\mu(x)-g_P(x)d\hd|_\Gamma(x))\\
		 = \int \bigg(\bigg(\fr{\dist(x,L_Q)}{\ell(Q)}\bigg)^2 - \bigg(\fr{\dist(x_P,L_Q)}{\ell(Q)}\bigg)^2\bigg)\, (\one_P(x)d\mu(x)-g_P(x)d\hd|_\Gamma(x)).
		\end{multline*}
		For $x\in P\cup\supp g_P\subset 2B_P\subset C_0 B_Q$ we have
		\begin{multline*}
		\bigg|\bigg(\fr{\dist(x,L_Q)}{\ell(Q)}\bigg)^2 - \bigg(\fr{\dist(x_P,L_Q)}{\ell(Q)}\bigg)^2\bigg|
		\le \fr{|x-x_P|}{\ell(Q)}\cdot \fr{\dist(x,L_Q)+\dist(x_P,L_Q)}{\ell(Q)}\\
		\lesssim \fr{\ell(P)}{\ell(Q)}\cdot\fr{\ell(Q)}{\ell(Q)}=\fr{\ell(P)}{\ell(Q)}.
		\end{multline*}
		Thus,
		\begin{equation*}
		I_2(P)\lesssim \fr{\ell(P)}{\ell(Q)}\mu(P).
		\end{equation*}
		Summing over $P\in\Reg(R)$ with $2B_P\subset C_0B_Q$ yields
		\begin{equation*}
		S_2\lesssim\sum_{P\in\Reg(R),\, 2B_P\subset C_0B_Q} \fr{\ell(P)}{\ell(Q)}\mu(P).
		\end{equation*}
		This corresponds to the second term from the right hand side of \eqref{eq:betaestmu}.
	\end{proof}
	\begin{lemma}\label{lem:betasumest}
		We have
		\begin{equation}\label{eq:betasumest}
			\sum_{Q\in\Tr(R)}\beta_{\mu,2}(2B_Q)^2\mu(Q)\lesssim \sum_{Q\in\Tr(R)}\beta_{\nu,2}(C_0B_Q)^2\mu(Q) + \theta_\mu(R)\,\mu(R).
		\end{equation}
	\end{lemma}
	\begin{proof}
		In the previous lemma we showed that for any $Q\in\Tr(R)$ we have \eqref{eq:betaestmu}. Multliplying by $\mu(Q)$ and summing over $Q\in\Tr(R)$ yields
		\begin{multline*}
			\sum_{Q\in\Tr(R)}\beta_{\mu,2}(2B_Q)^2\mu(Q)\lesssim \sum_{Q\in\Tr(R)}\beta_{\nu,2}(C_0B_Q)^2\mu(Q)\\
			 + \sum_{Q\in\Tr(R)}\theta_\mu(Q)\sum_{P\in\Reg(R),\, 2B_P\subset C_0B_Q} \fr{\ell(P)}{\ell(Q)}\mu(P).
		\end{multline*}
		Thus, in order to obtain \eqref{eq:betasumest} it suffices to prove that
		\begin{equation*}
			\sum_{Q\in\Tr(R)}\theta_\mu(Q)\sum_{P\in\Reg(R),\, 2B_P\subset C_0B_Q} \fr{\ell(P)}{\ell(Q)}\mu(P)\lesssim \theta_\mu(R)\mu(R).
		\end{equation*}
	First, recall that by \lemref{lem:muADR} we have $\theta_\mu(Q)\sim\theta_\mu(R)$ for all $Q\in\Tr(R)$. So we only need to show
	\begin{equation}\label{eq:goal1}
		\sum_{Q\in\Tr(R)}\sum_{P\in\Reg(R),\, 2B_P\subset C_0B_Q} \fr{\ell(P)}{\ell(Q)}\mu(P)\lesssim \mu(R).
	\end{equation}
	Changing the order of summation transforms the left hand side to
	\begin{equation*}
		\sum_{P\in\Reg(R)}\mu(P)\sum_{Q\in\Tr(R),\, 2B_P\subset C_0B_Q} \fr{\ell(P)}{\ell(Q)}.
	\end{equation*}
	The inner sum is essentially a geometric series: if $P\in\Reg(R)$ is fixed, then for any generation $k\in [0, N]$, there is at most a bounded number of $Q\in\Tr(R)\cap\DD_k$ such that $2B_P\subset C_0B_Q$ (the bound depends on $C_0\sim 1$). Moreover, all such $Q$ necessarily satisfy $C_0\ell(Q)\ge \ell(P)$. Hence,
	\begin{equation*}
		\sum_{P\in\Reg(R)}\mu(P)\sum_{Q\in\Tr(R),\, 2B_P\subset C_0B_Q} \fr{\ell(P)}{\ell(Q)}\lesssim_{C_0}\sum_{P\in\Reg(R)}\mu(P)\le \mu(6B_R)\sim \mu(R),
	\end{equation*}
	where in the last two estimates we used the fact that the cubes in $\Reg(R)$ are pairwise disjoint and contained in $6B_R$, and also the doubling property of $\mu$. This finishes the proof of \eqref{eq:betasumest}.
	\end{proof}

 To complete the proof of \eqref{eq:betasss} it remains to show the following standard estimate.
 \begin{lemma}
     We have
      \begin{equation*}
     \sum_{Q\in\Tr(R)}\beta_{\nu,2}(C_0B_Q)^2\mu(Q) \lesssim \int_{\Gamma_R}\int_0^{\ell(R)} \beta_{\nu,2}(x,r)^2\, \fr{dr}{r}d\nu(x) + \theta_\mu(R)\mu(R).
 \end{equation*}
 \end{lemma}
\begin{proof}
     We only sketch the proof. First, observe that for every $Q\in\Tree(R)$, $x\in \Gamma_R\cap 5B_Q$, and $10C_0\ell(Q)<r<20C_0\ell(Q)$ we have $C_0B_Q\subset B(x,r)$ and $r\sim r(C_0B_Q).$ By the definition of $\beta$-numbers, this gives
 \begin{equation*}
     \beta_{\nu,2}(C_0B_Q)^2\lesssim \beta_{\nu,2}(x,r)^2.
 \end{equation*}
 Integrating over $x\in \Gamma_R\cap 5B_Q$ and $10C_0\ell(Q)<r<20C_0\ell(Q)$ implies
 \begin{equation*}
     \beta_{\nu,2}(C_0B_Q)^2\nu(\Gamma_R\cap 5B_Q)\lesssim \int_{\Gamma_R\cap 5B_Q}\int_{10C_0\ell(Q)}^{20C_0\ell(Q)} \beta_{\nu,2}(x,r)^2\frac{dr}{r}\, d\nu(x).
 \end{equation*}
 Recalling that $\nu(\Gamma_R\cap 5B_Q)\sim\theta_\mu(R)\hd(\Gamma_R\cap 5B_Q)\sim\theta_\mu(R)\ell(Q)\sim\mu(Q)$ and summing over all $Q\in\Tree(R)$ yields
 \begin{equation*}
     \sum_{Q\in\Tree(R)}\beta_{\nu,2}(C_0B_Q)^2\mu(Q)\lesssim \sum_{Q\in\Tree(R)}\int_{\Gamma_R\cap 5B_Q}\int_{10C_0\ell(Q)}^{20C_0\ell(Q)} \beta_{\nu,2}(x,r)^2\frac{dr}{r}\, d\nu(x).
 \end{equation*}
 Observe that every pair $(x,r)\in\Gamma_R\times (0,\ell(R))$ belongs to at most a bounded number of sets from 
 \begin{equation*}
     \big\{\big(\Gamma_R\cap 5B_Q, (10C_0\ell(Q),20C_0\ell(Q))\big)\big\}_{Q\in\Tree(R)}.
 \end{equation*}
 Thus,
 \begin{equation*}
     \sum_{Q\in\Tree(R)}\beta_{\nu,2}(C_0B_Q)^2\mu(Q)\lesssim \int_{\Gamma_R}\int_{0}^{20C_0\ell(R)} \beta_{\nu,2}(x,r)^2\frac{dr}{r}\, d\nu(x).
 \end{equation*}
 Finally, it follows from the definition of $\beta$-numbers that
 \begin{equation*}
     \int_{\Gamma_R}\int_{\ell(R)}^{20C_0\ell(R)} \beta_{\nu,2}(x,r)^2\frac{dr}{r}\, d\nu(x) \lesssim \frac{\nu(\Gamma_R)}{\ell(R)^d}\nu(\Gamma_R)\sim \theta_\mu(R)\mu(R),
 \end{equation*}
 and so
 \begin{equation*}
     \sum_{Q\in\Tr(R)}\beta_{\nu,2}(C_0B_Q)^2\mu(Q) \lesssim \int_{\Gamma_R}\int_0^{\ell(R)} \beta_{\nu,2}(x,r)^2\, \fr{dr}{r}d\nu(x) + \theta_\mu(R)\mu(R).
 \end{equation*}
\end{proof}
Together with \lemref{lem:betasumest} this concludes the proof of \eqref{eq:betasss}, and of Proposition \ref{prop:coronization}.

\section{The PBP property of approximating sets}\label{s:PBP}
In this section we show that when the set $E$ is approximated as in Proposition \ref{prop:coronization}, then the approximating sets inherits the PBP property, with comparable constants. 

	\begin{propo}\label{lem:GammaPBP}
		Let $E \subset \R^n$ be compact and with the $d$-PBP. Let $N$ be an integer and let $\Top$ be a family of cubes as in Proposition \ref{l:corona} or Proposition \ref{prop:coronization} \footnote{Note that we may apply either proposition to $E$ thanks to Lemma \ref{l:low-reg-E} which says that $d$-PBP implies lower content $d$-regularity with constant depending on that of the PBP property.}. Then, for every $R\in\Top$ the set $\Gamma_R$ has $d$-PBP, with PBP constants depending only on $n, d$, and the PBP constants of $E$.
	\end{propo}

We first need to prove the following technical lemma.
 
	\begin{lemma}\label{lem:skeletonsprojection}
		Let $I\in\Delta$, and $V\in\dG(n,d)$. Then,
		\begin{equation}\label{eq:skeletonproj}
			\pi_V({I}) \subset \pi_V(\d_d I).
		\end{equation}
	\end{lemma}
	\begin{proof}
		Let $p \in \pi_V(I)$. Then $W\coloneqq \pi^{-1}(\{p\})$ is an $(n-d)$-dimensional plane intersecting $I$. We are going to show that
		\begin{equation}\label{eq:planeskeleton}
		W\cap\partial_d I\neq\varnothing,
		\end{equation}
		so that $p\in \pi_V(\d_d I)$, and in consequence \eqref{eq:skeletonproj} holds.
		
		By rotating and translating, we may assume that
		\begin{equation*}
		W=\spn(e_1,\dots,e_{n-d})=\{(x_1,\dots,x_{n-d},0,\dots,0)\,:\, (x_1,\dots,x_{n-d})\in\R^{n-d} \}.
		\end{equation*}
		In this coordinates, the cube $I$ can be described by a system of inequalities
		\begin{equation*}
		I = \{x\in\R^n\,:\, a\le Lx\le b\},
		\end{equation*}
		for some orthogonal matrix $L\in SO(n)$ and some vectors $a=(a_i)_{i=1,\dots,n}\in\R^n$ and $b=(a_i+\ell(I))_{i=1,\dots,n}\in\R^n$ (note that before the rotation and translation, we would have $L=\id_n$, and $a$ a vector consisting of dyadic numbers). 
		
		Recall that all $d$-dimensional faces of $I$ are of the following form: for any choice of indices $A,B\subset\{1,\dots,n\}$ with $A\cap B=\varnothing$ and $\#(A\cup B)=n-d$ we set
		\begin{equation*}
		F(A,B) = \{x\in\R^n\,:\, a\le Lx\le b,\, a_i=L_i\cdot x \text{ for $i\in A$, } L_i\cdot x = b_i \text{ for $i\in B$}\}.
		\end{equation*}
		Above, $L_i$ denotes the $i$-th row of $L$. This is a one-to-one representation, in the sense that for any such $A,B$ we get a $d$-dimensional face of $I$, and each $d$-dimensional face $F$ satisfies $F=F(A,B)$ for a unique choice of $A$ and $B$.
		
		Recalling $W=\spn(e_1,\dots,e_{n-d})$, we get that
		\begin{equation*}
		I\cap W = \{x=(x',\textbf{0})\in\R^{n-d}\times\{0\}^d\,:\, a\le Lx\le b\}\neq\varnothing.
		\end{equation*}
		Let $L'$ be a rectangular $(n-d)\times d$-matrix obtained by taking the first $(n-d)$ columns of $L$. Then,
		\begin{equation*}
		I\cap W = \{x=(x',\textbf{0})\in\R^{n-d}\times\{0\}^d\,:\, a\le L'x'\le b\},
		\end{equation*}
		and so $I\cap W$ might be identified with the (non-empty) polytope 
		\begin{equation*}
		P = \{x'\in\R^{n-d}\,:\, a\le L'x'\le b\}.
		\end{equation*}
		
		Let $v'\in P$ be a vertex of $P$, i.e., a $0$-dimensional face of $P$. We claim that the point $v=(v',\textbf{0})\in\R^{n-d}\times\{0\}^d$ belongs to some $d$-dimensional face of $I$. Since it clearly lies on $W$, this will give \eqref{eq:planeskeleton}.
		
		Since $v'$ is a vertex of $P$, we get that there exist sets of indices $A,B\subset\{1,\dots,n\}$ with $A\cap B=\varnothing$ and $\#(A\cup B)=n-d$ such that
		\begin{equation*}
		a_i=L'_i\cdot v' \text{ for $i\in A$,}\quad L'_i\cdot v' = b_i \text{ for $i\in B$},
		\end{equation*}
		where $L_i'$ denotes the $i$-th row of $L'$. Note that the property $\#(A\cup B)=n-d$ follows from the fact that $v'$ should be a unique solution to the equations above (a $0$-dimensional face is a single point), and we are in $\R^{n-d}$. We do not claim that the choice of $A$ and $B$ is unique, but we do not care. Finally, observe that the remaining inequalities $a_i\le L'_i\cdot v' \le b_i,\ i\in \{1,\dots,n\}\setminus(A\cup B)$ must also be satisfied since $v\in P$.
		
		It remains to observe that $v=(v',\textbf{0})\in F(A,B)\cap W$ because $Lv = L'v'$. In particular, $v\in \d_d I\cap W$, which gives \eqref{eq:planeskeleton}.
	\end{proof}

 \begin{lemma}\label{lem:aa}
     Let $E \subset \R^n$ be compact and lower content $(d,c_1)$-regular. Let $N$ be an integer and let $\Top$ be a family of cubes as in Proposition \ref{prop:coronization} or Proposition \ref{l:corona}. For every $R\in\Top$, $Q\in\Tree(R)$ and $V\in\dG(n,d)$ we have
     \begin{equation*}
         \pi_V(Q)\subset \pi_V(\Gamma_R\cap B(x_Q, C\ell(Q)),
     \end{equation*}
     where $C>1$ is absolute.
 \end{lemma}
 \begin{proof}
     By property (b) from Proposition \ref{l:corona} or \ref{prop:coronization}, we have $Q\subset \bigcup_{I\in\cC_R}I$. Denote by $\cC(Q)$ the subfamily of $\cC_R$ consisting of cubes intersecting $Q$. 
     
     Fix $I\in \cC(Q)$ and let $y\in I\cap Q$. Note that by the property (d) of our approximation we have
     \begin{equation*}
         \ell(I)\sim \inf_{x\in I}d_R(x)\le d_R(y) = \inf_{P\in\Tree(R)}(\ell(P)+\dist(y,P)\le \ell(Q),
     \end{equation*}
     where we used that $y\in Q\in\Tree(R)$. In particular, $I\subset B(x_Q,C\ell(Q))$ for some absolute $C>1$, and so
     \begin{equation*}
         \bigcup_{I\in\cC(Q)}\d_d I\subset \Gamma_R\cap B(x_Q, C\ell(Q)).
     \end{equation*}     
     Together with \lemref{lem:skeletonsprojection} we get
     \begin{equation*}
         \pi_V(Q)\subset \bigcup_{I\in\cC(Q)}\pi_V(I)\subset \bigcup_{I\in\cC(Q)}\pi_V(\d_d I)\subset \pi_V(\Gamma_R\cap B(x_Q, C\ell(Q))).
     \end{equation*}     
 \end{proof}
 We are ready to prove Proposition \ref{lem:GammaPBP}.

	\begin{proof}[Proof of Proposition \ref{lem:GammaPBP}]
		Fix $x\in\Gamma_R$, and let $I\in\cC_R$ be such that $x\in \d_d I$ (if there are many such $I$, just choose one). We will show that for any $0<r<\diam(\Gamma)$ there exists $V_{x,r}^\Gamma\in\dG(n,d)$ such that for all $V\in B(V_{x,r}^\Gamma,\delta)$ we have
		\begin{equation}\label{eq:projgoal}
		\hd(\pi_V(\Gamma\cap B(x,r)))\gtrsim r^d.
		\end{equation}
		
		Let $0<r<C'\ell(I)$ for some constant $C'>1$ chosen below. Since $x\in \d_d I$, it is immediate to see that for \emph{any} $V\in\dG(n,d)$ we have
		\begin{equation*}
			\hd(\pi_V(\Gamma_R\cap B(x,r)))\ge \hd(\pi_V(\d_dI\cap B(x,r)))\gtrsim_{C'} r^d.
		\end{equation*}
		So the PBP property holds trivially at small scales. 
		
		Assume now that $C'\ell(I)<r<\diam(\Gamma_R)\sim\ell(R)$. By property (d) of the coronization, we have
        \begin{equation*}
            \ell(I)\sim \inf_{y\in I}\inf_{P\in\Tree(R)}\ell(P)+\dist(y,P).
        \end{equation*}
  Since $I\subset 6B_R$, this estimate lets us choose a cube $Q\in\Tree(R)$ such that $\ell(I)\sim\ell(Q)$ and $I\subset 6B_Q$. Let $P\in\Tree(R)$ be the maximal cube satisfying $Q\subset P$ and $B(x_P,C\ell(P))\subset B(x,r)$, where $C>1$ is the constant from Lemma \ref{lem:aa}. Such cube exists assuming $C'>1$ is large enough depending on $C$ and the implicit constants above. 
  
  Clearly, $\ell(P)\sim r$. Recall that there exists a ball $B(P)$ centered at $x_{P}\in E$ with $r(B(P))\sim \ell(P)$ and such that $B(P)\cap E\subset P$ (see \lemref{theorem:christ}). Set 
		\begin{equation*}
			V^{\Gamma}_{x,r} \coloneqq V^E_{B(P)},
		\end{equation*}
		where $V^E_{B(P)}$ is the $d$-plane coming from the PBP property of $E$ in the ball $B(P)$. Fix $V\in B(V^{\Gamma}_{x,r}, \delta)$. 

  By Lemma \ref{lem:aa} we have
  \begin{equation*}
      \pi_V(E\cap B(P))\subset \pi_V(P)\subset \pi_V(\Gamma_R\cap x_P,C\ell(P))\subset\pi_V(\Gamma_R\cap B(x,r)).
  \end{equation*}
  Hence, by the PBP property of $E$ we have
  \begin{equation*}
      \mathcal{H}^d(\pi_V(\Gamma_R\cap B(x,r)))\ge \mathcal{H}^d(\pi_V(E\cap B(P)))\gtrsim \ell(P)^d\sim r^d.
  \end{equation*}
  This shows \eqref{eq:projgoal}, and so $\Gamma_R$ has $d$-PBP.
	\end{proof}

	\section{Vitushkin's conjecture}\label{sec:analytic}
	In this section we use Propositions \ref{prop:coronization} and \ref{lem:GammaPBP} to prove \thmref{thm:main higher dim}. 
 
 \vspace{0.5em}
 Without loss of generality, we may assume that $\diam(E)=1$ (this follows from scaling $\Gamma_{n,d}(\lambda E) = \lambda^d\,\Gamma_{n,d}(E)$ for $\lambda>0$). By \thmref{thm:capacitybetas}, in order to show that $\Gamma_{n,d}(E)\gtrsim 1$, it suffices to construct a measure with polynomial growth supported on $E$ such that $\mu(E)\sim1$ and
	\begin{equation}\label{eq:beta est 1}
	\iint_0^{\infty} \beta_{\mu,2}(x,r)^2\,\theta_\mu(x,r)\ \frac{dr}{r}d\mu(x) \lesssim \mu(E).
	\end{equation}
	Let $\mu$ be the measure from \lemref{lem:Frostman measure}. Then it has polynomial growth, $\supp\mu\subset E$, and $\mu(E)\sim1$. We are going to show that the estimate \eqref{eq:beta est 1} is satisfied by $\mu$. In fact, we will prove a somewhat stronger estimate.
	
	\begin{propo}\label{prop:main}
		Let $E\subset\R^n$ be a compact set with $d$-PBP. If $\mu$ is the measure from \lemref{lem:Frostman measure}, then
		\begin{equation}\label{eq:main prop}
		\iint_0^{\infty} \beta_{\mu,2}(x,r)^2\ \frac{dr}{r}d\mu(x) \lesssim \mu(E),
		\end{equation}
		where the implicit constant depends only on $n,d$ and the PBP constants.
	\end{propo}

	To see that \eqref{eq:main prop} implies \eqref{eq:beta est 1}, recall that $\theta_\mu(x,r)\lesssim 1$ by the polynomial growth condition.

	Observe that since we assume that $\diam(E)=1$, we have that $\DD_0$ consists of a single cube, namely $\DD_0 = \{E\}$. A standard computation shows that \eqref{eq:main prop} is equivalent to
	\begin{equation}\label{eq:main prop sum}
	\sum_{Q\in\DD}\beta_{\mu,2}(2B_Q)^2\,\mu(Q)\lesssim \mu(E).
	\end{equation}
	For the reader's convenience, we sketch the proof of $\eqref{eq:main prop sum}\Rightarrow\eqref{eq:main prop}$. Observe that if $B_1\subset B_2$ are balls, and $r(B_1)\sim r(B_2)$, then it follows from the definition of $\beta_{\mu,2}$-numbers that
	\begin{equation*}
	\beta_{\mu,2}(B_1)\lesssim\beta_{\mu,2}(B_2).
	\end{equation*}
	Given $x\in E$ and $0<r<\infty$ let $Q\in\mathcal{D}$ be the unique cube with $x\in Q$ and $\rho\ell(Q)<r\le \ell(Q)$, so that $B(x,r)\subset 2B_Q$ and $r(2B_Q)\sim r$. Then,
	\begin{equation*}
	\beta_{\mu,2}(x,r)\lesssim\beta_{\mu,2}(2B_Q),
	\end{equation*}
	and it follows easily that
	\begin{equation*}
	\iint_0^{\infty} \beta_{\mu,2}(x,r)^2\ \frac{dr}{r}d\mu(x)\lesssim \sum_{Q\in\DD}\beta_{\mu,2}(2B_Q)^2\,\mu(Q).
	\end{equation*}
	
	Thus, our goal is to prove \eqref{eq:main prop sum}. Recalling that $\mathcal{D}_k=\{E\}$ for $k\le 0$ and $\mu(E)\sim 1$, it follows immediately that
	\begin{equation*}
	\sum_{k\le 0}\sum_{Q\in\mathcal{D}_k} \beta_{\mu,2}(2B_Q)^2\,\mu(Q)\lesssim \sum_{k\le 0}\sum_{Q\in\mathcal{D}_k} \ell(Q)^{-d}\,\mu(E)^3 \sim\sum_{k\le 0} \rho^{-kd}\,\mu(E)^3\lesssim \mu(E).
	\end{equation*}
	So when proving \eqref{eq:main prop sum}, we may concentrate on $Q\in\mathcal{D}_k$ for $k\ge 0$. Fix some large integer $N>1$. We will show that
	\begin{equation}\label{eq:main prop sum2}
	\sum_{Q\in\DD_0^N}\beta_{\mu,2}(2B_Q)^2\,\mu(Q)\lesssim \mu(E),
	\end{equation}
	with the estimate independent of $N$. Letting $N\to\infty$ gives the desired bound \eqref{eq:main prop sum}. 
	


 We apply Proposition \ref{prop:coronization} to $E$ (recall that PBP implies lower content regularity) to obtain the decomposition
 \begin{equation*}
     \DD_0^N=\bigcup_{R\in\Top}\Tree(R).
 \end{equation*}
 To get \eqref{eq:main prop sum2} we need the following estimate.
	\begin{lemma}\label{lem:key lemma}
		For any $R\in\Top$ we have
		\begin{equation}\label{eq:beta sum tree}
			\sum_{Q\in\Tr(R)}\beta_{\mu,2}(2B_Q)^2\,\mu(Q)\lesssim \theta_\mu(R)\mu(R),
		\end{equation}
		with the implicit constant depending only on $n,d$ and the PBP constants of $E$.
	\end{lemma}
    \begin{proof}
        Fix $R\in\Top$, and set $\Gamma\coloneqq \Gamma_R.$
    
    In Lemma \ref{lem:GammaADR} and Proposition \ref{lem:GammaPBP} we showed that $\Gamma$ is an Ahflors regular set satisfying $d$-PBP, with Ahlfors regularity and PBP constants depending only on $n, d,$ and the PBP constants of $E$. Thus, we get from Corollary \ref{cor:PBPbetas} that $\Gamma$ satisfies
	\begin{equation*}
	\int_{\Gamma}\int_0^{\ell(R)} \beta_{\Gamma,2}(x,r)^2\, \fr{dr}{r}d\hd(x)\lesssim \ell(R)^d.
	\end{equation*}
	In the above we also use that $\Gamma\subset 6B_R$, by \eqref{eq:Gamma6BR}.

    Recall that $\nu$ was an Ahlfors regular measure supported on $\Gamma$ which was approximating $\mu$, see \eqref{eq:nu-def}.
	Recalling \eqref{eq:nuGammarel}, the estimate above implies
	\begin{equation*}
	\int_{\Gamma}\int_0^{\ell(R)} \beta_{\nu,2}(x,r)^2\, \fr{dr}{r}d\nu(x)\lesssim \theta_\mu(R)\,\ell(R)^d.
	\end{equation*}
	Together with \eqref{eq:betasss}, this gives the desired estimate \eqref{eq:beta sum tree}.
	\end{proof}

	\begin{proof}[Proof of \eqref{eq:main prop sum2}]
		It immediately follows from \eqref{eq:beta sum tree} that
		\begin{equation*}
			\sum_{Q\in\DD_0^N}\beta_{\mu,2}(2B_Q)^2\,\mu(Q) = \sum_{R\in\Top}\sum_{Q\in\Tr(R)}\beta_{\mu,2}(2B_Q)^2\,\mu(Q)
			\lesssim \sum_{R\in\Top}\theta_\mu(R)\mu(R).
		\end{equation*}
Using the packing estimate for $\Top$ cubes \eqref{e:packing-corona-nonsigmafinite} we get
	\begin{equation*}
		\sum_{Q\in\DD_0^N}\beta_{\mu,2}(2B_Q)^2\,\mu(Q)\lesssim \sum_{R\in\Top}\theta_\mu(R)\mu(R)\lesssim
  \mu(E).
	\end{equation*}
	This finishes the proof of \eqref{eq:main prop sum2}.
 \end{proof}

	\section{Traveling Saleseman Theorem}\label{sec:TST}
	In this section we prove the Analyst's Traveling Salesman Theorem for sets with PBP, Theorem \ref{t:main}. To do this, we follow the same method layed out in the previous sections. 
 \begin{itemize}
     \item First, we apply Proposition \ref{l:corona} to the set $E$, recalling that the $d$-PBP property implies lower content $(d,c_1)$ regularity with constant $c_1$ depending on the PBP constants (Lemma \ref{l:low-reg-E}). Hence we obtain a coronization of $E$ by approximating Ahlfors $d$-regular sets $\Gamma_R$, for $R \in \Top$ (see Proposition \ref{l:corona}.
     \item Second, we apply Proposition \ref{lem:GammaPBP}, and transfer the $d$-PBP property to each of the $\Gamma_R$. 
%
\item The third step is to notice that, by the result of Orponen, each $\Gamma_R$ is uniformly $d$-rectifiable with uniform constants only depending on the PBP constant.

\item The final step are the $\beta$ number estimates. Since we want to prove the Analyst's ATST Theorem \ref{eq:atst}, the estimates needs to be done with the $\beta$-coefficients defined in terms of Hausdorff content. Hence the estimates below \textit{do not} follow from those carried out in Section \ref{sec:analytic}.
  \end{itemize}

	\begin{remark}
		We will keep careful track of the dependence of the various constants on $\delta$ (the parameters from the PBP condition). On the other hand, we will usually not keep track of the dependence on $n,d$.
	\end{remark}
	Our main estimate is the following.
	\begin{propo}\label{l:main-lemma}
		If $E \subset \R^n$ has $d$-PBP with parameters $\delta>0$, then for any $Q_0 \in \dD$, we have
		\begin{align}\label{eq:atst}
		\ell(Q_0)^d + \sum_{Q \in \dD(Q_0)} \beta_E^{d,2}(3B_Q)^2 \ell(Q)^d \lesssim_{n,d} C( \delta) \, \hd(Q_0),
		\end{align}
		where $C( \delta) \to \infty$ as $\delta\to 0$.
	\end{propo}
	Theorem \ref{t:main} follows immediately from this proposition.
	\begin{proof}[Proof of Theorem \ref{t:main}]
		First, by \cite[(A.3)]{azzam2019quantitative}, we have that 
		\begin{equation}\label{eq:WLOGp2}
		\ell(Q_0)^d+\sum_{Q \subset Q_0} \beta_E^{d,p}(C_0 B_Q)^2 \,\ell(Q)^d \sim_{p, C_0} \ell(Q_0)^d+ \sum_{Q \subset Q_0} \beta_E^{d,2}(3B_Q)^2 \ell(Q)^d
		\end{equation}
		whenever $C_0 >1$ and $1 \leq p \leq p(d)$. Together with \propref{l:main-lemma} this establishes one of the estimates from \eqref{e:main}. The converse inequality in \eqref{e:main} follows from \cite[Theorem II]{azzam2018analyst}, see also \cite[Theorem A.1(1)]{azzam2019quantitative} for a more transparent statement. 
	\end{proof}

 We will need the following lemma.
	\begin{lemma}[{\cite{azzam2018analyst}, Lemma 2.21}] \label{lemma:azzamschul}
		Let $1 \leq p < \infty$ and $E_1, E_2 \subset \R^n$. Let $x \in E_1$ and fix $r>0$. Take some $y \in E_2$ so that $B(x,r) \subset B(y, 2r)$. Assume that $E_1, E_2$ are both lower content $d$-regular with constant $c$. Then
		\begin{align*}
		\beta_{E_1}^{d,p} (x,r) \lesssim_c \beta_{E_2}^{d,p} (y, 2r) + \left( \frac{1}{r^d} \int_{E_1 \cap B(x, 2r)} \left(\frac{\dist (y, E_2)}{r} \right)^p \, d \hdc(y)\right)^{\frac{1}{p}}.
		\end{align*}
	\end{lemma}

	We now focus on proving Proposition \ref{l:main-lemma}. Fix $E\subset\R^n$ with PBP and let $Q_0\in\DD_0$. Recall that for $N\ge 0$ we defined the truncated dyadic lattice as $\dD_0^N=\bigcup_{k=0}^{N}\{Q\in \dD_{k}\, : \,Q\subseteq Q_0\}$. Observe that to prove \eqref{eq:atst} it suffices to show
	\begin{equation*}
	\sum_{Q \in \dD_0^N} \beta_E^{d,2}(3B_Q)^2 \ell(Q)^d \lesssim_{n,d} C( \delta) \, \hd(Q_0),
	\end{equation*}
	with bounds independent of $N\ge 0$. The estimate $\ell(Q_0)^d\lesssim \delta^{-1} \hd(Q_0)$ follows immediately from lower content regularity of $E$. 

	Consider the coronization from \lemref{l:corona} applied with $A=6$ and sufficiently small $\tau$, to be fixed later. Let $R\in \Top(N)$. We start off by applying Lemma \ref{lemma:azzamschul} with $E_1=E$ and $E_2 = \Gamma_R$. For $Q \in \dD$, recall that $x_Q$ denotes the center of $Q$. By the definition of $ \Tree(R)$, we see that if $Q \in \Tree(R)$, then there exists a dyadic cube $I \in \cC_R$ with $x_Q\subset I$ (by Lemma \ref{l:corona}(2.b)). By \eqref{e:whitney-like}, $\ell(I) \lesssim \tau \ell(Q)$. Hence, we may find a point
	\begin{align}\label{e:XQ}
	y_Q \in \Gamma_R \quad \mbox{ such that } \quad |x_Q - y_Q| \leq \ell(Q),
	\end{align}
	and we obtain that
	\begin{align}\label{e:form500}
	3B_Q = B(x_Q, 3\ell(Q)) \subset 6B(y_Q, 6\ell(Q))=: 6B'_Q.
	\end{align}
	This implies that for each cube $Q \in  \Tree(R)$ the hypotheses of Lemma \ref{lemma:azzamschul} are satisfied (with $E_1 = E,\ E_2 = \Gamma_R,\ 3B_Q,$ and $6B_Q'$). Thus,
	\begin{multline}\label{eq:boh}
	\sum_{Q \in \Tree(R)} \beta_{E}^{d, 2} (3 B_Q)^2 \ell(Q)^d  \lesssim \sum_{Q \in \Tree(R)} \beta_{\Gamma_R}^{d,2} (6 B_Q')^2 \, \ell(Q)^d \\
	 + \sum_{Q \in  \Tree(R)} \ps{\frac{1}{\ell(Q)^d} \int_{6 B_Q\cap E} \ps{\frac{\dist(y, \Gamma_R)}{\ell(Q)}}^2 \, d \hdc(y) }\, \ell(Q)^d 
	 \eqqcolon I_1 + I_2.
	\end{multline}
	
	We estimate $I_1$. We denote by $\dD^{\Gamma_R}$ a family of Christ-David cubes for $\Gamma_R$ obtained by applying Lemma \ref{theorem:christ} to $\Gamma_R$. Using \eqref{e:XQ} it is immediate to see that for each $Q\in\Tree(R)$ there exists $P\in\dD^{\Gamma_R}$ with $\ell(P)\sim\ell(R)$ and $6B_Q'\subset 3B_P$, and that any $P\in \dD^{\Gamma_R}$ corresponds to at most a bounded number of $Q\in\Tree(R)$. Thus, we get
	\begin{align*}
	I_1 \lesssim \sum_{P \in \dD^{\Gamma_R}} \beta_{\Gamma_R}^{d,2}(3 B_P)^2 \ell(P)^d.
	\end{align*}
	
	Recall that $\Gamma_R$ is Ahlfors $d$-regular (Lemma \ref{l:corona}) with constant $C$ depending on $A, \tau, d$ and $c_1$ (and thus, by Lemma \ref{l:low-reg-E}, on $\delta$), and that it has $d$-PBP  with parameter $\sim\delta$ by Proposition \ref{lem:GammaPBP}. It follows from Corollary \ref{cor:PBPbetas} that for $\mu_R=\mathcal{H}^d|_{\Gamma_R}$ we have	
	\begin{equation*}
	\sum_{P \in \dD^{\Gamma_R}} \beta_{\mu_R,2}(3 B_P)^2 \ell(P)^d\lesssim_\tau C(\delta)\ell(R)^d.
	\end{equation*}
	Since $\Gamma_R$ is Ahlfors regular, we have $\beta_{\mu_R,2}(x,r)\sim \beta_{\Gamma_R}^{d,2}(x,r)$ (see \S 1.2 in \cite{azzam2018analyst}), and so the two estimates above give
	\begin{equation}
	I_1 \lesssim_\tau C(\delta) \ell(R)^d,
	\end{equation}
	where $C(\delta) \to \infty$ as $\delta\to 0$. 
	\newcommand{\wstop}{\wt{\Stop}}
	
	We move on to estimating $I_2$. Recall from Proposition \ref{l:corona}, \eqref{e:contains}, that
	\begin{equation}\label{e:form1000}
		 6 B_R \cap E = A B_R \cap E \subset \bigcup_{I\in\cC_R} I.
	\end{equation} 
	For each $I \in \cC_R$, set 
	\begin{equation}\label{e:form1002}
		\cF(I):= \{S \in \dD \ :\ S \cap I\neq \varnothing \mbox{ and } \rho \ell(I) < \ell(S) \leq \ell(I)\}.
	\end{equation}
	Then define $\cF(R)= \cup_{I \in \cC_R} \cF(I)$ and set $\wstop(R)$ to be the subfamily of maximal cubes in $\cF(R)$ (by maximality, this family is disjoint and covers $\cF(R)$).
	Note that
	\begin{equation}\label{e:form1001}
		6B_R \cap E \subset \bigcup_{S \in \wstop(R)} S. 
	\end{equation}
	Indeed, if $x \in 6B_R \cap E$, then by \eqref{e:form1000} there is an $I \in \cC_R$ so that $x \in I$. But $x$ is also contained in some Christ-David cube $S$ such that $\rho\ell(I) \leq \ell(S) \leq \ell(I)$, and thus \eqref{e:form1001} holds. Note that \eqref{e:form1001} trivially implies that for any $Q\in\Tree(R)$ we have $6B_Q \cap E \subset \bigcup_{S \in \wstop(R)} S$, since $6B_Q \cap E \subset 6B_R\cap E$.
	
	For any $Q\in\Tree(R)$ and $x \in 6B_Q \cap E$, we claim that there exists a unique cube $S \in \wstop(R)$ containing $x$, and it satisfies
	\begin{equation}\label{eq:blaaaa}
		\dist(x, \Gamma_R) \lesssim \ell(S).
	\end{equation}
	Indeed, by \eqref{e:form1001}, $x \in S$ for some $S \in \wstop(R)$, and uniqueness follows from the maximality of $\wstop(R)$. By definition, there is an $I \in \cC_R$ so that $S \cap I \neq \varnothing$ and $\ell(I)\sim \ell(S)$. But then
	\begin{equation*}
		\dist(x, \Gamma_R) \leq \diam(S) + \diam(I) \lesssim \ell(S).
	\end{equation*}
	
	We can now estimate
	\begin{multline}
	I_2= \sum_{Q \in \Tree(R)} \int_{6B_Q \cap E} \left( \frac{\dist(y, \Gamma_R)}{\ell(Q)} \right)^2	\, d \hdc(y) \\
	\stackrel{\eqref{e:form1001},\eqref{eq:blaaaa}}{\lesssim} \sum_{Q \in \Tree(R)} \sum_{\substack{S \in \wstop(R)\\ S \cap 6B_Q \neq \varnothing}} \int_S \frac{\ell(S)^{2}}{\ell(Q)^{2}} \, d \hdc
	 \lesssim \sum_{S \in \wstop(R)} \ell(S)^{d+2} \sum_{\substack{Q \in \Tree(R) \\ 6B_Q \cap S \neq \varnothing}} \frac{1}{\ell(Q)^2},\label{e:form1003}
		\end{multline}
	where we may switch the order of summation since both sums are finite.
	
	Fix $S \in \wstop(R)$. We claim that for any $k\ge 0$ the number of cubes $Q \in \Tree(R)\cap\DD_k$ with $6B_Q\cap S\neq\varnothing$ is bounded by some dimensional constant. Indeed, this is clear as soon as $\ell(S)\le \ell(Q)=\rho^k$. On the other hand, if $\ell(S)> \rho^k$, then there is no $Q \in \Tree(R)\cap\DD_k$ with $6B_Q\cap S\neq\varnothing$. To see this, observe that $6B_Q\cap S\neq\varnothing$ implies $\dist(S, Q) \lesssim \ell(Q)$. Also, if $I \in \cC_R$ is such that $S \in \cF(I)$, then $\dist(I, Q) \lesssim \ell(I) + \dist(S, Q) \lesssim \ell(S) + \ell(Q)$. Hence
	\begin{equation*}
		\ell(S) \stackrel{\eqref{e:form1002}} {\sim} \ell(I)  \stackrel{\eqref{e:whitney-like}}{\sim} \tau \inf_{x\in I} d_R(x) \stackrel{\eqref{e:d_F}}{\lesssim}  \tau (\ell(Q) + \dist(I, Q)) \lesssim \tau \ell(Q) + \tau \ell(S).
	\end{equation*}
	Choosing $\tau>0$ sufficiently small we get that $\ell(S) \le \ell(Q)$. 
	
	The observation above allows us to estimate the interior sum in \eqref{e:form1003} as a geometric sum:
	\begin{equation*}
		I_2\lesssim \sum_{S \in \wstop(R)} \ell(S)^{d+2} \sum_{\substack{Q \in \Tree(R) \\ 6B_Q \cap S \neq \varnothing}} \frac{1}{\ell(Q)^2}\lesssim \sum_{S \in \wstop(R)} \ell(S)^{d}.
	\end{equation*}
	
	Now, given $S \in \wstop(R)$, we claim that
	\begin{equation}\label{e:form1005}
		\# \{I \in \cC_R \ :\ S \in \cF(I)\} \sim 1.
	\end{equation}
	This follows immediately from the fact that for dyadic cubes $I,J$ as above we have $\ell(I)=\ell(J)$ and $\dist(I,J)\lesssim \ell(I)$. Hence, for each $S \in \wstop(R)$ there is a bounded number of cubes in $\cC_R$ of sidelength $\sim \ell(S)$ which intersect it.
	Finally, we may conclude that
	\begin{align*}
		I_2 & \lesssim \sum_{S \in \wstop(R)} \ell(S)^d \sim  \sum_{S \in \wstop(R)} \sum_{\substack{I \in \cC_R \\ S \in \cF(I)}} \ell(I)^d\\
		& = \sum_{I \in \cC_R} \ell(I)^d \#\{ S \in \wstop(R) \ : \ S \in \cF(I)\} \lesssim \sum_{I \in \cC_R} \ell(I)^d,
	\end{align*}
	and hence, using the Ahlfors $d$-regularity of $\Gamma_R$, 
	\begin{equation}
	I_2 \lesssim \sum_{I \in \cC_R} \ell(I)^d \sim \sum_{I \in \cC_R} \hd(\partial_d I) \sim \hd(\Gamma_R) \lesssim_{\tau,\delta} \ell(R)^d.
	\end{equation}

	Putting the estimates for $I_1$ and $I_2$ together with \eqref{eq:boh} gives
		\begin{align}\label{e:TreeEst}
		\sum_{Q \in \Tree(R)} \beta_E^{d,2}(3B_Q)^2 \ell(Q)^d \lesssim_\tau C(\delta) \ell(R)^d,
		\end{align}
		where $C(\delta) \to \infty$ as $\delta\to 0$.
	We conclude that
	\begin{multline}
	\sum_{Q \in \dD_0^N} \beta_{E}^{d,2}(3B_Q)^2 \ell(Q)^d \stackrel{{\rm Lemma}\, \ref{l:corona}}{\lesssim} \sum_{R \in \Top(N)} \sum_{Q\in \Tree(R)} \beta_{E}^{d,2}(3B_Q)^2 \ell(Q)^d  \\ \stackrel{\eqref{e:TreeEst}}{\lesssim_\tau}  C(\delta) \sum_{R \in \Top(N)} \ell(R)^d  \stackrel{\eqref{e:ADR-packing}}{\lesssim}  C(\delta)\eta^{-1} \hd(Q_0) \sim C(\delta) \hd(Q_0). \label{e:upper-bound}
	\end{multline}
	This finishes the proof of Proposition \ref{l:main-lemma}. 
	
	\section{Dimension of wiggly sets}\label{sec:wiggly}
	In this section we establish our estimate for the dimension of wiggly sets with PBP, \thmref{t:corollary}.
	
	Through this section we will work with David-Christ cubes (\lemref{theorem:christ}) associated to different sets. To avoid confusion we will write $\DD^E$ to denote the system of cubes associated to a given set $E\subset\R^n$. 
	
	Recall that $p(d)=\tfrac{2d}{d-2}$ for $d>2$ and $p(d)=\infty$ if $d\le 2$. Theorem \ref{t:corollary} will follow as an easy corollary of the following proposition. 
	\begin{propo}\label{prop-content-beta-wiggly}
		Let $E \subset \R^n$ be a closed set with $d$-PBP with constant $\delta>0$ and which is uniformly wiggly of dimension $d$, constant $\beta_0$, and with respect to the Azzam-Schul $\beta_E^{d,p}$-numbers, where $1 \leq p < p(d)$. That is, $E$ satisfies 
		\begin{equation*}
			\beta_{E}^{d,p}(x,r) \geq \beta_0 \mbox{ for all } x\in E,\ 0<r<\diam(E).
		\end{equation*}
		Then
		\begin{equation}\label{e:wiggly-dim2}
		\mathrm{dim}_{H}(E) \geq d + c \beta_0^2
		\end{equation}
		with $c$ depending on $\delta,p,n,d$.
	\end{propo}
	\begin{proof}[Proof of Theorem \ref{t:corollary} using Proposition \ref{prop-content-beta-wiggly}]
		Let $E$ be as in the statement of the theorem. In particular, $\beta_{E, \infty}^d(B_Q) > \beta_0$. By \cite[Lemma 2.12, Lemma 2.13]{azzam2018analyst}, we have that
		\begin{equation*}
		\beta_{E,\infty}^d (B_Q) \lesssim_{d} \beta_{E}^{d,p}(B_Q)^{\frac{1}{d+1}}.
		\end{equation*}
		Hence $E$ is uniformly wiggly of dimension $d$ with respect to $\beta_{E}^{d,p}$ with constant $\beta_0'\sim\beta_0^{d+1}$, and \eqref{e:wiggly-dim} follows immediately from \eqref{e:wiggly-dim2} with $\beta_0$ replaced by $\beta_0'$.
	\end{proof}

	We begin the proof of Proposition \ref{prop-content-beta-wiggly}. Let $E \subset \R^n$ be a closed set with $d$-PBP, with constant $\delta>0$. Without loss of generality we assume $E$ is compact, $\diam(E)=1/2$, so that $\mathcal{D}_0=\{E\}$, and $E\subset [0,1)^n$.
	
	Recall that $\Delta$ is the standard dyadic lattice on $\R^n$, and $\Delta_k = \{I\in\Delta\ :\ \ell(I)=2^{-k}\}$. Given a dyadic cube $I\in\Delta$ and $A\ge 1$, we will denote by $AI$ the cube with the came center as $I$ and sidelength $A\ell(I)$.

	Recall that $\rho=1/1000$ and that for $Q\in\DD^E_m$ we have $\ell(Q)=5\rho^m$, see \lemref{theorem:christ}. Given $k\ge0$ let $j(k)$ be the unique integer such that $2^{-j(k)}\le 5\rho^k < 2^{-j(k)+1}$. We set
	\begin{align} \label{e:Delta-ck}
	\Delta_{k}(E) :=  \ck{ I \in \Delta_{j(k)} \, :\, I \cap E \neq \varnothing}.
	\end{align}
	We defined $j(k)$ in such a way that if $I\in\Delta_k(E)$ and $Q\in \dD^E_k$ then $\ell(I)\le \ell(Q)\le 2\ell(I)$. In particular, $\ell(I)\sim \rho^k$.
	
	Set
	\begin{equation}\label{eq:Ekdef}
	E_k := \bigcup_{I \in \Delta_{k}(E)} \partial_d I,
	\end{equation}
	where $\partial_d I$ is the $d$-dimensional dyadic skeleton of $I$.
 \begin{lemma}
     For every $k\ge 0$ the set $E_k$ has $d$-PBP (with constant $\delta'\sim \delta$).
 \end{lemma}
 \begin{proof}
     The proof is similar to that of \propref{lem:GammaPBP}, although here the situation is much simpler. Let $x\in E_k$ and $0<r<\diam(E)$. Let $I\in \Delta_k(E)$ be such that $x\in \d_d I$. If $0<r\lesssim 2^{-j(k)}\sim\diam(I)$, then for all $V\in \dG(n,d)$
 \begin{equation*}
     \mathcal{H}^d(\pi_V(E_k\cap B(x,r)))\ge \mathcal{H}^d(\pi_V(\d_d I\cap B(x,r)))\gtrsim r^d.
 \end{equation*}
On the other hand, if $r\gtrsim 2^{-j(k)}$ then there exists $y\in E$ with $B(y,r/C)\subset B(x,r)$, and such that all $J\in\Delta_k(E)$ with $J\cap B(y,r/C)=\varnothing$ satisfy $J\subset B(x,r)$. Denote the family of such $J$ by $\Delta_k(x,r)$. Then, by Lemma \ref{lem:skeletonsprojection}
 \begin{multline*}
     \mathcal{H}^d(\pi_V(E_k\cap B(x,r)))\ge \mathcal{H}^d\bigg(\pi_V\bigg(\bigcup_{J\in\Delta_k(x,r)}\d_dJ\bigg)\bigg)\\
     \ge \mathcal{H}^d\bigg(\pi_V\bigg(\bigcup_{J\in\Delta_k(x,r)}J\bigg)\bigg)\ge \mathcal{H}^d(\pi_V(E\cap B(y,C/r))).
 \end{multline*}
 So for $V\in\dG(n,d)$ such that $\mathcal{H}^d(\pi_V(E\cap B(y,C/r)))\gtrsim r^d$ we also have $\mathcal{H}^d(\pi_V(E_k\cap B(x,r)))\gtrsim r^d$. This gives the PBP property for $E_k$.
 \end{proof}

	\begin{lemma} 
		Let $k\ge j \ge 1$ be integers, and let $I\in\Delta_j(E)$. We have
		\begin{equation} \label{e:BJ-a}
		\sum_{m=j}^{k} \sum_{\substack{Q \in \dD^E_m \\ Q \cap 3I \neq \varnothing}} \beta_{E}^{d,p}(B_Q)^2 \ell(Q)^d  \leq C\hd(E_k\cap AI),
		\end{equation}
		where $A=24$, and $C$ depends on $\delta, n, d$ and $p$.
	\end{lemma}
	\begin{proof}
		Without loss of generality we may assume that $p=2$ (see \cite[(A.3)]{azzam2019quantitative}). Consider a cube $Q \in \dD^E_m$ with $j\le m\le k$ and $Q\cap 3I\neq\varnothing$. Let $P\in\DD^{E_k}$ be a cube with $\ell(P)=\ell(Q)$ and such that $B_Q\cap B_P\neq\varnothing$ (such a cube exists because $\dist(x_Q,E_k)\le 2^{-j(k)}\le\ell(Q)$). In particular, we have $2B_Q\subset 4B_{P}$ and we may apply Lemma \ref{lemma:azzamschul} with $E_1=E$ and $E_2 = E_k$ to obtain
		\begin{align*}
		\beta_{E}^{d,2}(B_Q)\lesssim \beta_{E}^{d,2}(2B_Q) \lesssim \beta^{d,2}_{E_k}(4B_{P}) + \ps{ \frac{1}{\ell(Q)^d} \int_{4 B_{Q}\cap E} \ps{\frac{\dist(y, E_k)}{\ell(Q)}}^2 \, d \mathcal{H}_\infty^d(y)}^{\frac{1}{2}}.
		\end{align*}
		Note that we have $B_P\cap 11I\neq\varnothing$ because $\ell(Q)\le 2\ell(I),\ Q\cap 3I\neq\varnothing$ and $B_P\cap B_Q\neq\varnothing$.
		
		Multiplying by $\ell(Q)^d$ and summing over $Q \in \dD^E_m$ with $j\le m\le k$ and $Q\cap 3I\neq\varnothing$ yields
		\begin{multline}\label{e:form101}
		\sum_{m=j}^{k} \sum_{\substack{Q \in \dD^E_m \\ Q \cap 3I \neq \varnothing}} \beta_{E}^{d,2} (B_Q)^2\ell(P)^d 
		\lesssim \sum_{m=j}^{k}\sum_{\substack{P \in \dD_m^{E_k}\\ B_P \cap 12I \neq \varnothing}} \beta^{d,2}_{E_k}(4 B_{P})^2\ell(P)^d \\
		+ \sum_{m=j}^{k} \sum_{\substack{Q \in \dD^E_m \\ Q \cap 3I \neq \varnothing}} \ps{ \frac{1}{\ell(Q)^d} \int_{4 B_Q\cap E} \ps{\frac{ \dist(y, E_k)}{\ell(Q)}}^{2} \, \hdc(y) }\ell(P)^d \eqqcolon S_1+S_2.
		\end{multline}
		
		We begin by estimating $S_1$. Let 
		\begin{equation*}
		\mathcal{R}=\{ R\in \dD^{E_k}_j\ :\ B_R\cap 11I\neq\varnothing\}.
		\end{equation*}
		Note that for each cube $P$ appearing in $S_1$ there exists a unique $R\in\mathcal{R}$ such that $P\subset R$.
		Observe also that all the cubes $R\in\mathcal{R}$ are contained in $19I$. Applying these observations together with Theorem \ref{t:main} to the set $E_k$ and cubes $R\in\mathcal{R}$ yields
		\begin{equation}\label{eq:useATST}
		S_1\lesssim \sum_{R\in\mathcal{R}}\sum_{P\in\dD(R)} \beta^{d,2}_{E_k}(4 B_{P})^2\ell(P)^d  \lesssim_\delta \sum_{R\in\mathcal{R}} \mathcal{H}^d(R) \le \hd(E_k \cap 19I). 
		\end{equation}
		
		We move on to estimating $S_2$. This estimate is similar to that of "$I_2$" in \eqref{eq:boh}. First, we use the definition of $E_k$ and the subadditivity of Hausdorff content to get
		\begin{equation*}
			\int_{4B_Q\cap E} \left(\frac{\dist(y, E_k)}{\ell(P)} \right)^2 \, d \hdc(y) \lesssim \sum_{\substack{J \in \Delta_k(E) \\ J \cap 4B_Q \neq \varnothing}} \frac{\ell(J)^{2}}{\ell(Q)^2} \hdc(J \cap 4B_Q \cap E) \lesssim\sum_{\substack{J \in \Delta_k(E) \\ J \cap 4B_Q \neq \varnothing}} \frac{\ell(J)^{d+2}}{\ell(Q)^2}.
		\end{equation*} 
		It is easy to check that if $Q\in\dD^E_m$ with $j\le m\le k$ and $Q\cap 3I\neq\varnothing$, then for $J$ as above we have $J\subset 24I$. It follows that
		\begin{multline*}
			S_2\le \sum_{m=j}^{k} \sum_{\substack{Q \in \dD^E_m \\ Q \cap 3I \neq \varnothing}} \sum_{\substack{J \in \Delta_k(E) \\ J \cap 4B_Q \neq \varnothing}} \frac{\ell(J)^{d+2} }{\ell(Q)^2}  \lesssim \sum_{\substack{J \in \Delta_k(E) \\ J \subset 24I}} \ell(J)^{d+2} \sum_{\substack{Q \in \dD^{E} \\ \ell(Q)\ge \rho^{k}: \,  4B_Q \cap J \neq \varnothing}} \frac{1}{\ell(Q)^2} \\
			 \lesssim \sum_{\substack{J \in \Delta_k(E)  \\ J \subset 24I}} \ell(J)^{d}.
		\end{multline*}
		We also have
		\begin{equation*}
		\sum_{\substack{J \in \Delta_k(E) \\ J \subset 24I}} \ell(J)^d\sim \sum_{\substack{J \in \Delta_k (E)\\ J \subset 24I}} \mathcal{H}^d(\partial_d J) \sim \dH^d(E_k \cap 24I).
		\end{equation*}
		Hence, $S_2\lesssim \dH^d(E_k \cap 24I)$. Putting this together with \eqref{e:form101} and \eqref{eq:useATST} gives \eqref{e:BJ-a}.

	\end{proof}
	
	\begin{lemma}
		Let $k\ge j >0$ be integers and $I\in\Delta_j(E)$. We have
		\begin{equation}\label{e:form102}
		\sum_{m=j}^k \sum_{Q \in \dD^E_m \atop Q \cap 3I \neq \varnothing} \beta_{E}^{d,p}(B_Q)^2 \ell(Q)^d  \geq (k-j) C \delta  \beta_0^2 \rho^{jd},
		\end{equation}
		where $C$ is a constant depending only on $n, d$, but not on $\delta$.
	\end{lemma}

	\begin{proof}
		Recall first that $E$ is lower $(d, c_1)$-regular (with $c_1 \sim \delta$, see \lemref{l:low-reg-E}), so that in particular $\hdc(E\cap 3I)\gtrsim \delta \rho^{dj}$. Let $m\ge j$. Since the balls $\{B_Q\ :\ Q\in \DD^E_m,\, Q\cap 3I\neq\varnothing\}$ are a cover of $E\cap 3I$, we get
		\begin{equation*}
		\sum_{Q \in \dD^E_m \atop Q \cap 3I\neq\varnothing} \ell(Q)^d \gtrsim \delta \rho^{dj}.
		\end{equation*}		
		Recalling that $E$ is uniformly wiggly we arrive at
		\begin{equation*}
		\sum_{Q \in \dD^E_m \atop Q \cap 3I \neq \varnothing} \beta_{E}^{d,p}(B_Q)^2 \ell(Q)^d 
		\geq  \beta_0^2 \sum_{Q \in \dD^E_m \atop Q \cap 3I\neq\varnothing} \ell(Q)^d
		 \gtrsim \delta \beta_0^2 \rho^{dj}.
		\end{equation*}
		Summing over $j\le m\le k$ gives \eqref{e:form102}. 
	\end{proof}
	
	Recall that our aim was to prove that
	\begin{align*}
	{\rm dim}_H(E) \ge d+ c\beta_0^2.
	\end{align*}
	This estimate follows immediately from Theorem 8.8 in \cite{mattila} and the next lemma.
	\begin{lemma}
		There exists a non-zero measure $\mu$ with $\supp\mu\subset E$ and such that
		\begin{equation*}
		\mu(B(x,r))\lesssim r^{d+ c\beta_0^2}\quad\text{for $x\in E$ and $r>0$},
		\end{equation*}
		where $c$ depends on $\delta, p, n, d$.
	\end{lemma} 
\begin{proof}
	The construction of $\mu$ will is similar to the proof of Frostman's lemma, as in \cite[Theorem 8.8]{mattila}. First, we need to define a tree-like collection of dyadic cubes intersecting $E$ which we will use in our construction.
	
	Let $k> j \ge 1$ be integers, and let $I\in\Delta_j(E)$. Using \eqref{e:BJ-a} and \eqref{e:form102} we get that
	\begin{align}\label{eq:blabbla}
	\hd(E_k\cap AI) \gtrsim_\delta (k-j) \, \beta_0^2 \, \rho^{dj}.
	\end{align}	
	Let 
	\begin{equation*}
	D_k(I)\coloneqq\Delta_{k}(E\cap AI) = \{J\in\Delta_{k}\ :\ J\cap AI\cap E\neq\varnothing  \}.
	\end{equation*}
	It follows easily from the definition of $E_k$ \eqref{eq:Ekdef} that 
	\begin{equation*}
	E_k\cap AI \subset \bigcup_{J\in D_k(I)}\d_d J.
	\end{equation*}
	Hence,
	\begin{align*}
	\hd(E_k \cap AI) \lesssim \rho^{dk}\, \#D_k(I).
	\end{align*}
	Together with \eqref{eq:blabbla} this gives
	\begin{align} \label{e:CardA}
	\#D_k(I) \gtrsim_\delta (k-j) \, \beta_0^2\, \rho^{d(j-k)}.
	\end{align}
	
	Observe that $\Delta(E)=\bigcup_{k\ge 0}\Delta_k(E)$ can be endowed with a natural tree structure, and this structure is used in the usual proof of the Frostman's lemma. We would like to define a similar structure for the collections $D_k(I)$ for $I\in \Delta_j(E), k\ge j\ge 0$, but we need to take extra care due to possible overlaps between collections $D_k(I), D_k(J)$ when $I,J\in\Delta_j(E)$. 
	
	To overcome this technicality, for any integers $j,k$ with $k\ge j\ge 1$ and $I\in\Delta_j(E)$ let $D'_k(I)\subset D_k(I)$ be a maximal subfamily of $D_k(I)$ such that for any $J_1, J_2\in D'_k(I)$ we have $A'J_1\cap A'J_2=\varnothing,$ where $A'=A+1$. Note that since for each $J\in D'_k(I)$ we have $J\subset AI$, we get
	\begin{equation*}
	\bigcup_{J\in D'_k(I)}A'J\subset A'I.
	\end{equation*} 
	Moreover, since all the cubes in $D_k(I)$ have equal sidelength, the cubes $\{A'J\}_{J\in D_k(I)}$ have bounded overlap, and so it follows that 
	\begin{equation*}
	\#D'_k(I)\gtrsim \#D_k(I).
	\end{equation*}	
	Recalling \eqref{e:CardA}, we get that
	\begin{equation}\label{e:CardA-c}
	\#D'_k(I) \ge C_1 (k-j) \, \beta_0^2\, \rho^{d(j-k)}.
	\end{equation}
	for some $0<C_1<1$ depending on $\delta,p,n,d$. 
	
	Let
	\begin{align*}
	\kappa\coloneqq \lceil 10^6\, C_1^{-1} \beta_0^{-2}\rceil,
	\end{align*}
	and observe that for $c \coloneqq C_1/10^6$ we have
	\begin{equation*}
	\kappa\, C_1 \beta_0^2 \geq 10^6=\rho^{-2}\ge \rho^{-\kappa\,  c \beta_0^2},
	\end{equation*}
	where in the last inequality we use also the fact that $C_1,\beta_0\in (0,1)$.	
	Hence, if $k$ is such that $k =j+ \kappa$, the inequality \eqref{e:CardA-c} gives
	\begin{align} \label{e:CardA-b}
	\#D'_{j+\kappa}(I) \geq \rho^{-\kappa(d+ c\beta_0^2)}.
	\end{align}
	
	We are ready to define the tree structure we will use to construct $\mu$. Set
	\begin{align*}
	\dS_0 = \Delta_{0}(E)=\{[0,1)^n\},
	\end{align*}
	where we used our assumption $E\subset [0,1)^n$. Let $j\ge 0$ and assume that $\dS_j$ has already been defined, that $\dS_j\subset \Delta_{j\kappa}(E)$, and that for every $I,J\in\dS_{j},\ I\neq J,$ we have $A'I\cap A'J=\varnothing$. For each $I\in\dS_j$ we set
	\begin{equation*}
	\dS(I) = D'_{(j+1)\kappa}(I),
	\end{equation*}
	and
	\begin{equation*}
	\dS_{j+1} = \bigcup_{I\in\dS_j}\dS(I).
	\end{equation*}
	
	Note that for every $J\in\dS_{j+1}$ there is a unique ``parent'' $I\in\dS_j$ such that $J\subset AI$ (and $A'J\subset A'I$). Moreover, for every $I,J\in\dS_{j+1},\ I\neq J, $ we have $A'I\cap A'J=\varnothing$ (either $I$ and $J$ have distinct parents, or they have the same parent $I'$ in which case we use the definition of $D'_{(j+1)\kappa}(I')$). Finally, observe that if $s\coloneqq d+c\beta_0^2$, then by \eqref{e:CardA-b}
	\begin{equation}\label{eq:cardSI}
	\#\dS_j\ge \rho^{-j\kappa s}\quad\text{and}\quad \#\dS(I)\ge \rho^{-\kappa s} \text{ for each $I\in\dS_j$}.
	\end{equation}

	We are ready to construct the measure $\mu$. It is obtained as a weak limit of measures $\mu_j,\, j\ge0$. The definition of $\mu_j$ follows the usual ``bottom-to-top'' construction of the Frostman measure, as in Theorem 8.8 of \cite{mattila}.
	
	 First, we set
	\begin{equation*}
	\mu_j^j = \rho^{j\kappa s}\sum_{I\in \dS_j}\fr{\mathcal{L}^n|_{A'I}}{\mathcal{L}^n(A'I)}.
	\end{equation*}
	Assume that the measure $\mu_j^i$ with $1\le i\le j$ has already been defined, and that for each $I\in\dS_i$ we have $\mu_j^i(A'I)= \rho^{i\kappa s}$. We define $\mu_j^{i-1}$ by modifying $\mu_j^i$ at the level of $\dS_{i-1}$: for each $I\in \dS_{i-1}$ we set
	\begin{equation*}
	\mu_j^{i-1}|_{A'I} \coloneqq \rho^{(i-1)\kappa s}  \mu_j^{i}(A'I)^{-1} \mu_j^{i}|_{A'I} .
	\end{equation*}
	Finally, we define $\mu_j=\mu_j^0$. Clearly, $\mu_j([0,1)^n)=1$.
	
	Observe that the quantity used in the definition of $\mu_j^{i-1}|_{A'I}$ satisfies
	\begin{equation*}
	\rho^{(i-1)\kappa s}  \mu_j^{i}(A'I)^{-1} = \rho^{\kappa s}  (\#\dS(I))^{-1} \overset{\eqref{eq:cardSI}}{\le }1.
	\end{equation*}
	It follows that for every $i\le k\le j$ and $I\in\dS_k$ we have
$
	\mu_j^i(A'I)\le \ell(I)^s.
$
	In particular,
	\begin{equation}\label{eq:Frcond}
	\mu_j(A'I)\le \ell(I)^s\quad \text{for all $I\in\dS_k,\, 0\le k\le j$.}
	\end{equation} 
	
	Let $\mu_{j_k}$ be a subsequence of $\mu_j$ converging weakly, and define $\mu$ to be the weak limit of $\mu_{j_k}$. Since $\mu_j(\R^n)=1$ for all $j$, we also have $\mu(\R^n)=1$. Using \eqref{eq:Frcond}, a standard argument gives
	\begin{equation*}
	\mu(B(x,r))\lesssim r^s\quad\text{for $x\in\R^n,\,r>0$},
	\end{equation*}
	see p.114 in \cite{mattila} for details. Finally, $\supp\mu\subset E$ because
	\begin{equation*}
	\supp\mu_j\subset \bigcup_{I\in\Delta_{j\kappa}(E)}I.
	\end{equation*}	
	\end{proof}
	This completes the proof of Proposition \ref{prop-content-beta-wiggly}. 	
	
	\section{Denjoy's conjecture}\label{s:cap}
	In this section we prove \thmref{t:corollary-2b}. 
	Let $\Sigma \subset \R^{d+1}$ be a compact set with $d$-PBP with parameter $\delta>0$, and let $E \subset \Sigma$ be compact. Let $\DD$ denote the Christ-David cubes on $\Sigma$, as in \lemref{theorem:christ}. We are going to use Theorem \ref{thm:capacitybetas} to get a lower bound on $\Gamma_{n,d}(E)$.
	
	By Frostman's lemma (see \cite[Lemma 1.23]{tolsa2014analytic}), there exists a measure $\mu$ supported on $E$ such that $\mu(B(x,r)) \leq r^d$ for all $x\in \R^n,\, r>0$, and $\mu(E) \sim \dH_\infty^d(E)$. In fact, it follows immediately from the proof of \cite[Lemma 1.23]{tolsa2014analytic} that one even has $\mu(B(x,r)) \leq \hdc(E\cap B(x,r))$. 
	
	Recall that $\Sigma$ is lower content regular with constant $\sim\delta$ (see Lemma \ref{l:low-reg-E}). Using the fact that for any ball $x\in E$ and $0<r<\infty$ we have
	\begin{equation*}
	\mu(B(x,r)) \leq \hdc(E\cap B(x,r))\le \hdc(\Sigma \cap B(x,r))
	\end{equation*}
	it follows from the definitions of $\beta_{\mu,2}$ and $\beta_\Sigma^{d,2}$ that 
	\begin{align}\label{eq:betamubetasigma}
	\beta_{\mu,2}(x,r) \lesssim \beta_\Sigma^{d,2}(x,r).
	\end{align}
 Indeed, given a plane $L\in\mathcal{A}(n,d)$ infimising $\beta_E^{d,2}(x,r)$ set 
 \begin{equation*}
 E_t=E_t(x,r):=\{y \in B(x,r) \cap E \, : \, \dist(y, L)^2 > t r^2 \}.
 \end{equation*}
 Let $\dC$ be a covering of $E_t$ such that $\sum_{B \in \dC} \diam(B)^d \leq 2\hdc(E_t)$. We may assume that the covering consists of balls (see \cite{mattila}, Chapter 5). Since $\dC$ covers $E_t$, and $\mu$ is a Frostman measure, we have
 \begin{equation*}
 	\mu(E_t) \leq \sum_{B \in \dC} \mu(B) \le  \sum_{B \in \dC} r(B)^d \lesssim  \hdc(E_t).
 	\end{equation*}
 Hence, 
 \begin{equation*}
 	\beta_{\mu,2}(x,r)  \leq \int_0^\infty \mu(E_t)\,t \,dt
 	 \lesssim \int_0^\infty \hdc(E_t)\,t \, dt \sim \beta_E^{d,2}(x,r),
 \end{equation*}
and this gives \eqref{eq:betamubetasigma} since $E \subset \Sigma$.
	
	Now recall that $\theta_\mu(x,r)= \mu(B(x,r))\,r^{-d}$ is the $d$-density of $\mu$ in the ball $B(x,r)$. Since $\mu(B(x,r)) \le r^d$, we have $\theta_\mu(x,r) \le 1$. Arguing as we did below \eqref{eq:main prop sum} and using \eqref{eq:betamubetasigma} we get
	\begin{multline*}
	\beta^2(\mu) := \iint_0^{\infty} \beta_{\mu,2}(x,r)^2 \theta_\mu(x,r) \, \frac{dr}{r}d\mu(x)\\
	 \le \iint_0^{1} \beta_{\mu,2}(x,r)^2 \, \frac{dr}{r}d\mu(x)+\iint_1^{\infty} \beta_{\mu,2}(x,r)^2 \, \frac{dr}{r}d\mu(x) \\
	 \lesssim \sum_{Q\in \dD} \beta_{\mu,2}(3B_Q)^2 \mu(Q) + \iint_1^{\infty} \theta_\mu(x,r)^2 \, \frac{dr}{r}d\mu(x)\\
	 \lesssim \sum_{Q\in \dD} \beta_\Sigma^{d, 2}(3B_Q)^2 \mu(Q) + \mu(E) \lesssim \sum_{Q\in \dD_\Sigma} \beta_\Sigma^{d,2}(3B_Q)^2 \ell(Q)^d + \mu(E).
	\end{multline*}
	Applying Theorem \ref{t:main} to estimate the last sum above yields
	\begin{equation*}
	\beta^2(\mu)\lesssim\sum_{Q\in \dD} \beta_\Sigma^{d,2}(3B_Q)^2 \ell(Q)^d + \mu(E) \lesssim C(\delta)\hd(\Sigma)+\hdc(E)\lesssim C(\delta)\hd(\Sigma),
	\end{equation*}
	so that $\beta^2(\mu)\le C_1(\delta)\hd(\Sigma)$ for some $C_1(\delta)>1$.
	
	Set
	\begin{equation*}
	C_2:= \left( \frac{\mu(E)}{C_1(\delta) \hd(\Sigma)} \right)^{\frac{1}{2}},
	\end{equation*}	
	and define the measure $\sigma := C_2 \mu$. Since $C_2\in (0,1)$, we have $\sigma(B(x,r))\le\mu(B(x,r))\le r^d$. Moreover,
	\begin{equation*}
	\beta^2(\sigma) = C_2^3 \beta^2(\mu)  \leq  \left( \frac{\mu(E)}{C_1(\delta) \hd(\Sigma)} \right)^{\frac{3}{2}} C_1(\delta) \hd(\Sigma) = C_2\, \mu(E)=\sigma(E).
	\end{equation*}
	Then, by Theorem \ref{thm:capacitybetas}, 
	\begin{align*}
	&\Gamma_{n,d} (E) \geq \sigma(E) = C_2 \mu(E) = \left( \frac{\mu(E)}{C_1(\delta) \hd(\Sigma)} \right)^{\frac{1}{2}} \mu(E) \\
	& \qquad \qquad \qquad \qquad \qquad = C_1(\delta)^{-\frac{1}{2}} \frac{\mu(E)^{\frac{3}{2}}}{\hd(\Sigma)^{\frac{1}{2}}} \sim C_1( \delta)^{-\frac{1}{2}} \frac{\hdc(E)^{\frac{3}{2}}}{\hd(\Sigma)^{\frac{1}{2}}}.
	\end{align*}
	This finishes the proof of \thmref{t:corollary-2b}.

\appendix

\section{Frostman's lemma for lower content regular sets}\label{appendix}
We prove the following version of the classical Frostman's lemma.
\begin{theorem}\label{theorem:Frostman measure}
	Let $E\subset [0,1]^n$ be a compact, lower content $(d,c_1)$-regular set. Then, there exists a measure $\mu$ satisfying the following properties:
	\begin{enumerate}
		\item \label{i:spt} $\supp\mu = E$,
		\item \label{i:mu-diam} $\mu(E) = \hdc(E)$,
		\item \label{i:poly-growth} $\mu$ has polynomial growth, that is, there exists a constant $C_1\ge 1$ such that for all $x\in E$ and $0<r<\diam(E)$ we have
		\begin{equation*}
			\mu(B(x,r))\le C_1 r^d,
		\end{equation*}
		\item \label{i:doubling2} $\mu$ is doubling, that is, there exists a constant $C_{db}\ge 1$ such that for all $x\in E$ and $0<r<\diam(E)$ we have
		\begin{equation}
			\mu(B(x,2r))\le C_{db}\, \mu(B(x,r)).
		\end{equation}
		\item \label{it:density2} the $d$-dimensional density of $\mu$ is almost monotone, that is, there exists a constant $A\ge 1$ such that if $P, Q\in\DD$, and $P\subset Q$, then
		\begin{equation*}
			\theta_\mu(P)\le A\, \theta_\mu(Q).
		\end{equation*} 
	\end{enumerate}
	In the above, $C_1$ depends only on $d,n$, while $C_{db}$ and $A$ also depend on the lower regularity constant $c_1$.
\end{theorem}

The usual proof of Frostman's lemma is ``bottom-to-top'': it starts from small scales and goes up (see \cite[Theorem 8.8]{mattila}). An alternative, simpler proof is due to Tolsa \cite[Theorem 1.23]{tolsa2014analytic}, who came up with a ``top-to-bottom'' construction. To prove \thmref{theorem:Frostman measure} we will modify Tolsa's construction so that the resulting Frostman measure is doubling. Roughly speaking, at each step we are going to modify the measure by redistributing the mass between cubes where the doubling condition fails.

\subsection{Construction of a sequence of measures}
In this subsection we construct a sequence of measures $\mu_k$. The measure $\mu$ from \thmref{theorem:Frostman measure} is going to be the weak limit of $\mu_k$.

Let $\DD$ be the David-Christ lattice on $E$, as in Lemma \ref{theorem:christ}.
For a cube $R \in \dD_k$, we denote by $\Child(R)$ the cubes $Q$ in $\dD_{k+1}$ with $Q \subset R$. By $\Nei(Q)$ we denote the cubes $P \in \dD_k$ so that $\dist(P,Q) \leq \ell(Q)$. Note that there exists a dimensional constant $c_n>1$ such that
\begin{equation}\label{eq:childneicard}
\#\Child(Q)+\#\Nei(Q)\le c_n.
\end{equation}
Recall that $Q^1$ denotes the parent cube of $Q$, and also that the balls $B(Q) = B(x_Q, c_0 \ell(Q))$ and $B_Q= B(x_Q, \ell(Q))$ satisfy $B(Q) \cap E \subset Q \subset B_Q$, where $\ell(Q)=\rho^k$ for $Q\in\DD_k$.

\begin{propo}\label{propo-frostman}
	Let $E \subset [0,1]^n$ be a compact, lower content $(d,c_1)$-regular set. There exists a constant $C_0=C_0(n,d,c_1)>1$ and sequence of Radon measures $\mu_k$, $k \ge 0$, absolutely continuous with respect to $\dL^n$, such that, for each $k$, we have
	\begin{gather}
		\mu_k(\R^n) = \hdc(E) \label{e:D},\\
		\supp\mu_k=\bigcup_{Q\in\DD_k}B(Q)\label{eq:supp},
	\end{gather}
	and moreover for all $Q\in\DD_k$
	\begin{align}
		\mu_k(B(Q))&\le \hdc(Q)\le 2\ell(Q)^d,\label{eq:poly}\\
		\mu_k(B(Q))&\le C_0\,\mu_k(B(P))\quad\text{for all $P\in\Nei(Q)$},\label{eq:neigh}\\
		C_0^{-1}\mu_{k-1}(B(Q^1))&\le \mu_k(B(Q))\lesssim \mu_{k-1}(B(Q^1)) \quad\text{if $k\ge 1$}.\label{eq:child}
	\end{align}
\end{propo}

Note that since for all $P\in\Nei(Q)$ we have $Q\in\Nei(P)$, if the property \eqref{eq:neigh} holds for all $Q\in\DD_k$, then
\begin{equation}\label{eq:doubl}
	C_0^{-1}\mu_k(B(P))\le \mu_k(B(Q))\le C_0\,\mu_k(B(P))\quad\text{for all $P\in\Nei(Q)$}.
\end{equation}

We construct $\mu_k$ inductively. Let $Q_0=E\in\DD_0$; that is, $Q_0$ is the top cube. We define
\begin{equation*}
	\mu_0\coloneqq \hdc(Q_0) \cdot \frac{\dL^n|_{B(Q_0)}}{\dL^n(B(Q_0))}.
\end{equation*} 
Clearly, $\mu_0$ satisfies all the properties from \propref{propo-frostman}. Now, assume that $\mu_{k-1}$ has already been defined and satisfies the properties \eqref{e:D}--\eqref{eq:child}.
\subsubsection{Auxiliary measure $\eta_{k}$}
For $R \in \dD_{k-1}$ and for each $Q \in \Child(R)$, set 
\begin{equation}\label{e:def-mul}
	\eta_{k}|_{B(Q)}\coloneqq \left( \frac{\hdc(Q)}{\sum_{P \in \Child(R)} \hdc(P)} \cdot \mu_{k-1}(B(R)) \right) \frac{\dL^n|_{B(Q)}}{\dL^n(B(Q))}.
\end{equation} 
Observe that
\begin{equation}\label{eq:etachildrencons}
\sum_{Q\in\Child(R)} \eta_{k}(B(Q)) = \mu_{k-1}(B(R)).
\end{equation}

\begin{lemma}\label{l:M-mul}
	We have $\eta_k(\R^n)=\mu_{k-1}(\R^n)=\hdc(E)$, so that condition \eqref{e:D} holds for $\eta_k$.
\end{lemma}
\begin{proof}
	We have
	\begin{multline*}
		\eta_k(\R^n) = \sum_{Q \in \dD_k} \eta_k(B(Q)) = \sum_{R \in \dD_{k-1}} \sum_{Q \in \Child(R)} \eta_k(B(Q))\\
		\overset{\eqref{eq:etachildrencons}}{=} \sum_{R \in \dD_{k-1}} \mu_{k-1}(B(R))= \mu_{k-1}(\R^n) = \hdc(E).
	\end{multline*}
\end{proof}

\begin{lemma}\label{l:A-mul}
	The polynomial growth condition \eqref{eq:poly} holds for $\eta_k$.
\end{lemma}
\begin{proof}
	Let $R\in\DD_{k-1}$ and $Q\in\Child(R)$. Since \eqref{eq:poly} holds for $\mu_{k-1}$, we have $\mu_{k-1}(B(R)) \leq \hdc(R) \leq \sum_{P \in \Child(R)} \hdc(P)$, by subadditivity of $\hdc$. Then,
	\begin{equation}\label{e:A-mul}
		\eta_k(B(Q)) = \hdc(Q) \cdot \frac{\mu_{k-1} (B(R))}{\sum_{P \in \Child(R)} \hdc(P) } \leq \hdc(Q) \leq 2 \ell(Q)^d.
	\end{equation}
	
\end{proof}

\begin{lemma}\label{lem:childneigheta}
	The doubling condition \eqref{eq:neigh} holds for $\eta_k$ whenever both $Q$ and $P$ in $\dD_k$ have a common parent $R \in \dD_{k-1}$. Moreover, for all $Q\in\DD_k$
	\begin{equation}\label{eq:child2}
	\eta_k(B(Q))\ge c_2\,\mu_{k-1}(B(Q^1))
	\end{equation}
	for $0<c_2<1$ depending on $c_1,n,d$. In particular \eqref{eq:child} holds for $\eta_k$ assuming $C_0$ is big enough.
\end{lemma}
\begin{proof}
	Let $R \in \dD_{k-1}$ and $Q,P \in \Child(R)$. Using lower content regularity of $E$ we have
	\begin{equation}\label{eq:lcrcube}
		\hdc(Q) \geq \hdc(B(Q)\cap E) \geq c_1 (c_0\ell(Q))^d = \fr{c_1 (c_0)^d}{2 \rho^d}\cdot 2 (\ell(R))^d\ge \fr{c_1 (c_0)^d}{2 \rho^d} \hdc(R).
	\end{equation}
	Recall that $\# \Child(R) \leq c_n$. We then compute
	\begin{multline*}	
		\eta_k (B(Q))  = \frac{\hdc(Q)}{\sum_{P' \in \Child(R)} \hdc(P')} \cdot \mu_{k-1}(B(R)) \nonumber \\
		 \geq \fr{c_1 (c_0)^d}{2 \rho^d} \cdot \frac{\hdc(R)}{c_n \hdc(R)} \cdot \mu_{k-1}(B(R)) \nonumber \
		= \fr{c_1 (c_0)^d}{2 c_n \rho^d} \mu_{k-1}(B(R)).
	\end{multline*}
	This shows \eqref{eq:child2}. To see \eqref{eq:neigh} note that
	\begin{equation*}
		\eta_k(B(P)) \leq \mu_{k-1}(B(R)) \le c_2 \eta_k(B(Q))\le C_0 \eta_k(B(Q)),
	\end{equation*}
	so \eqref{eq:neigh} holds true whenever $P$ is a sibling of $Q$. 
\end{proof}

	The lemmas above say that $\eta_k$ satisfies almost all the required properties, except that instead of the full doubling condition it only satisfies a ``dyadic doubling'' condition. We need to make some adjustments. 

\newcommand{\Poor}{\mathsf{Poor}}
\newcommand{\Rich}{\mathsf{Rich}}

\subsubsection{Definition of $\mu_k$}
We define families of cubes where the doubling condition \eqref{eq:doubl} fails. For each $Q\in\DD_k$ let
\begin{equation*}
\Rich(Q) = \{P\in\Nei(Q)\ :\ \eta_k(B(Q))< 4C_0^{-1}\, \eta_{k}(B(P)) \},
\end{equation*}
and similarly
\begin{equation*}
\Poor(Q) = \big\{P\in\Nei(Q)\ :\ \eta_k(B(Q))> \frac{C_0}{4}\, \eta_{k}(B(P)) \big\}.
\end{equation*}
Note that 
\begin{equation}\label{eq:poorcard}
\#\Poor(Q) +\#\Rich(Q)\le 2\, \#\Nei(Q)\le 2\,c_n.
\end{equation}
Set 
\begin{align*}
	\Poor_k &= \{Q\in\DD_k \ :\ \Rich(Q)\neq\varnothing\},\\
	\Rich_k &= \{Q\in\DD_k \ :\ \Poor(Q)\neq\varnothing\}.
\end{align*} 
\begin{lemma}\label{lem:poorrichdisjoint}
	We have $\Poor_k\cap\Rich_k=\varnothing$.
\end{lemma}
\begin{proof}
	Suppose that $Q\in\Poor_k\cap\Rich_k$. Then, there exist $P_p, P_r\in\Nei(Q)$ such that
	\begin{equation*}
		\frac{C_0}{4}\,\eta_{k}(B(P_p))< \eta_k(B(Q)) < 4C_0^{-1}\, \eta_{k}(B(P_r)).
	\end{equation*}
	In particular, $\eta_{k}(B(P_p)) < 16C_0^{-2}\,  \eta_{k}(B(P_r))$. Let $R_p, R_r$ be the parents of $P_p, P_r$, respectively. Note that $R_p\in\Nei(R_r)$. Then, by \eqref{eq:child2}
	\begin{equation*}
		\mu_{k-1}(B(R_p))\le c_2^{-1}\,\eta_{k}(B(P_p))<16c_2^{-1}C_0^{-2}\eta_{k}(B(P_r))\le C_0^{-1}\mu_{k-1}(B(R_r)),
	\end{equation*}
	assuming $C_0$ big enough. This contradicts the doubling property \eqref{eq:doubl} for $\mu_{k-1}$. Hence, $\Poor_k\cap\Rich_k=\varnothing$.
\end{proof}

We define the measure $\mu_k$ with $\supp\mu_k = \supp\eta_k = \bigcup_{Q \in \dD_k} B(Q)$ in the following way. For each $Q\in\DD_k$ we set
\begin{equation*}
	\mu_k|_{B(Q)} = \begin{cases}
		\big(1-  C_0^{-1}\cdot \#\Poor(Q) \big) \eta_k|_{B(Q)} & \text{if $Q\in\Rich_k$,}\\
		\big(1 + C_0^{-1}\sum_{P\in\Rich(Q)} \frac{\eta_k(B(P))}{\eta_k(B(Q))}\big) \eta_k|_{B(Q)} & \text{if $Q\in\Poor_k$,}\\
		\eta_k|_{B(Q)} & \text{if $Q\notin\Rich_k\cup\Poor_k$.}
	\end{cases}
\end{equation*}

\begin{lemma}
	Property \eqref{eq:child} holds for $\mu_k$.
\end{lemma}
\begin{proof}
	Observe that by \eqref{eq:poorcard}, as soon as $C_0$ is large enough depending on $n$, we have
	\begin{equation}\label{eq:muketak}
	\mu_k \ge \frac{1}{2} \eta_k.
	\end{equation}
	Thus, we immediately get from \eqref{eq:child2} that $\mu_k$ satisfies
	\begin{equation*}
	\mu_k(B(Q))\ge \fr{c_2}{2}\mu_{k-1}(B(Q^1))\ge C_0^{-1}\mu_{k-1}(B(Q^1)).
	\end{equation*}
	This shows one of the inequalities in \eqref{eq:child}. Now we prove that
	\begin{equation}
	\mu_k(B(Q))\lesssim \mu_{k-1}(B(Q^1)).
	\end{equation}
	For $Q\notin\Poor_k$ this is trivial, because then by the definition of $\mu_k$ and $\eta_k$ we have $\mu_k(B(Q))\le \eta_{k}(B(Q))\le \mu_{k-1}(B(Q^1))$. On the other hand, for $Q\in\Poor_k$
	\begin{multline*}
	\mu_k(B(Q)) = \eta_{k}(B(Q))+C_0^{-1}\sum_{P\in\Rich(Q)}\eta_{k}(B(P))\\
	\le \mu_{k-1}(B(Q^1)) + C_0^{-1}\sum_{P\in\Rich(Q)}\mu_{k-1}(B(P^1)).
	\end{multline*}
	For each $P\in\Rich(Q)$ we have $P^1\in\Nei(Q^1)$, and so by \eqref{eq:neigh} applied to $\mu_{k-1}$ we get
	\begin{multline*}
	\mu_{k-1}(B(Q^1)) + C_0^{-1}\sum_{P\in\Rich(Q)}\mu_{k-1}(B(P^1))\\
	\le \mu_{k-1}(B(Q^1)) + \#\Rich(Q)\mu_{k-1}(B(Q^1))\lesssim \mu_{k-1}(B(Q^1)).
	\end{multline*}
	This finishes the proof.
\end{proof}

\begin{lemma}
	We have $\mu_k(\R^n)=\eta_{k}(\R^n)=\hdc(E),$ so that condition \eqref{e:D} holds for $\mu_k$.
\end{lemma}
\begin{proof}
	We have
	\begin{multline*}
		\mu_k(\R^n) = \sum_{Q\in\DD_k}\mu_k(B(Q)) = \sum_{Q\in\Rich_k} \big(1-  C_0^{-1}\cdot \#\Poor(Q) \big)\eta_{k}(B(Q))\\
		+ \sum_{Q\in\Poor_k} \big(\eta_k(B(Q)) + C_0^{-1}\sum_{P\in\Rich(Q)} \eta_k(B(P))\big) + \sum_{Q\notin\Rich_k\cup\Poor_k}\eta_k(B(Q))\\
		= \sum_{Q\in\DD_k} \eta_k(B(Q)) - C_0^{-1}\sum_{Q\in\Rich_k}\#\Poor(Q)\eta_{k}(B(Q)) + C_0^{-1}\sum_{Q\in\Poor_k} \sum_{P\in\Rich(Q)} \eta_k(B(P))\\
		 = \sum_{Q\in\DD_k} \eta_k(B(Q)) = \eta_k(\R^n)=\hdc(E).
	\end{multline*}
\end{proof}

\begin{lemma}
	The polynomial growth condition \eqref{eq:poly} holds for $\mu_k$.
\end{lemma}
\begin{proof}
	Observe that for $Q\in\DD_k\setminus\Poor_k$ we have $\mu_k(B(Q))\le\eta_{k}(B(Q))$, so in this case \eqref{eq:poly} follows from \lemref{l:A-mul}. Assume that $Q\in\Poor_k$. Then, there exists $P\in\Nei(Q)$ such that
	\begin{equation*}
	\eta_k(B(Q))\le 2 C_0^{-1}\eta_k(B(P))\le 4 C_0^{-1}\ell(Q)^d,
	\end{equation*}
	where in the last inequality we used again \lemref{l:A-mul}. Then,
	\begin{multline*}
	\mu_k(B(Q)) = \eta_k(B(Q)) + C_0^{-1} \sum_{P\in\Rich(Q)} \eta_k(B(P))\\
	 \le 4C_0^{-1}\ell(Q)^d + \#\Rich(Q)\cdot 2C_0^{-1}\ell(Q)^d \overset{\eqref{eq:poorcard}}{\lesssim_n} C_0^{-1}\ell(Q)^d.
	\end{multline*}
	Recalling that $\hdc(Q)\gtrsim c_1\ell(Q)^d$ by \eqref{eq:lcrcube}, we get that $\mu_k(B(Q))\le \hdc(Q)\le 2\ell(Q)^d$ assuming $C_0$ large enough depending on $c_1,n,d$.
\end{proof}

\begin{lemma}
	The doubling condition \eqref{eq:neigh} holds for $\mu_k$.
\end{lemma}
\begin{proof}
	We want to show that for any $Q\in\DD_k$ and $P\in\Nei(Q)$ we have $\mu_k(B(Q))\le C_0\,\mu_k(B(P))$. First, suppose that $Q\notin\Rich_k\cup\Poor_k$, so that $\mu_k(B(Q))=\eta_k(B(Q))$. Let $P\in\Nei(Q)$. We have $P\notin\Poor(Q)$, and so
	\begin{equation*}
	\mu_k(B(Q))=\eta_k(B(Q))\le\fr{C_0}{4}\eta_k(B(P))\overset{\eqref{eq:muketak}}{\le} C_0\,\mu_k(B(P)).
	\end{equation*}
	
	Assume now that $Q\in\Rich_k,$ and let $P\in\Nei(Q)$. There are two cases: either $P\in\Poor(Q)$, or $P\notin\Poor(Q)$. In the first case we have $Q\in\Rich(P)$, and so
	\begin{equation*}
	\mu_k(B(Q))\le \eta_{k}(B(Q))\le C_0\big(\eta_k(B(P)) + C_0^{-1}\sum_{Q'\in\Rich(P)}\eta_k(B(Q'))\big) = C_0\,\mu_k(B(P)).
	\end{equation*}
	If $P\notin\Poor(Q)$ then by the definition of $\Poor(Q)$
	\begin{equation*}
	\mu_k(B(Q))\le \eta_{k}(B(Q))\le \fr{C_0}{4}\eta_k(B(P))\overset{\eqref{eq:muketak}}{\le}C_0\, \mu_k(B(P)).
	\end{equation*}
	This establishes \eqref{eq:neigh} for $Q\in\Rich_k$.
	
	Finally, suppose that $Q\in\Poor_k$ and $P\in\Nei(Q)$. Since for any $R\in \Nei(Q)$ we have $R^1\in\Nei(P^1)$, we get
	\begin{multline*}
	\mu_k(B(Q)) = \eta_k(B(Q)) + C_0^{-1}\sum_{R\in\Rich(Q)}\eta_k(B(R))\\ 
	\le \eta_k(B(Q)) + C_0^{-1}\sum_{R\in\Nei(Q)}\eta_k(B(R))\le \eta_k(B(Q)) + C_0^{-1}\sum_{R'\in\Nei(P^1)}\sum_{R\in\Child(R')}\eta_k(B(R))\\
	\overset{\eqref{eq:etachildrencons}}{=} \eta_k(B(Q)) + C_0^{-1}\sum_{R'\in\Nei(P^1)}\mu_{k-1}(B(R')).
	\end{multline*}
	Concerning the first term on the right hand side, note that since $Q\in\Poor_k$, by \lemref{lem:poorrichdisjoint} we have $Q\notin\Rich_k$ and consequently $P\notin\Poor(Q)$, so that $\eta_k(B(Q))\le \frac{C_0}{4}\,\eta_k(B(P))$. To deal with the second term we use the doubling property \eqref{eq:neigh} for $\mu_{k-1}$ and the fact that $\#\Nei(P^1)\le c_n$:
	\begin{multline*}
	\mu_k(B(Q))\le \eta_k(B(Q)) + C_0^{-1}\sum_{R'\in\Nei(P^1)}\mu_{k-1}(B(R'))\\ \le \fr{C_0}{4}\eta_k(B(P)) + \#\Nei(P^1)\mu_{k-1}(B(P^1))
	\le \fr{C_0}{4}\eta_k(B(P)) + c_n \mu_{k-1}(B(P^1)) \\
	\overset{\eqref{eq:child2}}{\le} \fr{C_0}{4}\eta_k(B(P) + \fr{c_n}{c_2}\eta_{k}(B(P))\le \fr{C_0}{2}\eta_{k}(B(P))\overset{\eqref{eq:muketak}}{\le} C_0\,\mu_k(B(P)),
	\end{multline*}
	assuming $C_0$ large enough. This gives the desired inequality \eqref{eq:neigh} for $Q\in\Poor_k$.	
\end{proof}

We have checked that $\mu_k$ satisfies properties \eqref{e:D}--\eqref{eq:child}, and so the proof of \propref{propo-frostman} is complete.

We prove two more properties of $\mu_k$ that will be useful later on. Recall that the sequence of measures $\nu_k$ constructed in the classical Frostman lemma has the following pleasant property: if $Q\in\DD_k$ and $j\ge k$, then $\nu_j(Q)=\nu_k(Q)$ -- the mass never escapes $Q$ after step $k$ of the construction. Our modified sequence of measures $\mu_k$ doesn't satisfy this property, but it's not too far off. 

Recall that $B(Q)\cap E\subset Q\subset B_Q=B(x_Q,\ell(Q))$.
\begin{lemma}\label{lem:massstays}
	If $j\ge k\ge 0$ and $Q\in\DD_k$, then
	\begin{equation}\label{eq:massstays}
	\mu_j(2B_Q)\ge \mu_k(B(Q))
	\end{equation}
	and
	\begin{equation}\label{eq:massstays2}
	\mu_j(B_Q)\le \mu_k(10B_Q.)
	\end{equation}
\end{lemma}
\begin{proof}
	In the definition of $\mu_{k+1}$ we transfer the mass of $\eta_{k+1}$ only between neighbors. It follows easily that for any family $\dF\subset\DD_k$
	\begin{equation}\label{eq:AF}
	\sum_{Q\in\dF}\mu_{k}(B(Q)) \overset{\eqref{eq:etachildrencons}}{=} \sum_{Q\in\dF}\sum_{P\in\Child(Q)}\eta_{k+1}(B(P))\le \sum_{R\in\mathcal{A}(\dF)}\mu_{k+1}(B(R)),
	\end{equation}
	where
	\begin{equation*}
	\mathcal{A}(\dF)\coloneqq \bigcup_{P\in\mathcal{F}}\bigcup_{R\in\Child(P)}\Nei(R)\subset\DD_{k+1}.
	\end{equation*}
	
	Let $Q\in\DD_k$. Set $\mathcal{A}_0(Q)\coloneqq\{Q\}$, and then inductively
	\begin{equation*}
	\mathcal{A}_i(Q) \coloneqq \mathcal{A}(\mathcal{A}_{i-1}(Q))= \bigcup_{P\in\mathcal{A}_{i-1}(Q)}\bigcup_{R\in\Child(P)}\Nei(R)\subset \DD_{k+i}.
	\end{equation*}
	Applying \eqref{eq:AF} $i$-times yields
	\begin{equation}\label{eq:AF2}
	\mu_k(B(Q))\le \sum_{P\in\dA_i(Q)}\mu_{k+i}(B(P)).
	\end{equation}
	
	Observe that if $R\in\DD$, then $\bigcup_{R'\in \Nei(R)}2B_{R'} \subset 5B_R$. It follows that for all $i\ge 0$
	\begin{equation*}
	\bigcup_{P\in\dA_i(Q)}2B_P \subset 2B_Q.
	\end{equation*}
	Indeed, this is clear for $i=0$, and then by induction
	\begin{multline*}
	\bigcup_{P\in\dA_i(Q)}2B_P = \bigcup_{P\in\mathcal{A}_{i-1}(Q)}\bigcup_{R\in\Child(P)}\bigcup_{R'\in \Nei(R)}2B_{R'} \subset \bigcup_{P\in\mathcal{A}_{i-1}(Q)}\bigcup_{R\in\Child(P)} 5B_R\\
	\subset \bigcup_{P\in\mathcal{A}_{i-1}(Q)} 2B_P \subset 2B_Q.
	\end{multline*}
	Together with \eqref{eq:AF2} this gives \eqref{eq:massstays}.
	
	Similarly, if for $\dF\subset\DD_j$ we define
	\begin{equation*}
	\mathcal{B}(\dF)\coloneqq \bigcup_{P\in\mathcal{F}}\Nei(P^1)\subset\DD_{j-1},
	\end{equation*}
	then by the definition of $\mu_j$
	\begin{multline}\label{eq:blablab}
	\sum_{Q\in\dF}\mu_j(B(Q))\le \sum_{P\in \bigcup_{Q\in\dF}\Nei(Q)} \eta_j(B(P)) \le \sum_{R\in\dB(\dF)} \sum_{P\in \Child(R)} \eta_j(B(P))\\
	 \overset{\eqref{eq:etachildrencons}}{=} \sum_{R\in\dB(\dF)} \mu_{j-1}(B(R)).
	\end{multline}
	Let $Q\in\DD_k$ and fix $j\ge k$. Set $\dB_0(Q) = \{ P\in\DD_j\ :\ B(P)\cap B_Q\neq\varnothing\}$, and for $0< i \le j-k$ we define inductively
	\begin{equation*}
	\dB_i(Q) \coloneqq \dB(\dB_{i-1}(Q)) = \bigcup_{P\in\dB_{i-1}(Q)}\Nei(P^1)\subset\DD_{j-i}.
	\end{equation*}
	By \eqref{eq:blablab}
	\begin{equation}\label{eq:blablabla}
	\mu_j(B_Q)\le \sum_{P\in\dB_0(Q)}\mu_j(B(P)) \le \sum_{P\in \dB_{j-k}(Q)}\mu_k(B(P)).
	\end{equation}
	Now, we claim that
	\begin{equation*}
	\bigcup_{P\in\dB_{j-k}(Q)}B(P) \subset 10B_Q.
	\end{equation*}
	To see this, note that each $P\in\dB_i(Q)$ has a neighbor $P'$ such that $\Child(P')\cap\dB_{i-1}(Q)\neq\varnothing.$ Hence, 
	\begin{equation*}
	\dist(P,\bigcup_{R\in \dB_{i-1}(Q)}R)\le \dist(P,P')+\diam(P')\le 3\ell(P)=15\rho^{j-i}.
	\end{equation*}
	Fix $P\in\dB_{j-k}(Q)$, and let $P=P_{j-k}, P_{j-k-1},\dots, P_{0}$ be such that for each $i$ we have $P_i\in\dB_i(Q)$ and $\dist(P_i,\bigcup_{R\in \dB_{i-1}(Q)}R)=\dist(P_i,P_{i-1})$. Then,
	\begin{multline*}
		\dist(P,\bigcup_{R\in \dB_{0}(Q)}R)\le \sum_{i=j-k}^{1} \dist(P_i,\bigcup_{R\in \dB_{i-1}(Q)}R) + \ell(P_{i-1})\\
		\le \sum_{i=j-k}^{1} 4\ell(P_{i})=\sum_{i=j-k}^{1} 20\rho^{j-i} \le 25\rho^{k} = 5\ell(Q).
	\end{multline*}
	It is easy to see that $\bigcup_{R\in \dB_{0}(Q)}R\subset 3B_Q$, and so it follows that
	\begin{equation*}
		B(P)\subset 10 B_Q.
	\end{equation*}
	Together with \eqref{eq:blablabla}, this gives \eqref{eq:massstays2}.
\end{proof}

Another nice property of $\mu_k$ we are going to need is related to the approximate monotonicity of densities.
\begin{lemma}\label{lem:density parent}
	Assume that $Q\in\dD_k$. Then, there exists $R\in\Nei(Q^1)\subset\DD_{k-1}$ with
	\begin{equation}\label{eq:density parent}
	\frac{\mu_k(B(Q))}{\dH^d_\infty(Q)}\le \frac{\mu_{k-1}(B(R))}{\dH^d_\infty(R)}.
	\end{equation}
\end{lemma}
\begin{proof}
	There are two cases to consider. First, if $\mu_k(B(Q))\le \eta_k(B(Q))$, then by the definition
	\begin{equation*}
	\mu_k(B(Q)) \leq \eta_k(B(Q)) = \frac{\mathcal{H}^{d}_{\infty}(Q)}{\sum_{P \in \Child(Q^1)} \mathcal{H}^{d}_{\infty}(P)} \mu_{k-1}(B(Q^1)) \leq \frac{\mathcal{H}^{d}_{\infty}(Q)}{\mathcal{H}^{d}_{\infty}(Q^1)} \cdot \mu_{k-1}(B(Q^1)).
	\end{equation*}
	Hence, taking $R=Q^1$ gives \eqref{eq:density parent}.
	
	Now suppose that $\mu_k(B(Q))> \eta_k(B(Q))$, i.e. $Q\in\Poor_k$. Let $P\in\Rich(Q)$ be the cube maximizing $\eta_{k}(B(P))$. Then,
	\begin{multline*}
	\mu_k(B(Q)) = \eta_{k}(B(Q)) + C_0^{-1}\sum_{S\in\Rich(Q)}\eta_{k}(B(S))\\
	\le 4C_0^{-1}\,\eta_{k}(B(P)) + \#\Rich(Q)C_0^{-1}\,\eta_{k}(B(P)) \le C_0^{-1}(4+c_n) \fr{\hdc(P)}{\hdc(P^1)}\mu_{k-1}(B(P^1)).
	\end{multline*}	
	Using lower content regularity of $E$ we have $\hdc(Q)\ge \hdc(E\cap B(Q))\ge c_1 (c_0\ell(Q))^d$, so that
	\begin{multline*}
	\fr{\mu_k(B(Q))}{\hdc(Q)}\le C_0^{-1}(4+c_n)\fr{\hdc(P)}{\hdc{(Q)}}\cdot \fr{\mu_{k-1}(B(P^1))}{\hdc(P^1)}\\
	\le C_0^{-1}(4+c_n)\fr{(2\ell(P))^d}{c_1 (c_0\ell(Q))^d}\cdot \fr{\mu_{k-1}(B(P^1))}{\hdc(P^1)}\\
	 = C_0^{-1}(4+c_n)2^d c_1^{-1}c_0^{-d}\fr{\mu_{k-1}(B(P^1))}{\hdc(P^1)}\le \fr{\mu_{k-1}(B(P^1))}{\hdc(P^1)},
	\end{multline*}
	assuming $C_0$ large enough depending on $c_1,n,d$. Since $P^1\in\Nei(Q^1)$, choosing $R=P^1$ gives \eqref{eq:density parent}.
\end{proof}

\subsection{The limit measure}
For each $k$ we have $\mu_k(\R^n) = \mu_k([0,1]^n)= \hdc(E)< \infty$, and so there exists a subsequence $\mu_{k_j}$ converging in the weak sense to a Radon measure $\mu$ which also satisfies $\mu(\R^n) = \hdc(E)$. This shows property \eqref{i:mu-diam} from \thmref{theorem:Frostman measure}. In the next few lemmas we prove that $\mu$ satisfies all the other required properties.

The following two basic facts about weak limits of measures will often be used without explicit mention: if $\nu_j \to \nu$ weakly, then for any open $U\subset \R^n$ and any compact $K \subset \R^n$
\begin{equation*}
	\nu (U) \leq \liminf_{j \to \infty} \nu_j(U) \,\, \mbox{ and } \,\, \limsup_{j \to \infty}\nu_j(K) \leq \nu(K). 
\end{equation*}
See \cite[Theorem 1.24]{mattila}.

\begin{lemma}
	We have $\supp\mu = E$, so that property (\ref{i:spt}) in Theorem \ref{theorem:Frostman measure} is satisfied.
\end{lemma}
\begin{proof}
	Let $x \in \R^n \setminus E$. Since $E$ is compact, for sufficiently small $r>0$ and sufficiently large $k \in \N$ we have $B(x,r) \cap \dN_{k}(E) = \varnothing$, where $\dN_{k}(E)$ is the open $5\rho^{k}$ neighborhood of $E$. Note that for any $Q \in \dD_{k}$ the ball $B(Q)$ is compactly contained in $\dN_{k}(E)$. It follows from \eqref{eq:supp} that $\mu_{k}(B(x,r))=0$ for sufficiently large $k$. Thus,
	\begin{equation*}
	\mu(B(x,r)) \leq \liminf_{j \to \infty} \mu_{k_j}(B(x,r))=0.
	\end{equation*}
	Thus, $x\notin\supp\mu$. This shows $\supp\mu \subset E$. 
	
	Now, let $x\in E$ and $r>0$. For $k\in \N$ large enough we have that if $Q\in\DD_k$ contains $x$, then $2B_Q\subset B(x,r)$. Fix such $k$ and $Q\in\DD_k$. By \eqref{eq:massstays} we get that for all $j\ge k$
	\begin{equation*}
	\mu_j(\overline{B(x,r)})\ge \mu_j(2B_Q)\ge \mu_k(B(Q))>0.
	\end{equation*}
	In consequence,
	\begin{equation*}
	\mu(\overline{B(x,r)}) \ge \limsup_{j \to \infty} \mu_{k_j}(\overline{B(x,r)})\ge \mu_k(B(Q))>0. 
	\end{equation*}
	So $x\in\supp\mu$.
\end{proof}

\begin{lemma}\label{l:frostmann}
	The measure $\mu$ satisfies $\mu(B(x,r)) \lesssim r^d$. In particular, property (\ref{i:poly-growth}) from Theorem \ref{theorem:Frostman measure} is satisfied.
\end{lemma}
\begin{proof}
	Let $x \in E$ and $r>0$. If $r> 5\rho^2$, then we just use the fact that $\mu(\R^d)=\hdc(E)\lesssim 1$. Suppose that $r\le 5\rho^2$, so that there exists $k \in \N$ such that $5\rho^{k+2} \leq r \leq 5\rho^{k+1}$. Let $Q \in \dD_{k}$ be such that $x\in Q$. It follows that $B(x,r) \subset B_Q$, and then
	\begin{multline*}
		\mu(B(x,r)) \le \liminf_{j \to \infty} \mu_{k_j}(B(x,r)) \leq \liminf_{j \to \infty}\mu_{k_j}(B_Q) \stackrel{\eqref{eq:massstays2}}{\leq} \mu_k(10B_Q)\\
		\le \sum_{\substack{P\in\DD_k\\ B(P)\cap 10 B_Q\neq\varnothing}} \mu_k(B(P)) \overset{\eqref{eq:poly}}{\le} \sum_{\substack{P\in\DD_k\\ B(P)\cap 10 B_Q\neq\varnothing}} 2\ell(P)^d \lesssim \ell(Q)^d\sim r^d.
	\end{multline*}
\end{proof} 

\begin{lemma}\label{l:doubling}
	The measure $\mu$ satisfies $\mu(B(x,2r)) \lesssim_{C_0} \mu(B(x,r))$ for all $x \in E$ and $r>0$. In particular, property \eqref{i:doubling2} in Theorem \ref{theorem:Frostman measure} is satisfied.
\end{lemma}
\begin{proof}
	Let $x\in E$ and $r>0$, and let $Q\in\bigcup_{k=0}^\infty\DD_k$ be the largest cube with $2B_Q\subset B(x,r)$. Let $k\in\N$ be such that $Q\in\DD_k$. Observe that
	\begin{equation}\label{eq:bla2}
		\mu(B(x,r))\ge\mu(\overline{2B_Q})\ge\limsup_{j \to \infty}\mu_{k_j}(\overline{2B_Q})\overset{\eqref{eq:massstays}}{\ge}\mu_k(B(Q)).
	\end{equation}
	If $r\ge 1$, then $\ell(Q)\sim 1$ and $k\lesssim 1$, so that using \eqref{eq:child} $k$-times yields
	\begin{equation*}
		\mu(B(x,r))\ge \mu_k(B(Q))	\gtrsim_{C_0} \mu_0(E)=\hdc(E) \ge \mu(B(x,2r)).
	\end{equation*}

	Now suppose that $0<r<1$, so that $\ell(Q)\sim r$. Let
	\begin{equation*}
		\mathcal{F} = \{P\in\DD_k\ :\ B_P\cap B(x,2r)\neq\varnothing\}.
	\end{equation*}
	Note that $\#\dF\lesssim 1$. Using the fact that for each $P\in\dF$ we have $10B_P\subset B(x,12r)$, and that $\{10B_P\}_{P\in\dF}$ have bounded overlap, we estimate
	\begin{equation}\label{eq:bla1}
		\mu(B(x,2r))\le \sum_{P\in\dF}\mu(B_P) \le \liminf_{j \to \infty} \sum_{P\in\dF} \mu_{k_j}(B_P) \overset{\eqref{eq:massstays2}}{\le} \sum_{P\in\dF} \mu_k(10 B_P)\lesssim \sum_{R\in\dG} \mu_k(B(R)),
	\end{equation}
	where
	\begin{equation*}
		\dG = \{R\in\DD_k\ :\ B(R)\cap B(x,12r)\neq\varnothing\}.
	\end{equation*}
	
	Note that for $R\in\dG$ we have $\ell(R)=\ell(Q)$ and $\dist(Q,R)\lesssim 1$. It follows that $Q^1\in\Nei(R^1)$, so that using \eqref{eq:child} and \eqref{eq:doubl} yields
	\begin{equation*}
		\mu_k(B(Q))\ge C_0^{-1}\mu_{k-1}(B(Q^1))\ge C_0^{-2}\mu_{k-1}(B(R^1))\gtrsim C_0^{-2}\mu_{k}(B(R)).
	\end{equation*}
	Together with \eqref{eq:bla2}, \eqref{eq:bla1}, and the fact that $\#\dG\lesssim 1$, this gives
	\begin{equation*}
	\mu(B(x,r))\ge \mu_k(B(Q)) \gtrsim_{C_0} \sum_{R\in\dG} \mu_k(B(R))\gtrsim \mu(B(x,2r)).
	\end{equation*}
\end{proof}

\newcommand{\calH}{\mathcal{H}}
\newcommand{\calK}{\mathcal{K}}

\begin{lemma}
	If $P,Q\in\DD$ and $P\subset Q$, then $\theta_\mu(P)\lesssim_{c_1,C_{db}}\theta_\mu(Q)$. In particular, $\mu$ satisfies property \eqref{it:density2} from Theorem \ref{theorem:Frostman measure}.
\end{lemma}
\begin{proof}
	Let $P, Q \in \dD$ with $P \subset Q$. Then $P \in \dD_{j}$ and $Q \in \dD_{k}$ for some $j\ge k$. Let 
	\begin{equation*}
	\dF = \{ R\in\DD_j\ :\ B(R)\cap 10B_P\neq\varnothing \}.
	\end{equation*}
	Note that $\dF\lesssim 1$. For every $R\in\dF$ we apply \lemref{lem:density parent} $(j-k)$-many times to obtain a sequence of cubes $R=R_0, R_1, R_2,\dots, R_{j-k}$ such that $R_i\in \Nei(R_{i-1}^1)$, and 
	\begin{equation*}
	\frac{\mu_{j-i}(B(R_i))}{\calH^d_\infty(R_i)}\le \frac{\mu_{j-i-1}(B(R_{i+1}))}{\calH^d_\infty(R_{i+1})}.
	\end{equation*}
	Using this inequality $(j-k)$-many times, together with lower content regularity of $E$, and the fact that $\ell(R)=\ell(P),\, \ell(R_{j-k})=\ell(Q),$ yields
	\begin{equation*}
	\frac{\mu_{j}(B(R))}{\ell(P)^d}\lesssim \frac{\mu_{j}(B(R))}{\calH^d_\infty(R)} \le \frac{\mu_{k}(B(R_{j-k}))}{\calH^d_\infty(R_{j-k})}\sim_{c_1} \frac{\mu_{k}(B(R_{j-k}))}{\ell(Q)^d},
	\end{equation*}
	By \lemref{lem:massstays} for any $i\ge j\ge k$ we get
	\begin{equation*}
	\fr{\mu_i(B_P)}{\ell(P)^d}\le \fr{\mu_j(10 B_P)}{\ell(P)^d} \le \sum_{R\in\dF}\frac{\mu_{j}(B(R))}{\ell(P)^d} \lesssim_{c_1} \sum_{R\in\dF} \frac{\mu_{k}(B(R_{j-k}))}{\ell(Q)^d} \le \sum_{R\in\dF} \frac{\mu_{i}(2B_{R_{j-k}})}{\ell(Q)^d}.
	\end{equation*}	
	Note that $\dist(R_{j-k},Q)\lesssim \ell(Q) = \ell(R_{j-k})$ for every $R\in\dF$, so that $2B_{R_{j-k}}\subset CB_Q$ for some absolute $C$. Hence, for all $i\ge j$
	\begin{equation*}
	\fr{\mu_i(B_P)}{\ell(P)^d} \lesssim_{c_1} \frac{\mu_{i}(\overline{CB_{Q}})}{\ell(Q)^d}.
	\end{equation*}
	Passing to the limit we get
	\begin{equation*}
	\fr{\mu(B_P)}{\ell(P)^d} \le \liminf_{j \to \infty} \fr{\mu_{k_j}(B_P)}{\ell(P)^d} \lesssim_{c_1} \limsup_{j \to \infty} \frac{\mu_{k_j}(\overline{CB_{Q}})}{\ell(Q)^d} \le \frac{\mu(\overline{CB_{Q}})}{\ell(Q)^d}.
	\end{equation*}
	Using the doubling property of $\mu$ \eqref{i:doubling2} finishes the proof.
\end{proof}

\end{document}